\documentclass{amsart}

\usepackage[non-sorted-cites]{amsrefs}
\usepackage{euscript}
\usepackage{amsmath}
\usepackage{amscd}
\usepackage{amssymb}
\usepackage{enumerate}
\usepackage{stmaryrd}
 \usepackage{amsfonts}
\usepackage{graphicx}
\usepackage{subfigure}
\usepackage[all]{xy}
\usepackage[margin=1.5in]{geometry}
\usepackage{setspace}
\usepackage{eufrak}
\usepackage{array}
\usepackage{tikz}
\usetikzlibrary{calc,intersections,fit}
\usepackage{textcomp}
\newcommand*{\ExtractCoordinate}[1]{\path (#1); \pgfgetlastxy{\XCoord}{\YCoord};}
\newcommand{\mw}{14.17}
\newlength{\lw}
\newcommand{\Stripdraw}[2]
{
\draw[line width=\lw] (#1, #2+1/8*\mw pt)--(#1 - \mw pt, #2 +1/8*\mw pt);
\draw[line width=\lw] (#1, #2-1/8*\mw pt)--(#1 - \mw pt, #2  -1/8*\mw pt);
}

\newcommand{\StripdrawBig}[2]
{
\draw[line width=2*\lw] (#1, #2+1/2*\mw pt)--(#1 - 4*\mw pt, #2 +1/2*\mw pt);
\draw[line width=2*\lw] (#1, #2-1/2*\mw pt)--(#1 - 4*\mw pt, #2  -1/2*\mw pt);
}

\newcommand{\Conout}[3] 
{
\begin{scope}  
\Stripdraw{#1}{#2};
\fill (#1 -1/2*\mw pt, #2) circle (2*\lw);
\end{scope}
\node at (#1-1.15*\mw pt, #2)  {$\scriptscriptstyle{#3}$};
}

\newcommand{\ConoutBig}[3] 
{
\begin{scope}  
\StripdrawBig{#1}{#2};
\fill (#1 -2*\mw pt, #2) circle (8*\lw);
\end{scope}
\node at (#1-5*\mw pt, #2)  {$#3$};
}
\newcommand{\Conin}[3] 
{
\begin{scope} 
\Stripdraw{#1}{#2};
\fill (#1 -1/2*\mw pt, #2) circle (2*\lw);
\end{scope}
\node at (#1+.15*\mw pt, #2)  {$\scriptscriptstyle{#3}$};
}
\newcommand{\Conintwo}[3] 
{
\begin{scope}
\Stripdraw{#1}{#2};
\fill (#1 -5/8*\mw pt, #2) circle (2*\lw);
\fill (#1 -3/8*\mw pt, #2) circle (2*\lw);
\end{scope}
\node at (#1+.15*\mw pt, #2)  {$\scriptscriptstyle{#3}$};
}
\newcommand{\ConinBig}[3] 
{
\begin{scope} 
\StripdrawBig{#1}{#2};
\fill (#1 -2*\mw pt, #2) circle (8*\lw);
\end{scope}
\node at (#1+1*\mw pt, #2)  {$#3$};
}
\newcommand{\Con}[4] 
{
\begin{scope}
\Stripdraw{#1}{#2};
\fill (#1 -1/2*\mw pt, #2) circle (2*\lw);
\end{scope}
\node at (#1-9/8*\mw pt, #2 pt)  {$\scriptscriptstyle{#3}$};
\node at (#1+1/8*\mw pt, #2 pt)  {$\scriptscriptstyle{#4}$};
}

\newcommand{\Contwo}[4] 
{
\begin{scope}
\Stripdraw{#1}{#2};
\fill (#1 -5/8*\mw pt, #2) circle (2*\lw);
\fill (#1 -3/8*\mw pt, #2) circle (2*\lw);
\end{scope}
\node at (#1-9/8*\mw pt, #2 pt)  {$\scriptscriptstyle{#3}$};
\node at (#1+1/8*\mw pt, #2 pt)  {$\scriptscriptstyle{#4}$};
}
\newcommand{\Conthree}[4] 
{
\begin{scope}
\Stripdraw{#1}{#2};
\fill (#1 -1/2*\mw pt, #2) circle (2*\lw);
\fill (#1 -5/8*\mw pt, #2) circle (2*\lw);
\fill (#1 -3/8*\mw pt, #2) circle (2*\lw);
\end{scope}
\node at (#1-9/8*\mw pt, #2 pt)  {$\scriptscriptstyle{#3}$};
\node at (#1+1/8*\mw pt, #2 pt)  {$\scriptscriptstyle{#4}$};
}
\newcommand{\CtoE}[4] 
{
\begin{scope}  [blue]
\Stripdraw{#1}{#2};
\draw[line width=2*\lw] (#1-.5*\mw pt, #2+3/16*\mw pt)--(#1-.5*\mw pt, #2+1/16*\mw pt);
\end{scope}
\node at (#1-9/8*\mw pt, #2 pt)  {$\scriptscriptstyle{#3}$};
\node at (#1+1/8*\mw pt, #2 pt)  {$\scriptscriptstyle{#4}$};
}

\newcommand{\CtoEBig}[4] 
{
\begin{scope}  [blue]
\StripdrawBig{#1}{#2};
\draw[line width=4*\lw] (#1-2*\mw pt, #2+6/8*\mw pt)--(#1-2*\mw pt, #2+3/8*\mw pt);
\end{scope}
\node at (#1-5*\mw pt, #2 pt)  {$#3$};
\node at (#1+1*\mw pt, #2 pt)  {$#4$};
}

\newcommand{\CtoEBigL}[6] 
{
\begin{scope}  [blue]
\StripdrawBig{#1}{#2};
\draw[line width=4*\lw] (#1-2*\mw pt, #2+6/8*\mw pt)--(#1-2*\mw pt, #2+3/8*\mw pt);
\end{scope}
\node at (#1-5*\mw pt, #2 pt)  {$#3$};
\node at (#1+1*\mw pt, #2 pt)  {$#4$};
\node at (#1-2.5*\mw pt, #2+1*\mw pt) {$#5$};
\node at (#1-1.5*\mw pt, #2+1*\mw pt) {$#6$};
}
\newcommand{\EtoCBigL}[6] 
{
\begin{scope}  [red]
\StripdrawBig{#1}{#2};
\draw[line width=4*\lw] (#1-2*\mw pt, #2+6/8*\mw pt)--(#1-2*\mw pt, #2+3/8*\mw pt);
\end{scope}
\node at (#1-5*\mw pt, #2 pt)  {$#3$};
\node at (#1+1*\mw pt, #2 pt)  {$#4$};
\node at (#1-2.5*\mw pt, #2+1*\mw pt) {$#5$};
\node at (#1-1.5*\mw pt, #2+1*\mw pt) {$#6$};
}

\newcommand{\CtoEtwosplit}[4] 
{
\CtoE{#1}{#2}{#3}{#4};
\fill (#1 -1/4*\mw pt, #2) circle (2*\lw);
\fill (#1 -3/4*\mw pt, #2) circle (2*\lw);
}
\newcommand{\CtoEtworight}[4] 
{
\CtoE{#1}{#2}{#3}{#4};
\fill (#1 -1/8*\mw pt, #2) circle (2*\lw);
\fill (#1 -3/8*\mw pt, #2) circle (2*\lw);
}
\newcommand{\CtoEtwoleft}[4] 
{
\CtoE{#1}{#2}{#3}{#4};
\fill (#1 -7/8*\mw pt, #2) circle (2*\lw);
\fill (#1 -5/8*\mw pt, #2) circle (2*\lw);
}
\newcommand{\CtoEoneright}[4] 
{
\CtoE{#1}{#2}{#3}{#4};
\fill (#1 -1/4*\mw pt, #2) circle (2*\lw);
}
\newcommand{\CtoEoneleft}[4] 
{
\CtoE{#1}{#2}{#3}{#4};
\fill (#1 -3/4*\mw pt, #2) circle (2*\lw);
}
\newcommand{\EtoC}[4] 
{
\begin{scope}  [red]
\Stripdraw{#1}{#2};
\draw[line width=2*\lw] (#1-.5*\mw pt, #2+3/16*\mw pt)--(#1-.5*\mw pt, #2+1/16*\mw pt);
\end{scope}
\node at (#1-9/8*\mw pt, #2 pt)  {$\scriptscriptstyle{#3}$};
\node at (#1+1/8*\mw pt, #2 pt)  {$\scriptscriptstyle{#4}$};
}
\newcommand{\EtoConeright}[4] 
{
\EtoC{#1}{#2}{#3}{#4};
\fill (#1 -1/4*\mw pt, #2) circle (2*\lw);
}
\newcommand{\EtoCtworight}[4] 
{
\EtoC{#1}{#2}{#3}{#4};
\fill (#1 -1/8*\mw pt, #2) circle (2*\lw);
\fill (#1 -3/8*\mw pt, #2) circle (2*\lw);
}

\newcommand{\CtoEBent}[4] 
{
\begin{scope}  [blue]
\draw[line width=\lw] (#1 pt, #2+2*\mw pt) .. controls (#1+4*\mw pt, #2+ 2*\mw pt) and  (#1+4*\mw pt, #2-2*\mw pt) .. (#1 pt, #2 -2*\mw pt);
\draw[line width=\lw] (#1 pt, #2+3*\mw pt) .. controls (#1+5.5*\mw pt, #2+ 3*\mw pt) and (#1+5.5*\mw pt, #2- 3*\mw pt) .. (#1 pt, #2-3* \mw pt);
\draw[line width=4*\lw] (#1 + 4*\mw pt, #2 pt)--(#1+4.375*\mw pt, #2 pt);
\end{scope}
\node at (#1 pt -1*\mw pt , #2 pt -2.5 * \mw pt)  {$#3$};
\node at (#1 pt -1*\mw pt, #2 pt + 2.5 * \mw pt)  {$#4$};
}

\newcommand{\EtoCBent}[2] 
{
\begin{scope}  [red]
\draw[line width=\lw] (#1 pt, #2+2*\mw pt) .. controls (#1-4*\mw pt, #2+ 2*\mw pt) and  (#1-4*\mw pt, #2-2*\mw pt) .. (#1 pt, #2 -2*\mw pt);
\draw[line width=\lw] (#1 pt, #2+3*\mw pt) .. controls (#1-5.5*\mw pt, #2+ 3*\mw pt) and (#1-5.5*\mw pt, #2- 3*\mw pt) .. (#1 pt, #2-3* \mw pt);
\draw[line width=4*\lw] (#1 - 4*\mw pt, #2 pt)--(#1-4.375*\mw pt, #2 pt);
\end{scope}
}

\newcommand{\bC}{{\mathbb C}}
\newcommand{\bP}{{\mathbb P}}
\newcommand{\bR}{{\mathbb R}}
\newcommand{\bZ}{{\mathbb Z}}

\newcommand{\cE}{\mathcal E}

\newcommand{\cM}{\mathcal M}

\newcommand{\cU}{\mathcal U}

\newcommand{\cR}{\mathcal R}

\newcommand{\Adams}{\overline{\EuScript A}}
\newcommand{\Adamsop}{\EuScript A}

\newcommand{\scrC}{\EuScript C}
\newcommand{\scrD}{\EuScript D}

\newcommand{\scrF}{\EuScript F}
\newcommand{\scrG}{\EuScript G}
\newcommand{\scrL}{\EuScript L}

\newcommand{\scrO}{\EuScript O}
\newcommand{\scrP}{\EuScript P}
\newcommand{\scrU}{\EuScript U}

\newcommand{\Q}{Q}

\newcommand{\Mbar}{\overline{\cM}}
\newcommand{\Ubar}{\overline{\cU}}
\newcommand{\Rbar}{\overline{\cR}}

\newcommand{\triv}{\tau}

\newcommand{\vI}[1][\!]{\vec{I}^{\, #1}}
\newcommand{\vJ}[1][\!]{\vec{J}^{\, #1}}
\newcommand{\vK}[1][\!]{\vec{K}^{\, #1}}

\newcommand{\vr}[1][\!]{\vec{r}^{\, #1}}
\newcommand{\ro}{{\mathrm o}}

\newcommand{\scrJ}{{\EuScript J}}
\newcommand{\scrH}{{\EuScript H}}

\newcommand{\bfk}{\mathbf{k}}

\newcommand{\coker}{\operatorname{coker}}

\newcommand{\id}{\operatorname{id}}

\newcommand{\Ann}{{\mathcal C}}

\newcommand{\Cont}{{\mathcal K}}

\newcommand{\Hom}{\operatorname{Hom}}

\newcommand{\Strip}{B}

\newcommand{\Pin}{\operatorname{Pin}}
\newcommand{\Spin}{\operatorname{Spin}}

\newcommand{\tr}{\mathrm{tr}}

\renewcommand{\dbar}{\overline{\partial}}
\renewcommand{\det}{\operatorname{det}}

\newcommand{\val}{\operatorname{val}}

\newcommand{\Diff}{\operatorname{Diff}}
\newcommand{\Id}{\operatorname{Id}}
\newcommand{\inp}{\mathrm{in}}
\newcommand{\out}{\mathrm{ou}}
\newcommand{\opnu}{\mathring{\nu}}
\renewcommand{\min}{\operatorname{min}}
\renewcommand{\max}{\operatorname{max}}
\def\co{\colon\thinspace}

\numberwithin{equation}{section}

\newtheorem{thm}{Theorem}[section]
\newtheorem{cor}[thm]{Corollary}
\newtheorem{lem}[thm]{Lemma}
\newtheorem{prop}[thm]{Proposition}
\newtheorem{defin}[thm]{Definition}
\newtheorem{def-lem}[thm]{Definition-Lemma}

\newtheorem{rem}[thm]{Remark}

\theoremstyle{remark}
\newtheorem{example}[thm]{Example}

\newcommand{\superscript}[1]{\ensuremath{^{\textrm{#1}}} }

\renewcommand{\th}[0]{\superscript{th}}

\newcommand{\comment}[1]{}
\title[The family Floer functor is faithful]{The family Floer functor is faithful}

\author[M.~Abouzaid]{Mohammed Abouzaid} \date{\today} \thanks{The author was supported by NSF grant DMS-1308179 and the Simons Collaboration on Homological Mirror Symmetry.}
\begin{document}

\begin{abstract}
Family Floer theory is used to construct a functor from the Fukaya category of a symplectic manifold admitting a Lagrangian torus fibration to a (twisted) category of perfect complexes on the mirror rigid analytic space. This functor is shown to be faithful by a degeneration argument involving moduli spaces of annuli.
\end{abstract}
\maketitle

\tableofcontents

\section{Introduction}
\label{sec:introduction}

Applications of Fukaya categories to symplectic topology require an algebraic model for these categories: this involves finding a collection of Lagrangians which \emph{generate} the category in the sense that the Fukaya category fully faithfully embeds in the category of modules over the corresponding $A_\infty$ algebra. For closed symplectic manifolds, the known strategies for understanding such categories of modules rely on realising them, in an instance of homological mirror symmetry, as modules over the endomorphism algebra of (complexes of) coherent sheaves on an algebraic variety, or a non-commutative deformation thereof. Such descriptions are possible in a limited class of examples, which include Calabi-Yau hypersurfaces in projective space \cite{seidel-quartic,sheridan} and toric varieties \cite{FOOO-sign,AFOOO}. It reasonable to expect that these methods will lead to descriptions of Fukaya categories of complete intersections in toric varieties \cite{AAK}.

The goal of the family Floer program is to both give a more compelling proof of these equivalences, and to extend the class of examples for which they can be proved. Keeping with tradition, we shall call the symplectic side the $A$-side, and the algebro-geometric side the $B$-side. The current strategies rely on matching computations on the two sides, without having a good reason for the agreement. Moreover, these computations work for a very special class of symplectic structures; in the typical case of the $K3$ surface, homological mirror symmetry is only understood for the restriction of the Fubini-Study form to the quartic hypersurface, whereas the rank of the second cohomology group is $22$.

There are essentially only two previous results on family Floer cohomology.  In \cite[Section 6]{Fukaya-family} Fukaya outlined a strategy for assigning to Lagrangians (complexes of) coherent sheaves on the mirror, under some convergence assumptions which should yield a complex analytic mirror. Passing to the rigid analytic setting in which he derived convergence by a clever use of Gromov compactness for tame almost complex structures, Fukaya  gave a very general result  in \cite{Fukaya-cyclic} constructing the local charts of the $B$-side, which were shown by Tu to admit compatible identifications over the overlaps \cite{tu}. 

In the author's ICM address \cite{A-ICM}, the strategies behind these two results were combined, and a (rigid analytic) coherent sheaf was assigned to Lagrangians on the $A$-side, assuming the existence of a Lagrangian torus fibration.  This paper extends this result by (1) constructing a map of morphism spaces from the $A$-side to the $B$-side,  (2) constructing a map of morphism spaces from the $B$-side to the $A$-side, (3) showing that the composition of these two maps is the identity on the $A$-side, leading to the main result, and (4) constructing an $A_\infty$ functor.  The bulk of Section \ref{sec:lagr-torus-fibr} contains a construction of the mirror following \cite{A-ICM}, and corrects a minor oversight in the local-to-global construction of the earlier paper (see Remark \ref{rem:careful_cover_Q}).  The formal results are then stated in Theorem \ref{thm:main}, and a summary of the proof appears thereafter in Section \ref{sec:what-remains-be}.
\begin{rem}
In order to focus on the new ideas, we restrict the setting that we consider by assuming that (1) the ambient symplectic manifold admits a Lagrangian torus fibration all of whose fibres are smooth and bound no holomorphic discs, and (2) one can choose an almost complex structure for each Lagrangian so that it bounds no holomorphic discs. The requirement that the Lagrangians bound no holomorphic disc is really only technical, and meant to avoid discussing foundations of multivalued perturbations in Lagrangian Floer theory \cite{FOOO} (and multiplying the length of the paper by a potentially large factor). The reader may consult the introduction to \cite{A-ICM} for a discussion of the more serious difficulties one would encounter in the presence of singular fibres.
\end{rem}

\begin{figure}
  \centering
  \begin{tikzpicture}
    \CtoEBent{-4*\mw}{0}{x}{y};
   \EtoCBent{-6*\mw}{0};
\node at (-2*\mw pt, 0)  {$F_{q}$};
\node at (-7.5*\mw pt, 0)  {$F_{q}$};
\node[left] at (-1*\mw pt, 3*\mw pt)  {$L$};
\node[left] at (-1*\mw pt, -3*\mw pt)  {$L$};
\node[left] at (-8*\mw pt, 3*\mw pt)  {$L$};
\node[left] at (-8*\mw pt, -3*\mw pt)  {$L$};

\draw[line width=2*\lw][blue] (5*\mw pt,-3*\mw pt) arc (-90:90:3*\mw pt) ;
\node[left] at (4*\mw pt,3*\mw pt)  {$L$};
\draw[line width=2*\lw][blue] (5*\mw pt,-2*\mw pt) arc (-90:90:2*\mw pt) ;
\draw[line width=4*\lw][blue] (8*\mw pt -1/8*\mw pt, 0*\mw pt)--(8*\mw pt+1/4*\mw pt, 0*\mw pt);

\draw[line width=2*\lw][red] (5*\mw pt,3*\mw pt) arc (90:270:3*\mw pt) ;
\node at (5*\mw pt, 0)  {$F_{q}$};
\draw[line width=2*\lw][red] (5*\mw pt,2*\mw pt) arc (90:270:2*\mw pt) ;
\draw[line width=4*\lw][red] (2*\mw pt +1/8*\mw pt, 0*\mw pt)--( 2*\mw pt-1/4*\mw pt, 0*\mw pt);

\node[left] at (10*\mw pt, 0)  { };
\draw[line width=2*\lw] (10*\mw pt,0) arc (180:-180:3*\mw pt) ;
\node[left] at (12*\mw pt, 3*\mw pt)  {$L$};
\draw[line width=4*\lw][blue] (16*\mw pt -1/8*\mw pt,0)--(16*\mw pt+1/4*\mw pt,0);
\draw[line width=4*\lw][red] (10*\mw pt +1/8*\mw pt,0)--(10*\mw pt-1/4*\mw pt,0);
\fill (13*\mw pt, 0) circle (8*\lw);
\node[above] at (13*\mw pt, 0)  {$F_{q}$};
\end{tikzpicture}
  \caption{ }
  \label{fig:heuristic}
\end{figure}

Since the construction of the homotopy from the composition of the two maps we construct to the identity uses a moduli space of annuli, faithfulness can be seen as the analogue of the generation criterion \cite{A-generate}. Heuristically, the strategy for the proof is the following: let $X$ be a symplectic manifold equipped with a Lagrangian torus fibration over a base $\Q$ (we denote the fibre over $q \in \Q$ by  $F_q$), and $L$ a Lagrangian in $X$. Consider  moduli spaces of holomorphic discs  with $3$ marked points, on which we impose Lagrangian boundary conditions given by $L$ and  a fibre. We shall consider two flavours for this moduli space (see the leftmost diagram in Figure \ref{fig:heuristic}): in the first case,  one marked point is distinguished as an input mapping to $L$, and the remaining two are outputs mapping to intersection points $x$ and $y$ of $L$ with a fibre, while in the second case, the intersections of $L$ with a fibre correspond to inputs, while the marked point on $L$ is an output.

In the classical versions of Floer theory, one would consider the subcategory of the Fukaya category of $X$ whose objects are fibres, and the Yoneda module over this subcategory associated to $L$. By allowing the addition of an arbitrary number of marked points, the first of these moduli spaces defines the map from the Floer cohomology $L$ to the endomorphism algebra of this Yoneda module, and the second moduli space defines a map which one could hope to show is a right inverse by gluing the two triangles to an annulus, and degenerating this annulus to two discs meeting at an interior point; one of the discs has Lagrangian boundary conditions on an \emph{arbitrary fibre} and the other has Lagrangian boundary conditions on $L$ and carries the two boundary marked points. Since the moduli space of discs with boundary on an arbitrary fibre gives us a copy of the ambient space $X$, the first type of disc imposes no constraint, so we are simply considering the moduli space of discs with boundary on $L$ (and two marked points). This moduli space represents the identity on Floer cohomology.

Trying to implement this strategy in this setting runs into a convergence problem: since the fibres are disjoint, they are Floer theoretically orthogonal, so the Yoneda module defined by $L$ is a direct product of the corresponding modules for all fibres. The map back to Floer theory is not well-defined because it is the sum of infinitely many terms. The correct framework for this argument is in fact family Floer cohomology, and the main difficulty that arises is due to the need to make compatible families of perturbation in defining the Floer cohomology of $L$ with every fibre; in the classical case, one can choose such perturbations independently for all pairs of objects.

\subsection*{Acknowledgments}
Discussions with Denis Auroux, Kenji Fukaya, Paul Seidel, and Ivan Smith, exploring potential applications of  family Floer cohomology, were helpful in justifying the development of these techniques. I would also like to thank Robert Young for a discussion of triangulations of manifolds. I am tremendously grateful to an anonymous referee for carefully reading the early versions of this paper and writing detailed reports which helped catch several gaps and provided guidance in producing the current version.

The author was supported by NSF grant DMS-1308179 and by the Simons Foundation through the Simons Collaboration on Homological Mirror Symmetry.

\section{Lagrangian torus fibrations and their rigid analytic dual}
\label{sec:lagr-torus-fibr}

\subsection{Flux and integral affine structure}
\label{sec:geometry-base}

Let $(X, \omega)$ be a closed symplectic manifold of dimension $2n$, and $\pi \co  X \to \Q$ a  Lagrangian torus fibration, whose fibre at a point $q \in \Q$ we denote $F_q$. We briefly recall the construction of an integral affine structure on $\Q$ induced by the symplectic structure on $X$ (see \cite{GHJ} for an extended discussion geared toward mirror symmetry).

Since $\pi$ is a fibre bundle, the cohomology groups of the fibres form a local system over the base, e.g. for any continuous path $\{ q_t \}_{t \in [0,1]}$ in $\Q$, we have a canonical isomorphism
\begin{equation} \label{eq:classical_transport}
   H^1(F_{q_0}, \bZ) \to  H^1(F_{q_1}, \bZ)
\end{equation}
which depends only on the homotopy class of the path. The fact that the fibres are Lagrangian yields in addition an element of $H^1(F_{q_1}, \bR)$  associated to every homotopy class of paths; this is the \emph{flux} which is defined as the integral of the symplectic structure on cylinders in $X$ lying over paths in $\Q$ (see \cite{Mcduff-Salamon}). As a result, we obtain an integral affine map
\begin{equation} \label{eq:symplectic_transport}
     H^1(F_{q_0}, \bR) \to  H^1(F_{q_1}, \bR)
\end{equation}
which takes the origin to the flux and whose derivative agrees with the classical isomorphism of Equation \eqref{eq:classical_transport} by passing to real coefficients.

Since points which are sufficiently close are connected by a canonical homotopy class of paths, we obtain a map from a neighbourhood of every point $q \in \Q$ to a neighbourhood of the origin in $  H^1(F_q, \bR) $ which assigns to every point the flux of the corresponding short path.   The Arnol'd-Liouville theorem implies that this is a diffeomorphism near the origin, and in particular that we have a natural isomorphism
\begin{equation} \label{eq:iso_tangent_space_flux}
  T_q \Q \cong H^1(F_q, \bR) .
\end{equation}
In particular, whenever $p$ lies in a sufficiently small neighbourhood of $q$, we write $p-q$ for the corresponding element of $T_q \Q$.

We say that a subset $P \subset \Q$ is \emph{an integral affine polygon,} if its image under the flux map is a polygonal neighbourhood of the origin defined by inequalities of the form
\begin{equation}
  \langle u , v_i \rangle \geq \lambda_i
\end{equation}
with $v_i$ an integral vector, and $\lambda_i$ real. The key reason that this notion is well-behaved is that the differential of the isomorphism in Equation \eqref{eq:symplectic_transport} preserves the integral structure on first cohomology with real coefficients defined by the integral cohomology lattice.

\subsection{Flux and the energy of holomorphic strips}
\label{sec:areas-strips-flux}
Let $\scrJ$ denote the space of  $\omega$-tame almost complex structures on $X$.  Given a point $q \in \Q$ and a closed Lagrangian $L \subset X$, pick a Hamiltonian diffeomorphism $\phi$ so that $ \phi L$ is transverse to $F_{q}$.

In order to define the moduli spaces of holomorphic strips with boundary on $L$ and $F_q$, pick a family $J = \{ J_{t} \in \scrJ \}_{t \in [0,1]}$. We obtain a holomorphic curve equation on the strip $B = \bR \times [0,1]$ with Lagrangian boundary conditions:
\begin{align} \label{eq:differential_strip_J-holo}
  u \co B \to X &  \qquad \partial_{s} u = J_{t} \partial_{t}u \\ \label{eq:differential_strip_boundary}
u(s,0) \in F_{q} & \qquad u(s,1) \in \phi L.
\end{align}

Given a pair of points $(x, y) \in \phi L \cap F_{q}$, denote by $
\cM_q(x,y)$ the moduli space of Floer trajectories connecting $x$ to $y$: this is the  quotient by translation in the $\bR$ factor of $B$ of the space of solutions to Equations \eqref{eq:differential_strip_J-holo} and \eqref{eq:differential_strip_boundary} which in addition satisfy the asymptotic conditions
\begin{equation}
 \lim_{s \to -\infty} u(s,t) = x  \qquad \textrm{and} \quad \lim_{s \to +\infty} u(s,t) = y.
\end{equation}
We denote by $\Mbar_{q}(x,y) $ the Gromov-Floer compactification of this moduli space.

Floer theory uses the moduli spaces $\Mbar_{q}(x,y)$ to define the Floer complex between $L$ and $F_q$; the goal of family Floer cohomology is to study these complexes for varying $ q \in \Q$. To this end, we denote by 
\begin{equation}
  X_P \to P  
\end{equation}
the restriction of the fibration $\pi$ to a neighbourhood $P$ of $q$. Assume that such a neighbourhood is contractible, and fix a Lagrangian section
\begin{equation}
 \tau \co P \to X_P,
\end{equation}
which determines a base point $\tau P \cap F_p $ on each fibre of $\pi$ over $P$. 
By the Arnol'd-Liouville theorem, the projection of the cotangent bundle to $P$ factors through $X_P$
\begin{equation} \label{eq:cotangent_factors}
  T^* P \to X_P \to P
\end{equation}
and the choice of the section $\tau$ determines a canonical such factorisation which maps the $0$-section of $T^*P$ to the image of $\tau$. For this reason, we call the choice of $\tau$ a \emph{$0$-section} over $P$.

We now assume that $\phi L$ is transverse to all fibres $F_p$ over points $p \in P$, which can of course be achieved by shrinking $P$ since $\phi L$ and $F_q$ were assumed to be transverse.  Since $P$ is contractible, the transversality assumption implies that the intersection of $\phi L$ with $X_P$ is a union of components each of which is a (Lagrangian) section over $P$.  For each component $x$ of $\phi L \cap X_P$, choose a function
\begin{equation}
  g_x \co P \to \bR
\end{equation}
the graph of whose differential defines a lift of $x$ to $T^*P $ under Equation \eqref{eq:cotangent_factors}.  Let us write $x(p)$ for the intersection of a component $x$ of $\phi L \cap X_P$  with the fibre $F_p$ whenever $p \in P$. The function $g_x$ determines a path $\{ t dg_x\} _{t \in [0,1]}$  from the basepoint on $F_p$ to $x(p)$.

Our goal is to compare areas of the moduli spaces $  \Mbar_{p}(x(p), y(p)) $ for varying $p \in P$. To this end, we note that these moduli spaces decompose as unions of components labelled by classes $\beta \in \pi_2(X, \phi L \cup F_p)$. The transversality assumption (and contractibility of $P$) implies that we have canonical identifications between these relative homotopy groups. In particular, we say that $ u \in  \Mbar_{q}(x(q), y(q)) $ and $ v \in \Mbar_{p}(x(p), y(p)) $ are \emph{homotopic} if the classes they represent agree under this identification.

One of the basic invariants of homotopy classes of holomorphic curves is the energy:
\begin{equation}
    \cE(u) = \int_{B} u^{*}(\omega).
\end{equation}
Another such invariant is the class of the boundary
\begin{equation}
  [ \partial u] \in H_1(F_p, \bZ),
\end{equation}
which is defined by  concatenating the restriction of $u$ to the boundary component mapping to $F_p$ with the paths from these intersection points to the $0$-section that are determined by the functions $g_x$ and $g_y$. Whenever $u$ and $v$ are homotopic maps with boundary conditions on $F_q$ and $F_p$, the classes of their boundary are identified by the analogue of Equation \eqref{eq:classical_transport} on homology.

We now state the following basic result, which is a direct consequence of Stokes's theorem, and which expresses the difference of the energy of homotopic maps:
\begin{lem} \label{lem:areas-strips-flux-1}
If $u$ and $v$ are homotopic, then 
\begin{equation} 
  \cE(v) - \cE(u) = \langle p -q , [\partial u] \rangle + g_y(q) - g_y(p) + g_x(p) - g_x(q).
\end{equation} \qed
\end{lem}
\begin{figure}[h]
  \centering
\begin{tikzpicture}
\coordinate [label=below:$F_q$] (Xq) at (1,-1);
\coordinate [label=below:$F_{p}$] (Xp) at (3,-1);
\coordinate [label=right:$\phi L$] (L) at (3.5,0);
\coordinate [label=above:$\tau(P)$] (L) at (2,1);
\coordinate [label=above:$x$] (L) at (2,.3);
\coordinate [label=above:$y$] (L) at (2,-.6);
\draw (0,0) ellipse (.1 and 1);
\begin{scope}
\clip (3,-1)--(3,1)--(4,1)--(4,-1)--(3,-1);
\draw[style=dotted] (4,0) ellipse (.1 and 1);
\end{scope}
\begin{scope}
\clip (5,-1)--(5,1)--(4,1)--(4,-1)--(5,-1);
\draw (4,0) ellipse (.1 and 1);
\end{scope}
\draw[line width=1pt]  (0,1)--(4,1);
\draw  (0,-1)--(4,-1);
\begin{scope}
\clip (0,1) arc (90:-90:.1 and 1) -- (4,-1) arc (-90:90:.1 and 1) -- (0,1); 
\draw (-1,.75) -- (1,.25) .. controls  (3,-.25) and (3.5,1.5) .. (3.5,0) .. controls (3.5,-1.5) and  (3,.25) .. (1,-.25) -- (-1,-.75); 
\end{scope}
\draw  (3,1) arc (90:-90:.1 and 1);
\draw [style=dotted] (3,1) arc (90:270:.1 and 1);
\draw  (1,1) arc (90:-90:.1 and 1);
\draw [style=dotted] (1,1) arc (90:270:.1 and 1);
\begin{scope}
\clip (1,1) arc (90:-90:.1 and 1) -- (3,-1) arc (-90:90:.1 and 1) -- (1,1); 
\fill[blue] (-1,.75) -- (1,.25) .. controls  (3,-.25) and (3.5,1.5) .. (3.5,0) .. controls (3.5,-1.5) and  (3,.25) .. (1,-.25) -- (-1,-.75); 
\end{scope}

\begin{scope}[shift={(-3,3)}]
\draw (0,0) ellipse (.1 and 1);
\coordinate [label=below:$g_x(q) - g_x(p)$] (L) at (2,-1);
\begin{scope}
\clip (1,1) arc (90:-90:.1 and 1) -- (3,-1) arc (-90:90:.1 and 1) -- (1,1); 
\fill[blue] (-1,.75) -- (1,.25) .. controls  (3,-.25) and (3.5,1.5) .. (3.5,0) -- (3.5,1) -- (-1,1) -- cycle; 
\end{scope}
\begin{scope}
\clip (3,-1)--(3,1)--(4,1)--(4,-1)--(3,-1);
\draw[style=dotted] (4,0) ellipse (.1 and 1);
\end{scope}
\begin{scope}
\clip (5,-1)--(5,1)--(4,1)--(4,-1)--(5,-1);
\draw (4,0) ellipse (.1 and 1);
\end{scope}
\draw[line width=1pt]   (0,1)--(4,1);
\draw  (0,-1)--(4,-1);
\begin{scope}
\clip (0,1) arc (90:-90:.1 and 1) -- (4,-1) arc (-90:90:.1 and 1) -- (0,1); 
\draw (-1,.75) -- (1,.25) .. controls  (3,-.25) and (3.5,1.5) .. (3.5,0) .. controls (3.5,-1.5) and  (3,.25) .. (1,-.25) -- (-1,-.75); 
\end{scope}
\draw  (3,1) arc (90:-90:.1 and 1);
\draw [style=dotted] (3,1) arc (90:270:.1 and 1);
\draw  (1,1) arc (90:-90:.1 and 1);
\draw [style=dotted] (1,1) arc (90:270:.1 and 1);
\end{scope}

\begin{scope}[shift={(3,3)}]
\coordinate [label=below:$g_y(q) - g_y(p)$] (L) at (2,-1);
\begin{scope}
\clip (1,1) arc (90:-90:.1 and 1) -- (3,-1) arc (-90:90:.1 and 1) -- (1,1); 
\fill[blue] (-1,.75) -- (-1,1) -- (3.5,1) -- (3.5,0) .. controls (3.5,-1.5) and  (3,.25) .. (1,-.25) -- (-1,-.75); 
\end{scope}
\draw (0,0) ellipse (.1 and 1);
\begin{scope}
\clip (3,-1)--(3,1)--(4,1)--(4,-1)--(3,-1);
\draw[style=dotted] (4,0) ellipse (.1 and 1);
\end{scope}
\begin{scope}
\clip (5,-1)--(5,1)--(4,1)--(4,-1)--(5,-1);
\draw (4,0) ellipse (.1 and 1);
\end{scope}
\draw[line width=1pt]   (0,1)--(4,1);
\draw  (0,-1)--(4,-1);
\begin{scope}
\clip (0,1) arc (90:-90:.1 and 1) -- (4,-1) arc (-90:90:.1 and 1) -- (0,1); 
\draw (-1,.75) -- (1,.25) .. controls  (3,-.25) and (3.5,1.5) .. (3.5,0) .. controls (3.5,-1.5) and  (3,.25) .. (1,-.25) -- (-1,-.75); 
\end{scope}
\draw  (3,1) arc (90:-90:.1 and 1);
\draw [style=dotted] (3,1) arc (90:270:.1 and 1);
\draw  (1,1) arc (90:-90:.1 and 1);
\draw [style=dotted] (1,1) arc (90:270:.1 and 1);
\end{scope}

\begin{scope}[shift={(-3,-3)}]
\draw (0,0) ellipse (.1 and 1);
\coordinate [label=below:$\cE(u)$] (L) at (2,-1);
\begin{scope}
\clip (3,1) arc (90:-90:.1 and 1) -- (4,-1) -- (4,1) -- (3,1); 
\fill[blue] (-1,.75) -- (1,.25) .. controls  (3,-.25) and (3.5,1.5) .. (3.5,0) .. controls (3.5,-1.5) and  (3,.25) .. (1,-.25) -- (-1,-.75); 
\end{scope}
\begin{scope}
\clip (3,-1)--(3,1)--(4,1)--(4,-1)--(3,-1);
\draw[style=dotted] (4,0) ellipse (.1 and 1);
\end{scope}
\begin{scope}
\clip (5,-1)--(5,1)--(4,1)--(4,-1)--(5,-1);
\draw (4,0) ellipse (.1 and 1);
\end{scope}
\draw[line width=1pt]   (0,1)--(4,1);
\draw  (0,-1)--(4,-1);
\begin{scope}
\clip (0,1) arc (90:-90:.1 and 1) -- (4,-1) arc (-90:90:.1 and 1) -- (0,1); 
\draw (-1,.75) -- (1,.25) .. controls  (3,-.25) and (3.5,1.5) .. (3.5,0) .. controls (3.5,-1.5) and  (3,.25) .. (1,-.25) -- (-1,-.75); 
\end{scope}
\draw  (3,1) arc (90:-90:.1 and 1);
\draw [style=dotted] (3,1) arc (90:270:.1 and 1);
\draw  (1,1) arc (90:-90:.1 and 1);
\draw [style=dotted] (1,1) arc (90:270:.1 and 1);
\end{scope}

\begin{scope}[shift={(3,-3)}]
\coordinate [label=below:$\cE(v)$] (L) at (2,-1);
\begin{scope}
\clip (1,1) arc (90:-90:.1 and 1) -- (4,-1) -- (4,1) -- (1,1); 
\fill[blue] (-1,.75) -- (1,.25) .. controls  (3,-.25) and (3.5,1.5) .. (3.5,0) .. controls (3.5,-1.5) and  (3,.25) .. (1,-.25) -- (-1,-.75); 
\end{scope}
\draw (0,0) ellipse (.1 and 1);
\begin{scope}
\clip (3,-1)--(3,1)--(4,1)--(4,-1)--(3,-1);
\draw[style=dotted] (4,0) ellipse (.1 and 1);
\end{scope}
\begin{scope}
\clip (5,-1)--(5,1)--(4,1)--(4,-1)--(5,-1);
\draw (4,0) ellipse (.1 and 1);
\end{scope}
\draw[line width=1pt]   (0,1)--(4,1);
\draw  (0,-1)--(4,-1);
\begin{scope}
\clip (0,1) arc (90:-90:.1 and 1) -- (4,-1) arc (-90:90:.1 and 1) -- (0,1); 
\draw (-1,.75) -- (1,.25) .. controls  (3,-.25) and (3.5,1.5) .. (3.5,0) .. controls (3.5,-1.5) and  (3,.25) .. (1,-.25) -- (-1,-.75); 
\end{scope}
\draw  (3,1) arc (90:-90:.1 and 1);
\draw [style=dotted] (3,1) arc (90:270:.1 and 1);
\draw  (1,1) arc (90:-90:.1 and 1);
\draw [style=dotted] (1,1) arc (90:270:.1 and 1);
\end{scope}
\end{tikzpicture}  
  \caption{ }
  \label{fig:change-in-area-nearby-fibre}
\end{figure}

While the proof is omitted (see \cite[Lemma 3.2]{A-ICM} for a related result),  we shall comment on the basic intuition, assuming for simplicity that $[\partial u]$ vanishes (note that given $u$ and $g_x$, the function $g_y$ may be chosen to achieve this). As  illustrated in Figure \ref{fig:change-in-area-nearby-fibre}, the expression $g_x(q) - g_x(p) $ measures the area of the region bounded by the $0$-section and $x$, together with the fibres $F_q$ and $F_{p}$, while  $g_y(q) - g_y(p) $ measures the area of the corresponding region with one boundary on the section $y$. The difference between these expressions gives the area of the region bounded by the fibres $F_q$ and $F_{p}$, together with the sections $x$ and $y$, which is the difference in areas between homotopic strips in $\Mbar_q(x(q), y(q)) $ and  $\Mbar_{p}(x(p), y(p)) $.

\subsubsection{Fukaya's trick}
\label{sec:fukayas-trick}

While Lemma \ref{lem:areas-strips-flux-1} provides the key estimate comparing areas of strips with boundaries on nearby fibres, it suffers from the following major deficiency: the moduli spaces $\Mbar_q(x(q), y(q)) $  and $\Mbar_{p}(x(p), y(p)) $ may \emph{a priori} be completely different so that the Floer theories of $L$ with the fibres $F_q$ and $F_{p}$ may be unrelated. By using Gromov compactness, one can in fact see that the situation is not as dire as we just described: assuming all moduli spaces are regular, the components of $\Mbar_q(x(q), y(q)) $  and $\Mbar_{p}(x(p), y(p)) $ of bounded energy will be in bijective correspondence whenever $q$ and $p$ are sufficiently close. Unfortunately, the minimal distance between $q$ and $p$ which is required to achieve this bijection may shrink to $0$ if we drop the bound on energy.

In \cite{Fukaya-cyclic}, Fukaya introduced an elementary trick to resolve this difficulty using the fact that the space of tame almost complex structures on a symplectic manifold is open in the space of all almost complex structures. We discuss a minor variant which is adapted to our situation (see \cite{A-ICM} for more detail).

Recall that $\{J_t\}$ is a family of tame almost complex structures used to define $\Mbar_q(x(q),y(q)) $.  Let $\psi_{p}$ be a \emph{diffeomorphism} of $X$ mapping $F_q$ to $F_p$ which is supported over a contractible subset of $\Q$. Assume that $\psi_p$  preserves the submanifold $\phi L$, and that the pushforward $(\psi_p)_{*} J_t $ of $J_{t}$ with respect to $\psi_p$ is a tame almost complex structure for $t \in [0,1]$. The first condition is easy to achieve using the fact that $\phi L$ is transverse to all fibres over $F_q$, and the second by shrinking $P$ if necessary, since the space of tame almost complex structures is open in the space of all almost complex structures.
\begin{lem} \label{lem:fukayas-trick-1}
If $\Mbar_p(x(p),y(p))$ is defined with respect to $(\psi_p)_{*} J_t  $, then composition with $\psi_p$ defines a homeomorphism
\begin{equation}
 \xymatrix{     \Mbar_q(x(q),y(q))\ar[r]^{\psi_p} & \Mbar_p(x(p),y(p))}.
\end{equation} \qed
\end{lem}

Lemma \ref{lem:fukayas-trick-1} provides the key idea for our approach to family Floer cohomology. In particular, in addition to the usual choices of auxiliary structure that enter in the construction of Floer theoretic structures (see e.g. \cite{seidel-Book}), we must keep track of various diffeomorphisms, and in fact families of diffeomorphisms, which map fibres to each other, and which preserve the tameness of certain families of almost complex structures among other requirements. This is the main reason for the technical complexity which the reader will encounter in Section \ref{sec:floer-theory-conv}.

We shall return to the construction of the family Floer complex in Section \ref{sec:floer-equation}, after discussing the mirror side.

\subsection{The rigid-analytic $T$-dual}
\label{sec:rigid-analytic-t} \label{sec:twist-sheav-perf}
In this section, we construct the space $Y$ which will be mirror to $X$ by an analogue of SYZ duality in the non-archimedean setting. Such an approach was first suggested by Kontsevich and Soibelman \cite{KS}, and is discussed in more detail in \cite{A-ICM}.

The space $Y$ will be a rigid analytic space over a Novikov field; recall that we can associate to any field $\bfk$ the  universal Novikov field
\begin{equation}
  \Lambda = \{ \sum_{i=0}^{\infty} a_i T^{\lambda_i} | a_i \in \bfk, \, \lambda_i \in \bR, \, \lim_{i \to \infty} \lambda_i = +\infty \} .
\end{equation}
This is a non-archimedean field, whose  non-zero elements, denoted $\Lambda^*$, are equipped with a valuation  that assigns to a non-zero series the exponent of its leading order term:
\begin{align}
 \val \co \Lambda^* & \to \bR \\ 
  a_0 T^{\lambda_0} + \sum_{i=1}^{\infty} a_i T^{\lambda_i} & \mapsto \lambda_0,
\end{align}
where $a_0 \neq 0$, and $\lambda_0 < \lambda_i $ for $0 < i$. The elements of vanishing valuation form are called the unitary elements
\begin{equation}
  U_\Lambda = \{a_0  + \sum_{i=1}^{\infty} a_i T^{\lambda_i} |a_0 \neq 0, \, a_i \in \bfk, \, 0 < \lambda_i, \, \lim_{i \to \infty} \lambda_i = +\infty \}
\end{equation}
and form the analogue of the unit circle in $\bC$.

As a set, the analytic space $Y$ is simply the union
\begin{equation}
  \coprod_{q \in \Q} H^1(F_q,  U_\Lambda   )   
\end{equation}
where each fibre $H^1(F_q,  U_\Lambda   )  $ should be thought of by analogy with $H^1(F_q, S^1 ) = H^1(F_q, \bR ) /H^1(F_q, \bZ )   $ as the dual torus to $F_q$. In order to exhibit the analytic structure on $Y$, we shall construct $Y$ by gluing simpler pieces called \emph{affinoid domains} \cite{tate} which are associated to integral affine polygons in $\Q$. To this end, it is necessary to introduce some combinatorics to keep track of a cover consisting of such polygons.

\subsubsection{Covers of $\Q$ and $Y$}
\label{sec:covers-q}

Let $\Sigma$ be a partially ordered set labelling the vertices of a simplicial triangulation of $\Q$, i.e. there is a bijective correspondence between totally ordered subsets of $\Sigma$ and simplices of $\Q$, which assigns to a subset of $\Sigma$ the unique simplex spanned by its vertices.  Note in particular that all maximal totally ordered subsets of $\Sigma$ have $n+1$ elements, since these correspond to top-dimensional simplices in $\Q$, which is a manifold of dimension $n$. We write $\sigma_I$ for the cell associated to a (totally ordered) subset $I$  of $\Sigma$.

For each element $i \in \Sigma$, assume that we are given an integral affine polygon $P_i$ with basepoint $q_i = \sigma_i$ such that
\begin{align} \label{eq:nested_inclusion_polygons}
  P_j & \subset P_i \textrm{ whenever }i < j \\ \label{eq:Q-obtained-by-gluing}
Q & = \bigcup_{i} P_i / \!\! \sim 
\end{align}
where the equivalence relation $\sim$ is generated by the inclusion $P_j \subset P_i$ for $i<j$. As before, we denote the image of every point $p \in P_i$ in $H^1(F_i, \bR)  $  under the flux map by $p-q_i$, where $F_i \equiv F_{q_i}$. 

\begin{lem} \label{lem:cover_P_contain_sigma}
If $\Q$ is an integral affine manifold, there is a partially ordered set $\Sigma$  indexing the vertices of a simplicial triangulation and a cover by integral affine polygons, so that Equations \eqref{eq:nested_inclusion_polygons} and \eqref{eq:Q-obtained-by-gluing} hold, and such that $\sigma_I$ is contained in the interior of $P_{\max I}$ for each totally ordered subset $I \subset \Sigma$. 
\end{lem}
\begin{proof}
Pick a triangulation $\Delta$ of $\Q$ by affine (not-necessarily integral) simplices. For example, given an auxiliary Riemannian metric on $\Q$ we may choose, for a sufficiently small $\epsilon$,  a generic finite set in $\Q$ which is $\epsilon$-dense. There is a corresponding Delaunay triangulation \cite{BDG}, which is dual to the Voronoi diagram of nearest-neighbours to the points in the given finite set. Having chosen $\epsilon$ sufficiently small, we can construct the Delaunay triangulation to have flat simplices with respect to the affine structure (i.e. define vertices to be affine segments connecting vertices, faces to be affine simplices spanned by triples, etc).

For the remainder of the proof, we shall use letters $(i,j, ...)$ to denote vertices of the triangulation $\Delta$. Recall that the barycentric subdivision $B \Delta$ has vertices given by the barycenters of the simplices of $\Delta$.  We write $B\Delta_I$ for this barycenter, and
\begin{equation}
    B\Delta_{I_1 \subset \cdots \subset I_k} \subset \Delta_{I_k} 
\end{equation}
for a higher-dimensional simplex.   The partially ordered set $\Sigma$ will correspond to the vertices of the double barycentric subdivision so that its elements are given by sequences $I_1 \subset \cdots \subset I_k$ of nested simplices of $\Delta$ (note the conflict in notation with the statement, where $I$ refers to a subset of $\Sigma$).

To construct the integral affine cover,  choose for each set $I$ labeling a simplex $\Delta_I$ of $\Delta$ an integral affine polygon $P_I$ with non-empty interior $\mathring{P}_I$. To state the desired properties, we introduce the notation
\begin{equation}
 P_{ I_1 \subset \cdots \subset I_k} \equiv P_{I_1} \cap \cdots \cap P_{I_k}. 
\end{equation}
for a sequence $I_1 \subset \cdots \subset I_k  $ which corresponds to an element of $\Sigma$.

We require the following properties to hold:
\begin{align}
  \label{eq:P_I-cover_Q}
  & \parbox{36em}{ $\Delta_I$ is contained in the union of the open sets $\mathring{P}_J$ for $J \subset I$} \\  \label{eq:P_I-good-cover-B}
& \parbox{36em}{  $P_{I}$ is contained in the union of the interior of the barycentric simplices that contain the barycenter $B \Delta_I $.} \\
 \label{eq:P_I-good-cover}
& \parbox{36em}{  $B \Delta_{I_1 \subset \cdots \subset I_k} \cap  \mathring{P}_{I_1 \subset \cdots \subset I_k} \neq \emptyset$ for each nested sequence $I_1 \subset \cdots \subset I_k$.}
\end{align}
These conditions can be readily achieved by induction on the dimension of the simplex (see Figure \ref{fig:cover_nbd_cells} for an implementation of the same idea in a different context): first pick integral affine polygonal neighbourhoods $P_i$ of all vertices which are contained in the union of the interiors of  simplices of the barycentric subdivision which are adjacent to $\Delta_i$. In the inductive step, let  $\Delta_I' \subset \Delta_I$ be a convex subset of  the interior containing the barycenter such that $\Delta_I$ is covered by the interior of $\Delta_I'$ together with the union of the previously chosen polygons $P_J$ for $J \subset I$ and
\begin{equation} \label{eq:good_cover_q'}
 \Delta'_{I} \cap  B \Delta_{I_1 \subset \cdots \subset I_{k} \subset I} \cap  \mathring{P}_{I_1 \subset \cdots \subset I_{k}}  \neq \emptyset 
\end{equation}
for every nested sequence with largest subset $I$. This second property can be achieved because $   \mathring{P}_{I_1 \subset \cdots \subset I_{k}}  $  intersects $B \Delta_{I_1 \subset \cdots \subset I_k} $ non-trivially by induction, which is contained in the boundary of  $B \Delta_{I_1 \subset \cdots \subset I_k \subset I}$, which is itself a subset of $\Delta_I$.

We then choose $P_I$ to be an integral affine polygonal neighbourhood of $\Delta_I'$, which is sufficiently small so that Condition \eqref{eq:P_I-good-cover-B} holds.  Condition \eqref{eq:P_I-cover_Q} holds automatically, while Equation \eqref{eq:P_I-good-cover} follows from from the assumption that $\Delta'_I$ contains the barycenter of $\Delta_I$ for the sequence of length one, and in general from  Equation \eqref{eq:good_cover_q'}.

Note that Condition \eqref{eq:P_I-good-cover-B} implies that
\begin{equation}
    P_I \cap P_J = \emptyset
\end{equation}
unless $I \subset J$ or $J \subset I$. This implies that the cover $\{ P_I\}$ satisfies Equation \eqref{eq:Q-obtained-by-gluing}, while Equation \eqref{eq:nested_inclusion_polygons} follows by construction.

It remains to construct the triangulation labelled by $\Sigma$ so that any simplex is contained in the corresponding polytope. We shall construct the triangulation to have affine simplices (not necessarily integral) which are the (local) convex hull of their vertices. It thus suffices to pick points
\begin{equation}
 \sigma_{I_1 \subset \cdots \subset I_k} \in  P_{ I_1 \subset \cdots \subset I_k} \cap \mathring{ B \Delta}_{ I_1 \subset \cdots \subset I_k}.
\end{equation}
By construction, it automatically follows that whenever $J_1 \subset \cdots \subset J_\ell$ is a subset of the sequence $I_1 \subset \cdots \subset I_k$ we have
 \begin{equation}
 \sigma_{J_1 \subset \cdots \subset J_\ell} \in  P_{ I_1 \subset \cdots \subset I_k},
\end{equation}
hence the edge connecting these two vertices is contained in $P_{ I_1 \subset \cdots \subset I_k}$, and more generally for any simplex. Defining the partial order on the elements of $\Sigma$ to be given by reverse inclusion, the result follows.

Having finished the proof, we return to the notation where $(i,j, \cdots)$ correspond to elements of $\Sigma$.
\end{proof}
\begin{rem}
The proof of the result gives a stronger statement: namely, there is an integral affine polyhedral cover of $\Q$ which has dimension $n$, i.e. so that there are no non-empty $n+2$-fold intersections.
\end{rem}
Given a totally ordered subset $I$ in $ \Sigma$, let
\begin{equation} \label{eq:conventions_short_notation}
    P_I = \bigcap_{i \in I} P_i = P_{\max I} \textrm{, } q_I = q_{\max I} \textrm{, and } F_I = F_{\max I}.
\end{equation}

Let $Y_I$ denote the inverse image of $P_I$ under the valuation map
from $
H^1(F_{I}, \Lambda^*) $ to $H^1(F_{I}, \bR) \cong T_{q_I} \Q$. Since $ \Lambda^*$ splits as $\bR \times U_{\Lambda}$,  the affine isomorphism in Equation \eqref{eq:symplectic_transport} yields a natural identification 
\begin{equation}
  Y_I =  \coprod_{p \in P_I} H^1(F_p,  U_\Lambda   )
\end{equation}
of sets. Moreover, we obtain an isomorphism 
\begin{equation} \label{eq:lambda_iso_nearby_points}
 H^1(F_{j}, \Lambda^*) \to  H^1(F_{i}, \Lambda^*) .
\end{equation}
defining an inclusion of $Y_j$ into $Y_i$ whenever $i<j$.

\begin{defin}
  The \emph{rigid-analytic $T$-dual} of $X$ is the quotient
  \begin{equation}
    Y = \coprod_{i \in \Sigma} Y_i / \!\! \sim
  \end{equation}
where the equivalence relation identifies points in $Y_j$ with their image in $Y_j$ under the isomorphism from Equation \eqref{eq:lambda_iso_nearby_points}.
\end{defin}
\begin{rem} \label{rem:careful_cover_Q}
Equation \eqref{eq:Q-obtained-by-gluing} was missing from the corresponding discussion in \cite{A-ICM}. Note that, in general, the fact that the sets $P_i$ cover $Q$ implies that the natural map from the right hand side in Equation \eqref{eq:Q-obtained-by-gluing} to the left hand side is a surjection, which in fact admits a splitting (by mapping  $q \in Q$ to the corresponding point in $P_j$ for $j$ maximal among those containing $q$). Passing to the rigid analytic side, we would obtain a space which contains the rigid-analytic $T$-dual as a retract. 
\end{rem}

Each $Y_i$ is an affinoid domain, equipped with the ring of regular functions consisting of Laurent series on $H_1 ( F_{i}, \bZ)$ which converge in $Y_i$:
\begin{equation} \label{eq:element_of_reg_function_i}
  \scrO_i = {\Big \{}  \sum_{A \in H_{1}( F_{i}, \bZ)}  f_{A} z^{A}_{q_i}, \quad f_{A} \in \Lambda |   \forall v \in P_i, \lim_{|A| \to +\infty} \val(f_{A}) + \langle v,A \rangle= +\infty  {\Big\} }.
\end{equation} 
By construction, the inclusion $Y_j \to Y_i$ induces a (ring) map $
 \scrO_i \to \scrO_j$. The space $Y$ is therefore a rigid analytic space in the sense of Tate \cite{tate}. Given a totally ordered subset $I = \{ i_0 < i_1 < \ldots < i_d \}$, define the $\scrO_{i_0}$ module
\begin{equation}
  \scrO_I = \scrO_{i_d}  \otimes_{\scrO_{i_d}}\cdots  \otimes_{ \scrO_{i_0}} \scrO_{i_0}.
\end{equation}

\subsubsection{The twisting cocycle}
\label{sec:twisting-cocycle}

Let $X_I$ denote the inverse image of $P_I$ in $X$. Fix  Lagrangian sections $
  \triv_{i} \co  P_i \to X_i $
Moreover, choose functions $
f_{ij} \co P_{ij} \to \bR $
such that fibrewise addition by $df_{ij}$ agrees with the transition map between the restrictions of $\triv_i$ and $\triv_j$ to $P_{ij}$. We obtain a function
\begin{equation}
f_{ij} + f_{jk} - f_{ik} \co P_{ijk} \to \bR
\end{equation}
whose differential at $q_k$ lies in $H_{1}(F_{k}, \bZ) $ and define
\begin{equation} \label{eq:cocycle-alpha}
\alpha_{ijk} =  T^{ f_{ij}(q_k) + f_{jk}(q_k) - f_{ik}(q_k)} z_{q_k}^{d(f_{ij} + f_{jk} - f_{ik})} \in \scrO^*_{ijk}
\end{equation}
where $\scrO^*_{ijk} $ is the multiplicative group of non-vanishing functions.

If $v_{ijk}$ is a \v{C}ech representative of a class  $v \in H^{2}(Q; \bZ_2)$, define
\begin{equation}
\alpha_{ijk}^{v} = (-1)^{v_{ijk}} \alpha_{ijk}.
\end{equation}
Denote the corresponding cohomology class by $\alpha^v \in H^2(Y, \scrO^*)$.  Given a pair $(I,J)$ of ordered subsets of $\Sigma$ such that $\min J = \max I$, it is convenient to introduce the notation
\begin{equation}
  \alpha^{v}_{J,I} = \alpha^{v}_{\min I, \max I, \max J}.  
\end{equation}
\subsubsection{Twisted sheaves}
\label{sec:categ-assoc-cover}
\begin{defin}
  An \emph{$(\alpha^v)^{-1}$-twisted pre-sheaf of perfect $\scrO_Y$-modules} consists of (i) a finite rank graded free $\scrO_i$-module $\scrF(i)$ for each $i \in \Sigma$ and (ii) a degree $2-|I|$ map
\begin{equation}
\scrF_{I} \co   \scrO_I \otimes_{\scrO_{i_0}} \scrF(i_0)  \to  \scrF(i_d)
\end{equation}
for each ordered subset $I$. These data are required to satisfy the equation:
\begin{equation} \label{eq:a_oo_functor_equation}
  \sum_{i \neq  \min I, \max I} (-1)^{|I^{\leq}_i|+|a|}  \scrF_{I  \setminus  i }   a - \sum_{i \in I} \alpha^{v}_{I^\geq_i, I^\leq_i }  \scrF_{I^\geq_i}  \circ  \scrF_{I^\leq_i} a = 0.
\end{equation}
\end{defin}
\begin{rem}
  The sign conventions used are those of \cite{seidel-Book}. At the cohomological level, the justification for the terminology is given in \cite[Section 2.4]{A-ICM}, though the reader should be aware that the cocycle is on the wrong side of \cite[Equation (2.33)]{A-ICM}): from an algebraic point of view, the natural definition consists of maps $\scrG_I$ satisfying
\begin{equation}
  \sum_{i \neq  \min I, \max I} (-1)^{|I^{\leq}_i|+|a|}  (\alpha^{v}_{i_-ii_+})^{-1} \scrG_{I  \setminus i}   a -  \sum_{i \in I}   \scrG_{I^\geq_i} \circ \scrG_{I^\leq_i} a = 0,
\end{equation}
where $i_\pm$ are respectively the elements of $I$ immediately preceding and succeeding $i$. These two notions are equivalent by setting $  \scrG_I =  \scrF_I \cdot \prod_{i \in I}  \alpha_{ \min I i i_+}^{v}$.
\end{rem}
If $\scrF_{ij}$ is a quasi-isomorphism, we call such an object an $(\alpha^v)^{-1}$-twisted sheaf of perfect $\scrO_Y$-modules. Henceforth, we shall call such objects sheaves, specifying $\alpha^v$ only when necessary for clarity.

Sheaves form a differential graded category, with morphisms given by
\begin{equation}
  \Hom(\scrF, \scrF') \equiv \bigoplus_{I} \Hom_{\scrO_{\min I} }( \scrF(\min I), \scrF(\max I))[1-|I|],
\end{equation}
where the direct sum is taken over totally ordered subsets $I$, and $\scrF(\max I) $ is an $\scrO_{\min I}$ module via restriction. It is convenient to denote each summand in the right hand side by
\begin{equation}
     \Hom_I(\scrF, \scrF') \equiv  \Hom_{\scrO_{\min I} }( \scrF(\min I), \scrF(\max I))[1-|I|].
\end{equation}

Decomposing every element $T$ of this direct sum as $T = \sum_{I} T_I$, the differential acts on an element $a \in  \Hom_I(\scrF, \scrF')$ according to the formula:
\begin{multline} \label{eq:differential_morphism_sheaves}
  \mu^1T( a)  = \sum_{i \in I}\alpha^{v}_{I^\geq_i, I^\leq_i } \scrF_{I^\geq_i} \circ T_{I^\leq_i}  (a) + (-1)^{|I^{\leq}_i|+(1- | T|)}\alpha^{v}_{I^\geq_i, I^\leq_i }  T_{I^\geq_i} \circ \scrF_{I^\leq_i} (a) \\
\qquad + \sum_{i \neq \min I, \max I} (-1)^{|I^{\leq}_i|-1+|a|+|T|} T_{I \setminus i} (a).
\end{multline}
For a summand $T_I $, we may rewrite the above as
\begin{multline} \label{eq:differential_morphism_sheaves_alt}
  \mu^1 T_I  (a)  = \sum_{\substack{J  \\ \min J = \max I}}\alpha^{v}_{J,I}  \scrF_{J} \circ T_{I}  (a) + \\ \sum_{\substack{J \\ \min I = \max J}}(-1)^{|I^{\leq}_i|+(1- | T_I|)}\alpha^{v}_{I,J}  T_{I} \circ \scrF_{J} (a) + \sum_{\substack{ I \cup \{ j\}  \\ \min I < j < \max I}} (-1)^{|I^{\leq}_j|-1+|a|+|T_I|} T_{I} (a),
\end{multline}
where the sums on the right hand side are over all totally ordered subsets $J$ and $I \cup \{ j\}$.

The composition of morphisms in this category is given on  $a \in \Hom_I(\scrF, \scrF') $ by the formula
\begin{equation}
 \mu^{2}(S,T)(a) = \sum_{i \in I} (-1)^{(|S|-1) |I^\leq_i|} \alpha^{v}_{I^\geq_i, I^\leq_i }   S_{I^\geq_i}  \circ T_{I^\leq_i} (a).
\end{equation}

\subsection{The local mirror functor} \label{sec:floer-equation}

We now shift our attention back to the symplectic manifold $X$, whose Fukaya category we plan to relate to the category of twisted sheaves on $Y$.  We shall assume that  $\pi_2(\Q)=0$ which excludes the presence of holomorphic spheres in $X$, or holomorphic discs with boundary on any fibre $F_q $. Moreover, we shall only consider Lagrangians which are \emph{tautologically unobstructed} in the sense that
\begin{equation} \label{eq:everything_is_good_with_curves}
\parbox{35em}{there exists  $J_{L} \in \scrJ$ so that $L$ bounds no $J_{L}$-holomorphic discs.}
\end{equation}
Returning to the setting of Section \ref{sec:areas-strips-flux}, we now impose the condition that the family of almost complex structures $\{ J_{t} \} $ used to define the moduli space $\Mbar_q(x,y) $  satisfies the condition
\begin{equation}  \label{eq:push_J_forward}
  J_{1} = \phi_{*}(J_{L}) \equiv \phi_* \circ J_L \circ \phi^{-1}_*.
\end{equation}

For generic choices of families $\{ J_{t} \} $, classical Floer-theoretic methods imply, under these assumptions, that  $\Mbar_q(x,y)$ is a manifold with boundary given by the union
\begin{equation} \label{eq:boundary_Floer_moduli_space}
\coprod_{z \in  \phi L \cap F_{q}}  \cM_{q}(x,z) \times \cM_{q}(z,y).
\end{equation}

In particular, we can define (ungraded) Floer complexes $CF^*(L, F_q)$ over a Novikov field in characteristic $2$. Since there is much interest in working with $\bZ$-graded complexes in characteristic $0$, we discuss the necessary auxiliary conditions which $L$ must satisfy.

\subsubsection{Gradings and $\Pin^+$ structures}
\label{sec:floer-theory-with}

For any $J \in \scrJ$, there is a natural isomorphism of complex vector bundles $ T X \cong \pi^{*}(T \Q) \otimes_{\bR} \bC$. In particular, there is a natural homotopy class of \emph{quadratic complex volume forms} on $TX$ obtained by complexifying a density on $\Q$; let $\eta$ be a quadratic complex volume form in this class. We then require that $L$ be  \emph{graded} with respect to $\eta$, i.e. the map $  L \to \bR \bP^{1}  = \bR / \pi \bZ $ induced by $\eta$ is null-homotopic, and that a lift of this null-homotopy to the universal cover of $\bR \bP^1$ be fixed. For a fibre $L = F_q$, the map $  F_q \to \bR \bP^{1} $ is constant, and we fix the lift to $\bR$ with value $\pi$.  This ensures that the Floer complexes are $\bZ$ graded, as explained in \cite [Section (12b)]{seidel-Book}, since we can associate to each intersection point $x \in \phi L \cap F_q$ a well-defined Maslov index $ \deg(x) \in \bZ  $. The key reason for introducing the Maslov index is that the dimension of moduli spaces of strips can be expressed in terms of it; more precisely, in the setting of Section \ref{sec:areas-strips-flux}, we have
\begin{equation}
  \label{eq:dimension_moduli_strips}
 \dim_{\bR} \Mbar(x,y) =  \deg(x)  - \deg(y) -1.
\end{equation}

In order to work over a field of arbitrary characteristic, consider a class $w \in H^2(\Q,\bZ_2)$ which is the second-Stiefel Whitney class of a vector bundle $E$. We require that 
\begin{equation} \label{eq:relative_Pin_conditions}
  \parbox{33em}{the restriction of $\pi^*w$ to $L$ agree with $w_2(L)$.}
\end{equation}
\begin{rem}
The most important cases of interest are the trivial case and $w= w_2(\Q)$. As in \cite{FOOO}, one can drop the condition that $w$ be a Stiefel-Whitney class: the restriction of $w$ to the $3$-skeleton may be represented as the second Stiefel-Whitney class of a vector bundle. Since an orientation is a discrete datum, the uniqueness up to homotopy of a retraction from a $2$-dimensional complex to the $3$-skeleton suffices to establish consistency of orientations, and no higher coherence is required, even in the parametrised setting.   
\end{rem}
Assumption \eqref{eq:relative_Pin_conditions} implies that $\pi^*E|L \oplus TL$  admits a $\Pin^+$ structure which we fix; this is the choice of a \emph{relative $\Pin^+$ structure} on $L$. Given a Hamiltonian isotopy $\phi$, we obtain a corresponding relative $\Pin^+$ structure on $\phi L$. For the fibres, we choose a $\Pin^+$ structure on $T^*_{q}\Q \oplus E_{q}$, which induces a relative  $\Pin^+$ structure on $F_q$ using the isomorphism between $ T^*_{q}\Q  $ and $TF_q$.

 Given an intersection point $x \in \phi L \cap F_q$, pick a path $\gamma_x$ of linear Lagrangian subspaces of $T_xX$  starting at $T_x \phi L$ and ending at $T_xF_q$, in the homotopy class prescribed by the graded lift of these Lagrangian subspaces (see \cite[Section (11j)]{seidel-Book}). We obtain a vector bundle over the interval with fibre $
  \gamma_x(t) \oplus \pi^* E$.  The choice of relative $\Pin^+$ structures on $L$ and $F_q$  yields $\Pin^+$ structures on the restrictions of this bundle to $0$ and $1$. Let $\nu_x$ denote the free abelian group  generated by the two choices of extensions of this $\Pin^+$ structure to the interval, with the relation that their sum vanishes. Let $\ro_x$ denote the determinant line of the Cauchy-Riemann operator on complex linear maps from the upper half plane to $TX$ with Lagrangian boundary conditions $\gamma_x(t)$ (extended by $\phi L$ and $F_q$ outside the interval), and $|\ro_x| $ its orientation line. We define
\begin{equation} \label{eq:generator_Floer_x}
  \delta_x = |\ro_x| \otimes \nu_x.
\end{equation}

\subsection{Local mirror construction} \label{sec:local-mirr-constr}

Given $i \in \Sigma$, we recall in this section the construction of an $\scrO_i$-module $\scrF(L)$ associated to the Lagrangian $L$: assume that the cover $P_i$ is sufficiently fine that there exist Hamiltonian diffeomorphisms $\phi_i$ such that $\phi_i L$ is transverse to all fibres over $P$, and that, for all $p \in P$, there exists a diffeomorphism $\psi_{p}$ mapping $F_q$ to $F_p$,  preserving the submanifold $\phi L$, and such that $(\psi_p)_{*} J_t $ is tame.

Given a pair of intersection points $x_i, y_i \in  F_{i} \cap \phi_i L$ and an element $u \in \cM_{q_i}(y_i, x_i)$, denote the orientation line of the linearised Cauchy-Riemann operator at $u$ by 
\begin{equation}
  \delta_u = | \det D_u| =  |\coker D_u  |^{\vee} \otimes | \ker D_u |,
\end{equation}
where the absolute value symbol stands for the line of orientations, and $|V|^\vee$ is the dual line.   Index theory determines a canonical isomorphism (see \cite[Remark 11.6]{seidel-Book}):
\begin{equation} 
\delta_u \otimes \delta_{x_i} \cong \delta_{y_i}.
\end{equation}
Assuming that $\deg(y_i) = \deg(x_i) +1 $, the moduli space $ \cM(y_i, x_i) $  consists only of rigid curves, and  $\ker D_u $ is $1$-dimensional, which implies that it is generated by translation in the $s$-direction. Fixing the orientation of this kernel corresponding to the positive direction, yields a map
\begin{equation} \label{eq:map_strip_differential}
  \partial_u \co \delta_{x_i} \to \delta_{y_i};
\end{equation}
we denote by $\mu^1_u$ the product of $\partial_u$ by $(-1)^{\deg x}$.

The family Floer module and the differential are given by
\begin{align}
 \scrF(L,i)   & \equiv \bigoplus_{x_i \in \phi_i L \cap F_i} \scrO_i \otimes \delta_{x_i}  \\
  \scrF_i \co   \scrF(L,i) & \to   \scrF(L,i)[1] \\
\scrF_i| \delta_{x_i} & = \bigoplus_{y_i}  \sum_{u \in \cM_{q_i}(y_i, x_i)} T^{\cE(u)} z^{[\partial u]}_{i} \otimes \mu^1_u.
\end{align}

Fukaya's fundamental observation \cite{Fukaya-cyclic} is that Gromov compactness implies that this map is well defined, i.e. the expression $\sum_{u \in \cM(y_i, x_i)} T^{\cE(u)} z^{[\partial u]} $ gives a function in $\scrO_i$ (see \cite[Proposition 3.3]{A-ICM}). Indeed,  the condition of lying in $\scrO_i $ is equivalent to $T$-adic convergence at every point $z \in Y_P$. Assuming that $z$  lies over a point $p \in P_i$, we first use Lemma \ref{lem:fukayas-trick-1} to identify $\Mbar_{q}(y_i, x_i)$ with  $\Mbar_{p}(y_i(p), x_i(p))$, and then  Lemma \ref{lem:areas-strips-flux-1} to see that the result of evaluating $z$ at such a point recovers the Floer differential for the Lagrangian Floer theory of $\phi L$ with $F_p$ (equipped with a $U_{\Lambda}$ local system).  It is now a standard fact that the Floer differential converges as a consequence of Gromov compactness which asserts that there are only finitely many rigid holomorphic curves with bounded energy.

\subsection{Statement of the main theorem, and outline of the paper}
\label{sec:what-remains-be}
The main result of this paper is the following:
\begin{thm} \label{thm:main}
Let $X \to Q$ be a Lagrangian torus fibration with $\pi_2(X) = 0$, and $L$ and $L'$ Lagrangians satisfying Condition \eqref{eq:everything_is_good_with_curves}. Given  a sufficiently fine cover of $Q$, we can associate to $L$ and $L'$ (twisted) sheaves $\scrF(L)$ and $\scrF(L)$ of perfect complexes (with respect to the induced cover of $Y$),  as well as maps
\begin{equation}
  \xymatrix{  CF^*(L,L') \ar[r]^{\scrC} &  \Hom(\scrF(L), \scrF(L'))  \ar[r]^{\scrP} & CF^*(L,L').}  
\end{equation}
whose composition is homotopic to the identity up to sign.

Given a finite collection of Lagrangians, the map $\scrC$ extends to a faithful $A_\infty$ functor from the corresponding Fukaya category to the category of twisted sheaves of perfect complexes.  
\end{thm}

We now indicate how the proof can be pieced together from the paper:

\begin{itemize}
\item Section \ref{sec:unif-small-choic} introduces the precise notion of a sufficiently fine cover.
\item The construction of the twisted sheaf $\scrF(L)$ is done in Section \ref{lem:Floer_gives_sheaf}, with Lemma \ref{lem:Floer_gives_sheaf} asserting that the necessary equations hold.
\item The chain map $\scrC$ is constructed in Section \ref{sec:from-floer-vcech}, with Lemma \ref{lem:C-is-chain-map} showing that it is a chain map.
\item The corresponding results for $\scrP$ are proved in Section \ref{sec:from-vcech-floer}, in particular Lemma \ref{lem:P-chain-map}.
\item The proof that the composition is homotopic to the identity is given in Proposition \ref{prop:homot-from-comp}. 
\item In the Appendix, we construct the $A_\infty$ functor. In fact, we give a slightly simpler description of this map in the Appendix than in the main part of the paper. The additional complexity of the paper's main construction comes from the need to see the holomorphic curves defining such a map arise as components of  the boundary of a moduli space of annuli.
\end{itemize}

We now explain how the results of Section \ref{sec:floer-coch-morph} rely on the previous sections:
\begin{itemize}
\item The construction of $\scrF$ entails the construction of maps $\scrF_{I}$ for $2 \leq |I|$  yielding a (twisted) sheaf (of perfect complexes) on $Y$. This will require the study of \emph{higher continuation maps} in Floer theory and their convergence.
\item The construction of $\scrC$ and $\scrP$ is conceptually not too different from that of higher continuation maps, but the combinatorics required to keep track of the various moduli spaces and to appropriately formulate convergence are significantly more complicated.
\end{itemize}
In order to do this, we introduce certain abstract moduli spaces in Section \ref{sec:famil-riem-surf} and the corresponding spaces of maps in Section ~\ref{sec:floer-theory}.  The convergence problems  are discussed in Section \ref{sec:floer-theory-conv} which is at the heart of the paper. The key idea it to choose a very fine triangulation of the base of the fibration, make controlled choices at the vertices of this triangulation, and associate to higher dimensional cells families of equations which interpolate between these.

The remaining sections begin the transition from moduli spaces to the algebraic structures constructed in Section \ref{sec:floer-coch-morph}:
\begin{itemize}
\item In Section \ref{sec:moduli-space-degen} we  show that the composite $CF^*(L,L') \to \Hom(\scrF(L), \scrF(L'))  \to CF^*(L,L')  $ may be interpreted as a moduli space of degenerate annuli \emph{parametrised by $\Q$.}
\item In Section \ref{sec:cardys-relation} we build a cobordism between the moduli space of degenerate annuli and a moduli space which defines the identity on Floer cohomology.  The main delicate point is that it is not possible to perform the gluing construction continuously in such a way that the annuli over every point in the base are obtained by gluing the degenerate annuli corresponding to that point. This is responsible for the notion of an annulus gluing function introduced in Definition \ref{def:annulus_gluing}.
\end{itemize}

\section{Families of Riemann surfaces} \label{sec:famil-riem-surf} \label{sec:param-moduli-spac}
\subsection{Adams's moduli space} \label{sec:adams-univ-curve}

Let $\cM_{2,d}$ denote the moduli space of discs with two boundary punctures and $d$ interior marked points $(z_1, \ldots, z_{d})$. Since the complement of two points on the boundary of a disc is biholomorphic to a strip, and the biholomorphism is unique up to translation, we obtain a subset
\begin{equation}
    \cM^{\{1/2\}, ord}_{2,d} \subset \cM_{2,d}
\end{equation}
consisting of configurations for which the marked points lie on $\bR \times \{ 1/2 \}$ after identification with the strip, and whose ordering along the real line is opposite to the ordering of the labels.  We obtain coordinates for $\cM^{\{1/2\}, ord}_{2,d}$ by taking the differences in the first coordinates of the marked points, which identifies this space with $(0,\infty)^{d-1}$. In particular, the  fibre of the universal curve over  $\cM_{2,d}$ at a point in $ \cM^{\{1/2\}, ord}_{2,d} $ with coordinates $(r_1, \ldots, r_{d-1}) $ is biholomorphic to a strip with marked points satisfying
\begin{equation}
z_{i+1} - z_i = (-2r_i, 0).
\end{equation}

\begin{figure}
\centering
\begin{tikzpicture}
  \begin{scope}
 \coordinate (0) at (0,0);
\coordinate (1) at (4,0);
\coordinate (2) at (4,4);
\coordinate (3) at (0,4);

\foreach \x in {0,...,3} \fill (\x) circle (8*\lw);

\draw[line width=2*\lw] (0)--(1)--(2)--(3)--cycle;

\ExtractCoordinate{$(0)$};
\Conthree{\XCoord+4.5*\mw}{\YCoord +4*\mw}{3}{0};
\Contwo{\XCoord - 0.5*\mw}{\YCoord +4*\mw}{3}{0};
\Con{\XCoord-.5*\mw}{\YCoord - 0.5*\mw}{3}{0};  
\ExtractCoordinate{$(0)!.5!(1)$};
\Contwo{\XCoord+.5*\mw}{\YCoord - 0.5*\mw}{3}{0};  
\ExtractCoordinate{$(1)$};
\Con{\XCoord+1.5*\mw}{\YCoord -.5*\mw}{3}{2}; 
\Conin{\XCoord+2.75*\mw}{\YCoord -.5*\mw}{0};
\Con{\XCoord+1.5*\mw}{\YCoord +4*\mw}{3}{2}; 
\Conintwo{\XCoord+2.75*\mw}{\YCoord +4*\mw}{0};
\ExtractCoordinate{$(2)$};
\Con{\XCoord+1.5*\mw}{\YCoord+0.5*\mw}{3}{2};
\Conin{\XCoord+2.75*\mw}{\YCoord+0.5*\mw}{1}; 
\Conin{\XCoord+4*\mw}{\YCoord+0.5*\mw}{0}; 
\ExtractCoordinate{$(2)!.5!(3)$};
\Contwo{\XCoord-.25*\mw}{\YCoord+0.5*\mw}{3}{1};
\Conin{\XCoord+1*\mw}{\YCoord+0.5*\mw}{0}; 
\ExtractCoordinate{$(3)$};
\Con{\XCoord-1.75*\mw}{\YCoord+0.5*\mw}{3}{1}; 
\Conin{\XCoord-.5*\mw}{\YCoord+0.5*\mw}{0}; 

  \end{scope}

\end{tikzpicture}

\caption{The moduli space $ \Adams_{3}$, with strata labelled by the fibre of the universal curve.}
\label{fig:moduli-vs-adams}
\end{figure}
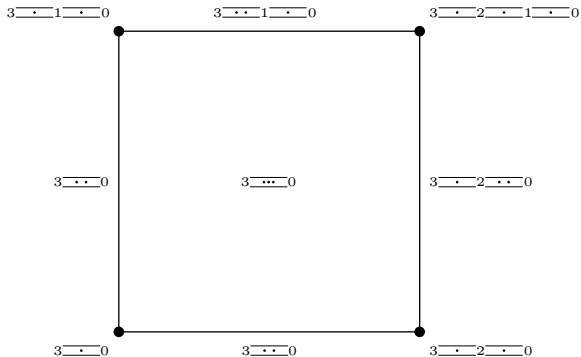

The universal curve over $ \cM_{2,d}^{\{1/2\}, ord} $ naturally extends to a universal curve $\Ubar_{d}$ with marked point over the product
\begin{equation} \label{eq:Adams_moduli_product}
\Adams_{d} \equiv [0,\infty]^{d-1},
\end{equation}
with the property that setting a coordinate equal to $\infty$ increases the number of components by one, while setting it equal to $0$ does not change the number of components but reduces the number of marked points by one. We shall presently give an explicit description of this universal curve, but it is useful to note that it can be constructed more abstractly: consider the closure $  \Mbar^{\{1/2\}, ord}_{2,d}  $ of $ \cM^{\{1/2\}, ord}_{2,d}$ in $ \Mbar_{2,d}  $  and note that there is a natural projection 
\begin{equation}
    \Mbar^{\{1/2\}, ord}_{2,d} \to \Adams_{d}
\end{equation}
which is obtained by forgetting all components which are not discs. The universal curve over $\Adams_{d} $ is then obtained by taking the union of disc components of the universal curve over $ \Mbar^{\{1/2\}, ord}_{2,d}  $. For example, $\Mbar^{\{1/2\}, ord}_{2,3} $ can be obtained from $\Adams_{3}$ by replacing the stratum with coordinates $(0,0)$ with an interval. The fibre of $\Ubar_{3} $ over this point is a strip with a single interior marked point (see Figure \ref{fig:moduli-vs-adams}), whereas the fibre of the universal curve over a point in $\Mbar^{\{1/2\}, ord}_{2,3}   $  which projects to this vertex is a nodal curve which is the union of this strip with a sphere with $3$ marked points in addition to the node. The cross ratio between the $4$ points is a real number because the condition of lying in $\Mbar^{\{1/2\}, ord}_{2,3}   $ is a real colinearity condition, and the corresponding real number parametrises the fibre of $ \Mbar^{\{1/2\}, ord}_{2,3}$ over the points $(0,0)$. 
\begin{rem}
We can alternatively construct $\Adams_{d} $ as a subset of the \emph{configuration space} of points on a disc, in which points are allowed to collide (unlike in the moduli space where they bubble).  In this case, $\Ubar_{d}$ is simply the restriction of the universal curve over configuration space.
\end{rem}

 We shall now give an explicit description of this fibre over a point $\vr \in \Adams_{d} $: first, we associate to each pair of non-negative real numbers $(r,r')$ the finite strip
\begin{equation}
B_{r,r'} = [-r,r'] \times [0,1].  
\end{equation}
Whenever $r$ or $r'$ are infinite, we let the corresponding half of the interval be open:
\begin{equation}
B_{\infty, r'} = (-\infty, r'] \times [0,1]  \textrm{ and }  B_{r,\infty} = [-r , \infty) \times [0,1].
\end{equation}
Finally, we set $B_{\infty, \infty} = B$, $B_+ = B_{0, \infty} $, and $B_- = B_{\infty, 0}$. 

Given $\vr \in \Adams_{d} $, consider the union of strips
\begin{equation} \label{eq:union_disjoint_finite_strips}
     B_{\infty, r_{d-1}}  \coprod B_{r_{d-1}, r_{d-2}} \coprod \cdots \coprod B_{r_2, r_1}  \coprod B_{r_1, \infty}.
\end{equation}
The fibre $\Ubar_{\vr}  $ is the quotient of the above union by the following equivalence relation: if $r_i$ is finite, we  identify
\begin{align}
    \{ r_i \} \times [0,1]  \sim \{ -r_i \} \times [0,1] 
\end{align}
where the first interval lies in $B_{r_{i+1}, r_i}$ and the second in $B_{r_i, r_{i-1}}$. This fibre has marked points which are the images of the points $ (0, 1/2)$ in each strip $B_{r_i, r_{i-1}} $, as shown in Figure \ref{fig:gluing-strips-to-strip}.

\begin{figure}
  \centering
  \begin{tikzpicture}

\draw[line width=4*\lw,dotted] (-12*\mw pt,-1*\mw pt) -- ( -13*\mw pt, -1*\mw pt);
\draw[line width=4*\lw,dotted] (-12*\mw pt,1*\mw pt) -- ( -13*\mw pt, 1*\mw pt);
\draw[line width=4*\lw,dotted] (12*\mw pt,-1*\mw pt) -- ( 13*\mw pt, -1*\mw pt);
\draw[line width=4*\lw,dotted] (12*\mw pt,1*\mw pt) -- ( 13*\mw pt, 1*\mw pt);
\draw[line width=4*\lw] (-12.5*\mw pt,-1*\mw pt) -- ( 12.5*\mw pt, -1*\mw pt);
\draw[line width=4*\lw] (-12.5*\mw pt,1*\mw pt) -- ( 12.5*\mw pt, 1*\mw pt);

\draw[line width=4*\lw] (-4*\mw pt,-2*\mw pt) -- ( 2*\mw pt, -2*\mw pt);
\node[below] at (-4*\mw pt, -2*\mw pt)  {$-r_3$};
\draw[dotted] (-4*\mw pt, 1*\mw pt) -- (-4*\mw pt, -1*\mw pt) ;
\draw[line width=4*\lw] (-4*\mw pt,-2*\mw pt -1/8*\mw pt) -- (-4*\mw pt,-2*\mw pt +1/8*\mw pt);
\node[below] at (2*\mw pt, -2*\mw pt)  {$r_2$};
\draw[dotted] (2*\mw pt, 1*\mw pt) -- (2*\mw pt, -1*\mw pt) ;
\draw[line width=4*\lw] (2*\mw pt,-2*\mw pt -1/8*\mw pt) -- (2*\mw pt,-2*\mw pt +1/8*\mw pt);
\node[below] at (0*\mw pt, -2*\mw pt)  {$0$};
\draw[line width=4*\lw] (0*\mw pt,-2*\mw pt -1/8*\mw pt) -- (0*\mw pt,-2*\mw pt +1/8*\mw pt);
\fill (0*\mw pt, 0*\mw pt)  circle (8*\lw);

\draw[line width=4*\lw] (8*\mw pt,2*\mw pt) -- ( 2*\mw pt, 2*\mw pt);
\node[above] at (8*\mw pt, 2*\mw pt)  {$r_1$};
\draw[dotted] (8*\mw pt, 1*\mw pt) -- (8*\mw pt, -1*\mw pt) ;
\draw[line width=4*\lw] (8*\mw pt,2*\mw pt -1/8*\mw pt) -- (8*\mw pt,2*\mw pt +1/8*\mw pt);
\node[above] at (2*\mw pt, 2*\mw pt)  {$-r_2$};
\draw[line width=4*\lw] (2*\mw pt,2*\mw pt -1/8*\mw pt) -- (2*\mw pt,2*\mw pt +1/8*\mw pt);
\node[above] at (4*\mw pt, 2*\mw pt)  {$0$};
\draw[line width=4*\lw] (4*\mw pt,2*\mw pt -1/8*\mw pt) -- (4*\mw pt,2*\mw pt +1/8*\mw pt);
\fill (4*\mw pt, 0)  circle (8*\lw);

\draw[line width=4*\lw] (-10*\mw pt,2*\mw pt) -- ( -4*\mw pt, 2*\mw pt);
\node[above] at (-10*\mw pt, 2*\mw pt)  {$-r_4$};
\draw[dotted] (-10*\mw pt, 1*\mw pt) -- (-10*\mw pt, -1*\mw pt) ;
\draw[line width=4*\lw] (-10*\mw pt,2*\mw pt -1/8*\mw pt) -- (-10*\mw pt,2*\mw pt +1/8*\mw pt);
\node[above] at (-4*\mw pt, 2*\mw pt)  {$r_3$};
\draw[line width=4*\lw] (-4*\mw pt,2*\mw pt -1/8*\mw pt) -- (-4*\mw pt,2*\mw pt +1/8*\mw pt);
\node[above] at (-8*\mw pt, 2*\mw pt)  {$0$};
\draw[line width=4*\lw] (-8*\mw pt,2*\mw pt -1/8*\mw pt) -- (-8*\mw pt,2*\mw pt +1/8*\mw pt);
\fill (-8*\mw pt, 0*\mw pt) circle (8*\lw);

\draw[line width=4*\lw] (-12.5*\mw pt,-2*\mw pt) -- (-10*\mw pt, -2*\mw pt);
\node[below] at (-12*\mw pt, -2*\mw pt)  {$0$};
\node[below] at (-10*\mw pt, -2*\mw pt)  {$r_4$};
\draw[line width=4*\lw] (-10*\mw pt,-2*\mw pt -1/8*\mw pt) -- (-10*\mw pt,-2*\mw pt +1/8*\mw pt);
\draw[line width=4*\lw,dotted] (-12.5*\mw pt,-2*\mw pt) -- ( -13*\mw pt, -2*\mw pt);
\fill (-12*\mw pt, 0*\mw pt) circle (8*\lw);

\draw[line width=4*\lw] (12.5*\mw pt,-2*\mw pt) -- (8*\mw pt, -2*\mw pt);
\node[below] at (12*\mw pt, -2*\mw pt)  {$0$};
\node[below] at (8*\mw pt, -2*\mw pt)  {$-r_1$};
\draw[line width=4*\lw] (8*\mw pt,-2*\mw pt -1/8*\mw pt) -- (8*\mw pt,-2*\mw pt +1/8*\mw pt);
\draw[line width=4*\lw] (12*\mw pt,-2*\mw pt -1/8*\mw pt) -- (12*\mw pt,-2*\mw pt +1/8*\mw pt);
\draw[line width=4*\lw,dotted] (12.5*\mw pt,-2*\mw pt) -- ( 13*\mw pt, -2*\mw pt);
\fill (12*\mw pt, 0*\mw pt) circle (8*\lw);

\draw[line width=4*\lw] (-12*\mw pt,-2*\mw pt -1/8*\mw pt) -- (-12*\mw pt,-2*\mw pt +1/8*\mw pt);

\end{tikzpicture}
  \caption{Decomposition of $\Ubar_{\vr}$ into finite strips.}
  \label{fig:gluing-strips-to-strip}
\end{figure}

\begin{rem}
In \cite{adams}, Adams constructed a family  of paths in the $d$-simplex from the initial to the terminal vertex which are parametrised by the $d-1$-cube. Identifying the interval $[0,1]$ with $[0,\infty]$, we obtain the moduli spaces $\Adams_{d}$ as we defined them. We shall not explicitly need the connection between our construction and Adams's.
\end{rem}

The fact that the moduli space of discs with $d-1$ marked points is denoted $\Ubar_d$ is justified by the eventual Floer-theoretic application as explained in the next remark:
\begin{rem} \label{asi:first_informal_comment}
Assuming some familiarity with Floer theory, we give an informal description of how these moduli spaces shall be used. The informal ideas discussed here will be implemented in detail later in the text, hence can be safely skipped by those seeking only precise definitions. 

Assume that $L$ and $F$ are Lagrangian submanifolds, and that $\{J_i\}_{i=0}^{d}$ are choices of almost complex structures with respect to which one can define the Floer complexes $\{ CF^*(L,F; J_i) \}_{i=0}^{d}$. A choice of path connecting $J_i$ to $J_k$ allows one to write a \emph{continuation equation} on the strip $\bR \times [0,1]$, which defines a chain map $ CF^*(L,F; J_i) \to   CF^*(L,F; J_k) $.

We shall only be interested in such continuation maps for $i < j$; in particular, there are finitely many ways (in fact, exactly $2^{d-1}$) of composing these continuation maps to obtain a map $ CF^*(L,F; J_0) \to   CF^*(L,F; J_d) $. Each such composition corresponds to a path, along the $1$-skeleton of the simplex $\Delta_d$, with initial point $0$ and terminal point $d$, with the property that any intermediate vertices appear in increasing order along the path. This is naturally a description of the vertices of $\Adams_{d}$ as follows:  the composition of continuation maps associated to the sequence  $(J_0, J_{i_1}, \cdots, J_{i_k}, J_d)$ corresponds to the vertex of  $\Adams_{d}$ with coordinates  labelled $(i_1, \ldots, i_k)$ equal to $\infty$, and all others equal to $0$.

Recall that the continuation map $ CF^*(L,F; J_i) \to   CF^*(L,F; J_k) $ is associated to a path $J_{ik}(s)$ of almost complex structures, parametrised by $s \in (-\infty,\infty)$, which agrees with $J_i$ for $ 0 \ll s$ and with $J_k$ for $s \ll 0$. In order to see a general point in $ \Adams_{d} $ arise from Floer-theoretic considerations, we heuristically think of such a path as obtained by smoothing a discontinuous path of almost complex structures which agree with $J_i$ for $ 0 < s$ and $J_k$ for $s < 0$. In later sections, we shall choose smooth paths, but for this informal discussion, it is simpler to take discontinuous paths, which have the advantage of being canonical, even though the continuation map is strictly speaking not defined for them.

With the above in mind, we can interpret a point in $\Adamsop_{d}  $ as giving rise to a continuation map  $ CF^*(L,F; J_i) \to   CF^*(L,F; J_k) $.  Identify $[0,1)^{d-1}$ with $[0,\infty)^{d-1}$, and associate to a point $(r_1, \ldots, r_{d-1}) \in [0,\infty)^{d-1}$ the path of almost complex structures $J(s)$ given by the piece-wise conditions:
\begin{equation}
J(s) = J_k  \textrm{ if } 2( r_1 + \cdots + r_{k-1}) < s < 2( r_1 + \cdots + r_{k})
\end{equation}
Whenever all coordinates $r_k$ vanish, we obtain the continuation map $  CF^*(L,F; J_0) \to   CF^*(L,F; J_d) $, whereas, in the limit where all $r_k$ are infinite, we recover the composition of the continuation maps
\begin{equation}
   CF^*(L,F; J_0) \to   CF^*(L,F; J_1) \to \cdots \to    CF^*(L,F; J_d),
\end{equation}
by considering the Gromov-Floer limit of continuation equations. More generally, requiring that a given coordinate $r_k$ vanish corresponds to omitting it from the continuation map, i.e. considering the family of continuation equations corresponding to paths in the sub-simplex $ \Delta_{d-1} \subset \Delta_d$  obtained by omitting the $k$\th vertex. In limit $r_k \to \infty $, we obtain paths which pass through the $k$\th vertex; in Floer theoretic terms, this corresponds to factoring through $ CF^*(L,F; J_k) $.
\end{rem}

\subsection{Stratification of the Adams moduli space}

Before giving the explicit description of the boundary strata of $\Adams_{d} $ and the corresponding fibre of the universal curve, it is useful to introduce a more general choice of labels:

\begin{defin}
Let $K$ be a totally ordered set. The \emph{compactified Adams moduli space of paths}  $\Adams_{K}$ is the product
 \begin{equation}
      \Adams_{K} = [0,\infty]^{K \setminus \{ \min K, \max K \}}.
 \end{equation}
\end{defin}

We also denote by $\Adamsop_{K} \subset \Adams_{K} $ the open subset corresponding to the inclusion $[0,\infty) \subset [0,\infty]$. If $d = |K| -1$, let  $\Ubar_{K}$ denote the copy of the space $   \Ubar_{d}$  over $\Adams_{K}$, and $\Ubar_{\vr}$ denote the fibre over $\vr \in \Adams_{K}$. The fibre over each point in $\Adamsop_{K}$ is a strip, and the elements of $K \setminus \{ \min K, \max K \} $ label the intervals between the marked points (as subsets of $\bR \times \{ 1/2 \}$),  while $\min K$ labels the positive end, and $\max K$ the negative end.  The fibre over points in the complement of $\Adamsop_K$ are disjoint unions of strips.  

To keep the notation consistent, one can identify the space $ \Adams_{d} $ with that corresponding to $K = \{ 0, \ldots, d \}$ (to reduce the complexity of the notation, we often write $K=01\cdots d$ for such a set). 

The space $  \Adams_{K}  $ is naturally stratified, and the partially ordered set of strata, with ordering given by inclusion, consists of pairs of subsets $I$ and $J$ of $K$ such that
\begin{equation} \label{eq:boundary_stratum_label_adams}
\{\min K, \max K \} \subset  I \subset J.
\end{equation}
The partial ordering is such that the pair $I \subset J$ precedes $I' \subset J'$ whenever $ I \subset I' \subset J' \subset J$. We write $\Adams_{I \subset J} $ for the stratum corresponding to an element of this poset, and note the identification
\begin{equation}
  \label{eq:top_dim_stratum_adams}
 \Adams_{ \{\min K, \max K \}  \subset K} \equiv \Adams_{K}.
\end{equation}

\begin{figure}
\centering
\begin{tikzpicture}
\begin{scope}                          
\coordinate (1) at (-1.25,-1/2);
\node at (1) {${\scriptstyle 02 \subset 02}$};
\coordinate (2) at (1.25,-1/2);
\node at (2) {${\scriptstyle 012 \subset 012}$};

\draw[line width=2*\lw] (-1,0)--(1,0);
\fill (-1,0) circle (8*\lw);
\fill (1,0) circle (8*\lw);

\node at (-1,1.375) {${\scriptstyle 0}$};
\draw[line width=\lw] (-1-1/8,2-1/2 )--(-1-1/8 , 2+1/2);
\draw[line width=\lw] (-1+1/8, 2 - 1/2)--(-1+1/8, 2+1/2);
\fill (-1, 2) circle (2*\lw);
\node at (-1,2.675) {${\scriptstyle 2}$};

\node at (0,1.375) {${\scriptstyle 0}$};
\draw[line width=\lw] (0-1/8,2-1/2 )--(0-1/8 , 2+1/2);
\draw[line width=\lw] (1/8, 2 - 1/2)--(1/8, 2+1/2);
\fill (0, 2+1/8) circle (2*\lw);
\fill (0, 2-1/8) circle (2*\lw);
\node at (0,2.675) {${\scriptstyle 2}$};

\node at (1,.75) {${\scriptstyle 0}$};
\draw[line width=\lw] (1-1/8,1.5-1/8-1/2 )--(1-1/8 , 1.5-1/8+1/2);
\draw[line width=\lw] (1+1/8, 1.5-1/8 - 1/2)--(1+1/8, 1.5-1/8+1/2);
\fill (1, 1.5-1/8) circle (2*\lw);
\node at (1,2) {${\scriptstyle 1}$};
\draw[line width=\lw] (1-1/8,2.5+1/8-1/2 )--(1-1/8 , 2.5+1/8+1/2);
\draw[line width=\lw] (1+1/8, 2.5+1/8 - 1/2)--(1+1/8, 2.5+1/8+1/2);
\fill (1, 2.5+1/8) circle (2*\lw);
\node at (1,3.25) {${\scriptstyle 2}$};
\end{scope}
\end{tikzpicture}

\caption{The moduli space $\Ubar_{012}$ over $\Adams_{012}$.}
\label{fig:continuation_family_basic_case}
\end{figure}

The geometric description of strata of $ \Adams_{K} $ is recovered as follows (see Figure \ref{fig:continuation_family_basic_case}): the pair $I \subset J$ labels the stratum of $  \Adams_{K} = [0,\infty]^{K \setminus \{ \min K, \max K \}} $  for which the coordinates in $I \setminus   \{ \min K, \max K \}$ equal $\infty$, and the coordinates in $K \setminus J$ vanish. 
\begin{rem}
Continuing the discussion of Remark \ref{asi:first_informal_comment}, recall that the vanishing of a coordinate $r_k$ corresponds to omitting a given choice of almost complex structure $J_k$ from the construction of families of continuation equations, whereas requiring that it  equal $\infty$ corresponds to continuation maps which factor through the Floer complex for $J_k$. In particular, the stratum labelled by the pair $I \subset J$ corresponds to continuation equations constructed from the almost complex structures $J_k$ with $k \in J$, with the additional constraint that all maps factor through the Floer complexes for $J_i$ with $i \in I$. 
\end{rem}

There is an alternate description of the boundary strata which is often more useful: introduce the notation
\begin{equation} \label{eq:smaller_larger_sets}
  K^\geq_{i}  = \{ j \in K | j \geq i \}  \textrm{ and } K^\leq_{i}  = \{ j \in K | j \leq i \}.   
  \end{equation}
 \begin{lem} \label{lem:boundary_strata_adams}
  The boundary stratum corresponding to $I \subset J$ admits a natural product decomposition
  \begin{equation}
 \Adams_{I \subset J} \cong  \Adams_{J^\geq_{i_d}} \times \Adams_{J^\leq_{i_d}  \cap J^\geq_{i_{d-1}}} \times \cdots \times \Adams_{ J^\leq_{i_1} \cap J^\geq_{i_0} } \times \Adams_{J^\leq_{i_0}},
  \end{equation}
  where $I = \{\min K ,  i_0, i_1, \ldots, i_d , \max K\}$.  \qed
 \end{lem}

Given $i \in K \setminus \{ \max K , \min K \}$, and using Equation \eqref{eq:top_dim_stratum_adams}, the above result for facets yields a natural identification
\begin{align} \label{eq:subtrajectories_skip_i}
 \Adams_{\{\min K, \max K \} \subset K \setminus \{i \}} &  \cong \Adams_{K\setminus \{i\} }   \\  \label{eq:subtrajectories_factor_i}
\Adams_{\{\min K, i, \max K \} \subset K} &  \cong \Adams_{K^\geq_{i}} \times \Adams_{K^\leq_{i}} 
\end{align}
corresponding to the locus where the $i$\th coordinate vanishes or equals $\infty$.  The union of the images of these inclusions over all $i \in K \setminus \{ \max K, \min K \} $ covers the boundary of $\Adams_{K}$. 

 The identification in Equation \eqref{eq:subtrajectories_skip_i} induces a natural inclusion
\begin{align}  \label{eq:boundary_stratum_skip}
\Ubar_{K\setminus \{i\} } & \to \Ubar_{K},
\end{align}
which on the top stratum can be described in terms of the marked points $(z_1, \ldots, z_{d-1})$ giving rise to $ (z_1, \ldots, z_{i-1}, z_{i-1} , z_{i}, \ldots, z_{d-1})$. On the other hand, Equation \eqref{eq:subtrajectories_factor_i} induces a map
\begin{align}
 \label{eq:boundary_stratum_stop}
\Ubar_{K^\geq_{i}} \times \Adams_{K^\leq_{i}} \cup \Adams_{K^\geq_{i}} \times   \Ubar_{K^\leq_{i}}  & \to \Ubar_{K}.
\end{align}
If we restrict attention to the boundary strata of top dimension, this corresponds to the fact that  the fibre of $ \Ubar_{K}  $ over a point in $ \Adamsop_{K^\geq_{i}} \times \Adamsop_{K^\leq_{i}}  $ consists of a union of two curves, and that the marked points which occur to the left of  the interval labelled $i$ lie on one of these curves, with the other points lying on the other.
\begin{rem} \label{rem:convention-nors}
The case $|K|=1$ is a degenerate case of the above discussion. Whenever it appears, we shall fix the convention that $\Adams_{K}$ is a point.
\end{rem}

Denote by
\begin{equation}
   \Ubar_{K}^{\{0\}} \subset    \Ubar_{K} \supset   \Ubar_{K}^{\{1\}}
\end{equation}
the two subsets of $\Ubar_{K}$ which are represented by points of the form $(t,0)$ (respectively $(t,1)$) in a constituent finite strip of a fibre $\Ubar_{\vr}$.

\subsubsection{Gluing strips}
Consider the projection map 
\begin{equation} \label{eq:project_map_adams}
f_{I \subset J} \co \Adams_{K} \cong [0, \infty]^{K \setminus \{\min K, \max K \} } \to [0,\infty]^{J \setminus I} \cong \Adams_{I \subset J},
\end{equation}
and note that composition with the inclusion $  \Adams_{I \subset J} \to  \Adams_{K}$ is obtained by setting all coordinates labelled by $I \setminus \{ \min K, \max K\}$ equal to infinity, and all coordinates in $K \setminus J$ equal to $0$.

Expressing the boundary stratum $\Adams_{I \subset J}$ as a product of Adams moduli spaces yields a projection $   \Adams_{I \subset J} \to \Adams_{I' \subset J'} $
whenever $I \subset I' \subset J' \subset J$, so that the following  diagram commutes:
\begin{equation} \label{eq:project_boundary_Adams_commute}
  \xymatrix{  \Adams_{K} \ar[r] \ar[dr] &    \Adams_{I \subset J} \ar[d] \\
&  \Adams_{I' \subset J'}. }
\end{equation}

In order to define the gluing map, let us fix the maps
\begin{equation} \label{eq:maps_infinite-finite-strip}
B_{r,\infty} \to B_{r,r'} \textrm{ and } B_{\infty,r'} \to B_{r,r'} 
\end{equation}
which split the inclusion $  B_{r,r'} \to B_{r,\infty}$ and $B_{r,r'} \to B_{\infty,r'}$ and are given outside these regions by the projections
\begin{align}
[r', \infty) \times [0,1] & \to \{ r' \} \times [0,1] \\
(-\infty,-r] \times [0,1] & \to \{  -r \} \times [0,1] .
\end{align}

Returning to the description of the fibre $\Ubar_{\vr}$ in Equation \eqref{eq:union_disjoint_finite_strips}, we see that Equation \eqref{eq:maps_infinite-finite-strip} induces a surjective (continuous) map
\begin{equation}
 \Ubar_{f_{I \subset J}( \vr)}  \to \Ubar_{\vr}
\end{equation}
which is a diffeomorphism outside the infinite ends of $  \Ubar_{f_{I \subset J} (\vr)} $ that are labelled by elements of $ K \setminus J $. Letting $\vr$ vary over a neighbourhood $  \nu \Adams_{I \subset J} $  of $  \Adams_{I \subset J} $ in $\Adams_K$, we obtain a gluing map
\begin{equation} \label{eq:gluing_maps_universal_curve_paths}
G_{I \subset J} \co  f_{I \subset J}^{*} \left( \Ubar_K | \Adams_{I \subset J}  \right)| \nu \Adams_{I \subset J} \to \Ubar_K | \nu \Adams_{I \subset J} .
\end{equation}

Let $F \co \Ubar_{K} \to Z$ be a map from $\Ubar_{K}$ to a topological space $Z$. Such a map is \emph{constant along the positive (respectively negative) end} if there is a map $  f \co [0,1] \to Z $ such that the restriction of $F$ to each fibre $  \Ubar_{\vr} $  agrees with $f$ near $s=+\infty$ on the last component  in Equation \eqref{eq:union_disjoint_finite_strips} (respectively near $-\infty$ on the first component).

\begin{defin} \label{def:obtained_by_gluing}
The map $F$ is \emph{obtained by gluing} if its restriction to neighbourhoods of all boundary strata yields a commutative diagram
\begin{equation}
\xymatrix{f_{I \subset J}^{*} \left(  \Ubar_K | \Adams_{I \subset J}  \right) \ar[r]^{G_{I \subset J}} \ar[d]^{f_{I \subset J}} &  \Ubar_K | \nu \Adams_{I \subset J} \ar[d]^{F} \\
\Ubar_K |  \Adams_{I \subset J} \ar[r]^{F} & Z.} 
\end{equation}
\end{defin}
Since the gluing map $G_{I \subset J}$ is surjective, $F$ is determined, on a neighbourhood of a boundary stratum, by its restriction to the stratum and the gluing map. In addition, continuity implies that its restriction to the boundary is constant along each glued end.

To achieve transversality of moduli spaces of maps, it is convenient to introduce a notion which is less rigid than the above gluing construction. To this end, we introduce the thick part of the fibre:
\begin{defin} \label{def:R-thick-strip}
If $R < r_i$ for all $i \in I$,  the \emph{$R$-thick} part of $\Ubar_{\vr}$ (relative to $I \subset K$) is the union:
\begin{enumerate}
\item for $i \in I$ of the strips:
\begin{equation} \label{eq:thick-part-in-I}
B_{0, R} \subset B_{r_{i_+}, r_i} \textrm{ and }  B_{R , 0 } \subset B_{r_i, r_{i_-}}
\end{equation}
where $i_- < i < i_+$ are successive elements of $K$.
\item for $j \notin I$ of the strips:
\begin{equation} \label{eq:thick-part-not-in-I}
B_{0, r_j} \subset B_{r_{j_+}, r_j} \textrm{ and }  B_{r_j, 0 } \subset B_{r_j, r_{j_-}}
\end{equation}
where $j_- < j < j_+$ are successive elements of $K$.
\end{enumerate}
\end{defin}

\begin{figure}
  \centering
  \begin{tikzpicture}

\draw[line width=4*\lw,dotted] (-12*\mw pt,-1*\mw pt) -- ( -13*\mw pt, -1*\mw pt);
\draw[line width=4*\lw,dotted] (-12*\mw pt,1*\mw pt) -- ( -13*\mw pt, 1*\mw pt);
\draw[line width=4*\lw,dotted] (12*\mw pt,-1*\mw pt) -- ( 13*\mw pt, -1*\mw pt);
\draw[line width=4*\lw,dotted] (12*\mw pt,1*\mw pt) -- ( 13*\mw pt, 1*\mw pt);
\draw[line width=4*\lw] (-12.5*\mw pt,-1*\mw pt) -- ( 12.5*\mw pt, -1*\mw pt);
\draw[line width=4*\lw] (-12.5*\mw pt,1*\mw pt) -- ( 12.5*\mw pt, 1*\mw pt);

\draw[line width=4*\lw] (-4*\mw pt,-2*\mw pt) -- ( 2*\mw pt, -2*\mw pt);
\node[below] at (-4*\mw pt, -2*\mw pt)  {$-r_2$};
\draw[dotted] (-4*\mw pt, 1*\mw pt) -- (-4*\mw pt, -1*\mw pt) ;
\draw[line width=4*\lw] (-4*\mw pt,-2*\mw pt -1/8*\mw pt) -- (-4*\mw pt,-2*\mw pt +1/8*\mw pt);
\node[below] at (2*\mw pt, -2*\mw pt)  {$r_1$};
\draw[line width=4*\lw] (1*\mw pt,-2*\mw pt -1/8*\mw pt) -- (1*\mw pt,-2*\mw pt +1/8*\mw pt);
\node[below] at (1*\mw pt, -2*\mw pt)  {$R$};
\draw[dotted] (2*\mw pt, 1*\mw pt) -- (2*\mw pt, -1*\mw pt) ;
\draw[line width=4*\lw] (2*\mw pt,-2*\mw pt -1/8*\mw pt) -- (2*\mw pt,-2*\mw pt +1/8*\mw pt);
\node[below] at (0*\mw pt, -2*\mw pt)  {$0$};
\draw[line width=4*\lw] (0*\mw pt,-2*\mw pt -1/8*\mw pt) -- (0*\mw pt,-2*\mw pt +1/8*\mw pt);
\fill (0*\mw pt, 0*\mw pt)  circle (8*\lw);

\draw[line width=4*\lw] (12.5*\mw pt,2*\mw pt) -- ( 2*\mw pt, 2*\mw pt);
\node[above] at (1.5*\mw pt, 2*\mw pt)  {$-r_1$};
\draw[line width=4*\lw] (2*\mw pt,2*\mw pt -1/8*\mw pt) -- (2*\mw pt,2*\mw pt +1/8*\mw pt);
\node[above] at (3*\mw pt, 2*\mw pt)  {$-R$};
\draw[line width=4*\lw] (3*\mw pt,2*\mw pt -1/8*\mw pt) -- (3*\mw pt,2*\mw pt +1/8*\mw pt);
\node[above] at (4*\mw pt, 2*\mw pt)  {$0$};
\draw[line width=4*\lw] (4*\mw pt,2*\mw pt -1/8*\mw pt) -- (4*\mw pt,2*\mw pt +1/8*\mw pt);
\node[above] at (5*\mw pt, 2*\mw pt)  {$R$};
\draw[line width=4*\lw] (5*\mw pt,2*\mw pt -1/8*\mw pt) -- (5*\mw pt,2*\mw pt +1/8*\mw pt);

\draw[line width=4*\lw,dotted] (12.5*\mw pt,2*\mw pt) -- ( 13*\mw pt, 2*\mw pt);
\fill (4*\mw pt, 0)  circle (8*\lw);

\draw[line width=4*\lw] (-12.5*\mw pt,2*\mw pt) -- ( -4*\mw pt, 2*\mw pt);
\draw[line width=4*\lw,dotted] (-12.5*\mw pt,2*\mw pt) -- ( -13*\mw pt, 2*\mw pt);
\node[above] at (-4*\mw pt, 2*\mw pt)  {$r_2$};
\draw[line width=4*\lw] (-4*\mw pt,2*\mw pt -1/8*\mw pt) -- (-4*\mw pt,2*\mw pt +1/8*\mw pt);
\node[above] at (-8*\mw pt, 2*\mw pt)  {$0$};
\draw[line width=4*\lw] (-8*\mw pt,2*\mw pt -1/8*\mw pt) -- (-8*\mw pt,2*\mw pt +1/8*\mw pt);
\node[above] at (-9*\mw pt, 2*\mw pt)  {$-R$};
\draw[line width=4*\lw] (-9*\mw pt,2*\mw pt -1/8*\mw pt) -- (-9*\mw pt,2*\mw pt +1/8*\mw pt);
\fill (-8*\mw pt, 0*\mw pt) circle (8*\lw);

\fill[gray, opacity=.25]  (-9*\mw pt,-1*\mw pt) rectangle (1*\mw pt,1*\mw pt);
\fill[gray, opacity=.25]  (3*\mw pt,-1*\mw pt) rectangle (5*\mw pt,1*\mw pt);

\node at (13*\mw pt, 0*\mw pt)  {$0$};
\node at (2*\mw pt, 0*\mw pt)  {$1$};
\node at (-4*\mw pt, 0*\mw pt)  {$2$};
\node at (-13*\mw pt, 0*\mw pt)  {$3$};

\end{tikzpicture}
  \caption{The $R$-thick part of $\Ubar_{0123}$ relative to $I = 013$.  }
  \label{fig:R-thickpart-0123}
\end{figure}

Let  $F_{I \subset J} \co  \Ubar_K | \nu \Adams_{I \subset J} \to Z$ be a function obtained by gluing, and assume $Z$ is a Fr\'echet manifold whose tangent space is equipped with a fixed collection of semi-norms.  We shall say that a section  of $F_{I \subset J}^* TZ  $ is consistent if the following properties hold: (i)  there is a constant $R$ so that the support is contained in the interior of the $R$-thick part and (ii) the seminorms of the sections at a point $\vr \in \nu \Adams_{I \subset J} $ are bounded by a constant multiple of $\sum_{i \in I \setminus \{\min K, \max K\} } e^{-r_i} $. A \emph{consistent perturbation} of $F_{I \subset J} $ is the image under exponentiation of a consistent section.

\begin{defin} \label{def:obtained_by_perturbed_gluing}
A map $F \co \Ubar_{K} \to Z$  is obtained by \emph{perturbed gluing} if its restriction to $ \nu \Adams_{I \subset J} $ for all pairs $I \subset J$ agrees with a consistent perturbation of a function obtained by gluing.
\end{defin}

In addition, we fix maps
\begin{align}
  \epsilon_{\min(K)} \co \Adams_K \times B_{+} & \to \Ubar_{K} \\  
\epsilon_{\max(K)} \co \Adams_K \times B_{-} & \to \Ubar_{K} 
\end{align}
whose restrictions to $\vr \in \Adams_{K}$ gives positive (resp. negative) strip like ends on the fibre $\Ubar_{\vr}$, which agree up to translation with the natural ones coming from the last (respectively first) factor in Equation \eqref{eq:union_disjoint_finite_strips}, and which are compatible with the gluing maps near every boundary stratum.

\subsection{Adams spaces with a distinguished marker}
\label{sec:adams-paths-prism}
Equip $K \times \{ -, + \}$  with the total ordering obtained by extending the ordering on $K$ via $
  (+, i ) < (-,j) $ for all $i, j \in K$. Write every subset of $ K \times \{ -, + \}$ as a union $
 K_- \times \{-\} \cup K_+ \times \{ + \}$. The constructions of the previous section, applied to this ordered set, yield a family of Riemann surfaces 
\begin{equation} \label{eq:diagram_paths_in_prism}
  \xymatrix{ \Ubar_{K_-,K_+} \ar[r] & \Adams_{K_-,K_+}, }
\end{equation}
and an open subset $\Adamsop_{K_-,K_+} \subset \Adams_{K_-,K_+}$ over which the fibres of $\Ubar_{K_-,K_+} $  are connected. There are inclusions of the boundaries of each fibre $ \Ubar^{\{ i\}}_{K_-,K_+}  \subset  \Ubar_{K_-,K_+}$
for $i=0,1$.

\begin{figure}
  \centering
  \begin{tikzpicture}

\draw[line width=4*\lw,dotted] (-12*\mw pt,-1*\mw pt) -- ( -13*\mw pt, -1*\mw pt);
\draw[line width=4*\lw,dotted] (-12*\mw pt,1*\mw pt) -- ( -13*\mw pt, 1*\mw pt);
\draw[line width=4*\lw,dotted] (12*\mw pt,-1*\mw pt) -- ( 13*\mw pt, -1*\mw pt);
\draw[line width=4*\lw,dotted] (12*\mw pt,1*\mw pt) -- ( 13*\mw pt, 1*\mw pt);
\draw[line width=4*\lw] (-12.5*\mw pt,-1*\mw pt) -- ( 12.5*\mw pt, -1*\mw pt);
\draw[line width=4*\lw] (-12.5*\mw pt,1*\mw pt) -- ( 12.5*\mw pt, 1*\mw pt);

\node at (-12*\mw pt, 0*\mw pt) {$k^1_-$};
\fill (-8*\mw pt, 0*\mw pt) circle (8*\lw);
\node at (-4*\mw pt, 0*\mw pt)  {$k^0_-$};;
\node[above] at (0*\mw pt, 1*\mw pt)  {$w$};
\draw[line width=4*\lw] (0*\mw pt,1*\mw pt -1/8*\mw pt) -- (0*\mw pt,1*\mw pt +1/8*\mw pt);
\node at (2*\mw pt, 0*\mw pt) {$k^2_+$} ;
\fill (4*\mw pt, 0)  circle (8*\lw);
\node at (7*\mw pt, 0*\mw pt) {$k^1_+$} ;
\fill (10*\mw pt, 0*\mw pt) circle (8*\lw);
\node at (12*\mw pt, 0*\mw pt) {$k^0_+$};



\end{tikzpicture}
  \caption{A fibre of $\Ubar_{K_-, K_+}$, with $K_+ =\{ k_+^0 < k_+^1 < k_+^2 \} $, and $K_- = \{ k_-^0 < k_-^1  \} $ .}
  \label{fig:fibre_adams_-+}
\end{figure}

The additional data of the decomposition distinguishes the finite strip $B_{r_{(-,\min K_-)},r_{(+,\max K_+)}}$  in each fibre of  $\Ubar_{K_-,K_+}$, and hence the corresponding component. We let $w$ denote the image of $(0,1)$ under this embedding; we represent elements of $ \Adams_{K_-,K_+} $ by drawing this marked point, and dropping the interior marked point which is the image of $(0,1/2)$ under the above embedding, as in Figure \ref{fig:fibre_adams_-+}.

To state the compatibility of the distinguished marker with the boundary decomposition of $ \Adams_{K_-,K_+}$, note that such a stratum is labelled by pairs $I_\pm$ and $J_\pm$ such that
\begin{align}
 \min K_+ \in I_+ \subset & J_+ \subset K_+ \\
 \max K_- \in I_- \subset & J_- \subset K_-.
\end{align}
Whenever $I_\pm$ are singletons (i.e. respectively consist only of $\min K_+$ and $\max K_-$), the corresponding boundary stratum is naturally identified with the moduli space $  \Adams_{J_-, J_+}$.

\begin{lem}
The codimension $1$ boundary strata of $  \Adams_{K_-,K_+} $  are:
\begin{align} \label{eq:boundary_adams_prism_0} 
  &   \coprod_{i \in K_- \setminus  \max K_-} \Adams_{K_- \setminus i  ,  K_+}  \\ \label{eq:boundary_adams_prism_1}
  &   \coprod_{i \in K_+ \setminus  \min K_+} \Adams_{K_-  ,   K_+  \setminus i}  \\  \label{eq:stop_K-U}
& \coprod_{i \in K_- \setminus  \max K_-} \Adams_{K^\geq_{-,i}} \times  \Adams_{K^\leq_{-,i}  ,  K_+} \\   \label{eq:boundary_adams_prism_4} 
& \coprod_{i \in K_+ \setminus  \min K_+} \Adams_{K_-  ,   K^\geq_{+,i}} \times  \Adams_{K^\leq_{+,i}}
\end{align} 
The restriction of $ \Ubar_{K_-,K_+}$ to these strata is naturally isomorphic to the union of pullbacks of the universal curves on each factor.

\qed
\end{lem}

Using the above description of the boundary strata, inductively choose families of positive (respectively negative) strip like ends
\begin{align} \label{eq:positive_end_moduli_prism}
\epsilon_{\pm} \co \Adams_{K_-,K_+} \times B_\pm & \to \Ubar_{K_-,K_+} 
\end{align}
whose restrictions to the boundary strata are compatible with the inductive choices and those made in Section \ref{sec:adams-univ-curve}.

Assume now that $\max K_+ \leq \min K_-$, with respect to the ordering on $K$. Let $z_\inp$ denote the boundary marked point $w$. For each  $\vr \in \Delta_{K_-,K_+}$  fix the following  positive strip like ends near $z_{\inp}$:
\begin{align}
\epsilon_{\inp} \co B_{+} & \to B \subset \Ubar_{\vr}  \\ \label{eq:positive_strip_like_end}
(s,t) & \mapsto \sqrt{-1} - e^{- s - t\sqrt{-1} \pi },
\end{align}
where the complex coordinates on $B$ are given by its embedding in $\bC$. By construction, these strip-like ends are compatible with gluing. Since $ \Ubar^{\{1\}}_{K_-,K_+} $ is naturally ordered via its identification with a union of real lines,  the  points preceding  or succeeding $z_{\inp}$ define subsets of the boundary:
\begin{align} \label{eq:subset_boundary_before_marked_point}
 \Ubar_{K_-,K_+}^{ z_{\inp}>} & = \{ z \in \Ubar^{\{1\}}_{K_-,K_+}  | z < z_{\inp} \} \\ \label{eq:subset_boundary_after_marked_point}
     \Ubar_{K_-,K_+}^{z_{\inp} <} & = \{ z \in \Ubar^{\{1\}}_{K_-,K_+} |  z_{\inp} < z  \}.
\end{align}

Assume instead that $\max K_- \leq \min K_+ $ with respect to the ordering on $K$, and repeat the same procedure to obtain a marked point $z_{\out}$ on each fibre of $\Ubar_{K_-,K_+}$.  Pick negative strip-like ends
\begin{align}
  \epsilon_{\out} \co B_{-} & \to B \subset \Ubar_{\vr} 
\end{align}
whenever $\vr \in \Delta_{K_-,K_+}$ which are compatible with gluing. The points preceding  or succeeding  $z_{\out}$ yield subsets $
     \Ubar_{K_-,K_+}^{ z_{\out}>}   \subset  \Ubar^{\{1\}}_{K_-,K_+} \supset    \Ubar_{K_-,K_+}^{z_{\out} <}$.

Given  $I \subset K$, introduce the notation
\begin{equation}
 K^\geq_I  \equiv  K^\geq_{\max I}  \textrm{ and }
  K^\leq_I  \equiv  K^\leq_{\min I}
\end{equation}
where the sets $K^\geq_i$ and $K^\leq_i$ are as in Equation \eqref{eq:smaller_larger_sets}. Given a nested pair $ I \subset K$, define the following subsets of $K \times \{ +, -\}$:
\begin{align} \label{eq:K_I_inp}
K^{\inp}_I & \equiv (K^\geq_I\times \{ -\} , K^\leq_I \times \{ +\}) \\ \label{eq:K_I_out}
K^{\out}_I &  \equiv (K^\leq_I\times \{ -\}, K^\geq_I  \times \{ + \}).
\end{align}

\begin{lem}
The minimal element of $K^\inp_I$ is $(\min K,+)$, and the maximal element is $(\max K, -)$. The minimal element of $K^\out_I$ is $(\max I, +)$ and the maximal element is $(\min I, -)$. \qed
\end{lem}

\subsection{Strips with one input marked point} \label{sec:strips-with-one-positive}
Let $\vK$ be a nested sequence of totally ordered sets whose maximal element is $K$ and minimal element is $K_0$.

\begin{defin}
 The \emph{Adams moduli space with one input} $ \Adams_{\vK;\inp}$ is the product
 \begin{equation}
  \Adams_{\vK;\inp}  \equiv [0,\infty]^{\vK \setminus K}. 
 \end{equation}
\end{defin}

The cells of $   \Adams_{\vK;\inp} $ are given by pairs of subsets $  K \in \vI \subset \vJ \subset \vK$. Write $\Adams_{\vI \subset \vJ; \inp  } $ for the corresponding stratum.  Define the open subset $ \Adamsop_{\vK;\inp}  \subset  \Adams_{\vK;\inp}   $ corresponding to the inclusion $[0,\infty) \subset [0,\infty]$. It is the union of the strata for which $\{ K  \}= \vI$.

We shall build a universal curve on $ \Adams_{\vK;\inp}  $ by pulling back the universal curve on $\Adams_{K^{\inp}_{K_0}}$ as follows: let $\min \vK \subset K$ denote the set of minimal elements of subsets of $K$ which lie in $\vK$, and $\max \vK$ the set of maximal elements. Assign to a sequence $\vJ \subset \vK$ a subset of $ K^{\inp}_{K_0}$
\begin{equation}
\mu_\inp \vJ  \equiv \max \vJ  \times \{- \} \coprod \min \vJ \times \{ + \}.
\end{equation}

There is a natural map
\begin{equation} \label{eq:barycentric_to_inp}
\mu_\inp:    \Adams_{\vK;\inp} \to \Adams_{\mu_\inp \vK } \subset \Adams_{K^\inp_{K_0}}
\end{equation}
inducing the map of posets $  \vI \subset \vJ \mapsto \mu_\inp \vI \subset \mu_\inp \vJ$. In coordinates, we simply set the coordinate of $ \mu_\inp(\vr)$ labelled by $(j,+)$ to equal  $\sum_{ \max J = j} r_{j}  $, while the coordinate  labelled by $(j,-)$ is given by  $\sum_{ \min J = j} r_{j}$, with both sums taken over $J \in \vK$.

\begin{defin}
The \emph{universal curve over $\Adams_{\vK; \inp}$} is the projection map  
\begin{equation}
  \Ubar_{\vK;\inp} \equiv \mu^{*}_\inp( \Ubar_{K^\inp_{K_0}} ) \to \Adams_{\vK, \inp}.
\end{equation}
\end{defin}
\begin{figure}
\centering
\begin{tikzpicture}
\coordinate (0) at (0,0);
\coordinate (1) at (4,0);
\coordinate (2) at (4,4);
\coordinate (3) at (0,4);

\foreach \x in {0,...,3} \fill (\x) circle (8*\lw);

\draw[line width=2*\lw] (0)--(1)--(2)--(3)--cycle;

\ExtractCoordinate{$(0)$};
\CtoEtworight{\XCoord+4.5*\mw}{\YCoord +4*\mw}{2}{0};
\CtoEoneright{\XCoord-.5*\mw}{\YCoord +4*\mw}{2}{0};
\CtoE{\XCoord-.5*\mw}{\YCoord - 0.5*\mw}{2}{0};  

\ExtractCoordinate{$(0)!.5!(1)$};
\CtoEoneright{\XCoord+.5*\mw}{\YCoord - 0.5*\mw}{2}{0};

\ExtractCoordinate{$(1)$};
\CtoE{\XCoord+1.5*\mw}{\YCoord -.5*\mw}{2}{2}; 
\Conin{\XCoord+2.75*\mw}{\YCoord -.5*\mw}{0};
\CtoE{\XCoord+1.5*\mw}{\YCoord +4*\mw}{0}{2}; 
\Conintwo{\XCoord+2.75*\mw}{\YCoord +4*\mw}{0};

\ExtractCoordinate{$(2)$};
\CtoE{\XCoord+1.5*\mw}{\YCoord+0.5*\mw}{2}{2};
\Conin{\XCoord+2.75*\mw}{\YCoord+0.5*\mw}{1}; 
\Conin{\XCoord+4*\mw}{\YCoord+0.5*\mw}{0}; 

\ExtractCoordinate{$(2)!.5!(3)$};
\CtoEoneright{\XCoord+.5*\mw}{\YCoord+ 0.5*\mw}{2}{1};
\Conin{\XCoord+1.75*\mw}{\YCoord+0.5*\mw}{0}; 

\ExtractCoordinate{$(3)$};
\CtoE{\XCoord-1.75*\mw}{\YCoord+ 0.5*\mw}{2}{1};
\Conin{\XCoord-.5*\mw}{\YCoord+0.5*\mw}{0}; 
\end{tikzpicture}

\caption{The moduli space $\Adams_{\vK;\inp}$, with vertices labelled by fibres of $\Ubar_{\vK;\inp}$ ($\vK = 2\subset 12 \subset 012$).}
\label{fig:CtoE012}
\end{figure}

We give some examples of these universal curves for $K= 012$.

\begin{example}
If $\vK = \{ 2 \subset 12 \subset 012 \}$, then $  \mu_\inp \vK = \{ (0,+)  < (1,+) < (2,+) < (2,-) \} $. In particular, $\Adams_{\vK;\inp}$ and $\Adams_{\mu_\inp \vK }$ are both $2$-dimensional. In order to describe the image of the vertices, note that they canonically correspond to subsets of $\vK$ containing the maximal element (because $\vI = \vJ$ if and only if the corresponding stratum is $0$-dimensional). Using this convention to simplify the notation, we can write the map  $\mu_\inp$ on vertices as:
\begin{align}
  012  & \mapsto \{ (0,+),(2,-) \}  \\
  2 \subset 012 & \mapsto  \{  (0,+), (2,+), (2,-) \} \\
12 \subset 012 & \mapsto  \{  (0,+), (1,+), (2,-) \} \\
2 \subset 12 \subset 012 & \mapsto  \{  (0,+),(1,+) (2,+), (2,-) \}. 
\end{align}
It is therefore easy to see that $\mu_\inp$ is an isomorphism. The fibres of $\Ubar_{\vK;\inp}  $ are shown in Figure \ref{fig:CtoE012}.
\end{example}
It is easy to find a map $\mu_\inp$ which cannot be injective because the dimension of the source is larger than that of the target:
\begin{example}
 If $ \vK = \{ 0 \subset 02 \subset 012 \} $, then $  \mu_\inp \vK = \{ (0,+) < (0,-)  < (2,-) \} $, so $\Adams_{\vK;\inp}$ is $2$-dimensional while $\Adams_{\mu_\inp \vK }$ has dimension $1$.
\end{example}

There are more interesting examples of maps $\mu_\inp$, which are neither injective nor surjective, despite the source and target having the same dimension:
\begin{figure}
\centering
\begin{tikzpicture}
\coordinate (0) at (0,0);
\coordinate (1) at (4,0);
\coordinate (2) at (4,4);
\coordinate (3) at (0,4);

\foreach \x in {0,...,3} \fill (\x) circle (8*\lw);

\draw[line width=2*\lw] (0)--(1)--(2)--(3)--cycle;

\ExtractCoordinate{$(0)$};
\CtoEtwosplit{\XCoord+4.5*\mw}{\YCoord +4*\mw}{2}{0};
\CtoEoneleft{\XCoord-.5*\mw}{\YCoord +4*\mw}{2}{0};
\CtoE{\XCoord-.5*\mw}{\YCoord - 0.5*\mw}{2}{0};  

\ExtractCoordinate{$(0)!.5!(1)$};
\CtoEtwosplit{\XCoord+.5*\mw}{\YCoord - 0.5*\mw}{2}{0};

\ExtractCoordinate{$(1)$};
\Conout{\XCoord+1.5*\mw}{\YCoord -.5*\mw}{2};
\CtoE{\XCoord+2.75*\mw}{\YCoord -.5*\mw}{1}{1}; 
\Conin{\XCoord+4*\mw}{\YCoord -.5*\mw}{0};

\Conout{\XCoord+1.5*\mw}{\YCoord +4*\mw}{2};
\CtoE{\XCoord+2.75*\mw}{\YCoord +4*\mw}{1}{1}; 
\Conin{\XCoord+4*\mw}{\YCoord +4*\mw}{0};
\ExtractCoordinate{$(2)$};
\Conout{\XCoord+1.5*\mw}{\YCoord +.5*\mw}{2};
\CtoE{\XCoord+2.75*\mw}{\YCoord +.5*\mw}{1}{1}; 
\Conin{\XCoord+4*\mw}{\YCoord +.5*\mw}{0};

\ExtractCoordinate{$(2)!.5!(3)$};
\CtoEoneright{\XCoord+1.75*\mw}{\YCoord+ 0.5*\mw}{1}{0};
\Conout{\XCoord+.5*\mw}{\YCoord+0.5*\mw}{2}; 

\ExtractCoordinate{$(3)$};
\Conout{\XCoord-1.75*\mw}{\YCoord+0.5*\mw}{2};
\CtoE{\XCoord-.5*\mw}{\YCoord+ 0.5*\mw}{1}{0};
\end{tikzpicture}

\caption{The moduli space $\Adams_{\vK;\inp}$, with vertices labelled by fibres of $\Ubar_{\vK;\inp}$ ($\vK = 1 \subset 01 \subset 012$).}
\label{fig:CtoE012-inj-not-equal}
\end{figure}

\begin{example}
If $ \vK = \{ 1 \subset 01 \subset 012 \} $, then $  \mu_\inp \vK = \{ (0,+) < (1,+) < (1,-) < (2,-) \} $, so $\Adams_{\vK;\inp}$ and $\Adams_{\mu_\inp \vK }$ both again have dimension $2$. On vertices, we have
\begin{align}
  012  & \mapsto \{ (0,+),(2,-) \}  \\
  1 \subset 012 & \mapsto  \{  (0,+), (1,+), (1,-), (2,-) \} \\
01 \subset 012 & \mapsto  \{  (0,+), (1,-), (2,-) \} \\
1 \subset 01 \subset 012 & \mapsto  \{  (0,+),(1,+), (1,-), (2,-) \}. 
\end{align}
Since the second and the last vertex above have the same image in $\Adams_{\mu_\inp \vK }$, this map cannot be injective. In fact, we can easily compute that it is given, in coordinates, by
\begin{equation}
 (r_1, r_{01}) \mapsto (r_1, r_1 + r_{01}),
\end{equation}
so the map is in fact injective on $\Adamsop_{\vK;\inp}$, but not on the closure. 

\end{example}

\subsubsection{Maps induced by minimal and maximal elements}
\label{sec:maps-induces-minimal}

Before proceeding with a more convenient combinatorial description of the  boundary strata of $\Adams_{\vK, \inp} $, we introduce maps associated to minimal and maximal elements: let $\vJ$ be a sequence of nested subsets with maximal element $J$ that is ordered, and minimal element $I$. Since the minimal (respectively maximal) element of $J' \in \vJ[\geq]_I $ lies between the minimal (resp. maximal) elements of $I$ and $J$, we have maps:
\begin{align}\label{eq:projection_min-0}
  \min \co \Adams_{ \vJ} & \to \Adams_{J^{\leq}_{I}} \\ \label{eq:projection_max-0}
    \max \co \Adams_{ \vJ} & \to \Adams_{J^{\geq}_{I}}.
\end{align}
Given $i \in  J^{\leq}_{I} \setminus \{ \min I, \min J \}$ the $i$ coordinate of $\min (\vr) $ (resp. $\max (\vr)$)  is the sum of the coordinates $r_{J'}$  for elements $ J' \in \vJ \setminus \{J,I\}$ whose minimum  equals $i$ (if there are no such elements, the corresponding coordinate vanishes).   In the same way, given $i \in J^{\geq}_{I} \setminus \{ \max I, \max J \}$, the $i$-coordinate of  $\max (\vr)$  is the sum of the coordinates $r_{J'}$  for elements $ J' \in \vJ \setminus \{J,I\}$ whose maximum  equals $i$.  On the associated partially ordered sets,  the maps are simply given by
\begin{align}
  \vJ_1 \subset \vJ_2 & \mapsto \min \vJ_1 \subset \min \vJ_2 \\
  \vJ_1 \subset \vJ_2 & \mapsto \max \vJ_1 \subset \max \vJ_2.
\end{align}
\begin{table}
  \centering
  \begin{tabular}{|c|c|c|c|c|c|}
\hline
$\vJ$  &  $\dim  \Adams_{J^{\geq}_{I}}$  & $\dim  \Adams_{J^{\leq}_{I}}$ & Injective & Surjective \\
\hline $01  \subset 012 \subset 0123 \subset 01234$  &  $2$ & $0$ & Yes & Yes \\
$34  \subset 234 \subset 1234 \subset 01234$  &  $0$ & $2$ & Yes & Yes \\
$34  \subset 134 \subset 1234 \subset 01234$&  $0$ & $1$ & No & Yes \\
$3  \subset 134 \subset 1234 \subset 01234$&  $0$ & $2$ & No & No \\
$2 \subset 12 \subset 123 \subset 01234$ & $1$ & $1$ & No & No \\ 
$2 \subset 124 \subset 1234 \subset 01234$ & $1$ & $1$ & No & No \\ 
\hline
  \end{tabular}
\caption{Properties of the map $\min \times \max$ on $\Adams_{\vJ}$ in some examples.}
 \label{tab:min_max_properties}
\end{table}
Some examples of sequences $\vJ$ of length $4$, i.e. such that the corresponding moduli spaces $ \Adams_{ \vJ}  $ have dimension $2$ are given in Table  \ref{tab:min_max_properties}. We describe two of the cases in more detail:
\begin{example}
  Consider the sequence $\vJ = \{ 2 \subset 12 \subset 123 \subset 01234 \}$.  The moduli space $ \Adams_{ \vJ}  $ is $2$-dimensional, with coordinates given by $r_{12}$ and $r_{123}$, and both  $ \Adams_{J^{\leq}_{I}}$ and $ \Adams_{J^{\geq}_{I}} $ are $1$-dimensional, with coordinates $r_1$ and $r_3$. In these coordinates, the map $(\max, \min) $ is $( r_{123}, r_{12} + r_{123})$. This map is injective on $ \Adamsop_{ \vJ} $, but not on the closure.  
\end{example}
\begin{example}
Consider the sequence $\vJ = \{ 2 \subset 124 \subset 1234 \subset 01234 \}$.  The moduli space $ \Adams_{ \vJ} $ has coordinates given by $r_{124}$ and $r_{1234}$, while the targets of $\min$ and $\max$ have coordinates $r_1$ and $r_3$. In these coordinates, the map $(\max, \min) $ factors through the inclusion of the diagonal in the square $ \Adams_{J^{\geq}_{I}} \times \Adams_{J^{\leq}_{I}} $.
\end{example}

With the above examples in mind, we characterise the sequences $\vJ$ of maximal length for which the map $(\max, \min)$ is injective:

\begin{lem} \label{lem:min_max_maps_deg}
Assume that  $\vJ$ consists of $|J| - |I| + 1$ elements. We have
\begin{equation}
    \dim \Adams_{\vJ} \geq \dim \left( \Adams_{J^\geq_{I}} \times  \Adams_{J^\leq_{I}} \right)
\end{equation}
with equality if and only if one of the following three conditions holds: (i)  $\max I = \max J$ and $\min$ is injective on $\vJ $ (ii)  $\min I = \min J$ and $\max$ is  injective on $\vJ$ or (iii)  $I = J \setminus \{ i\}$, for $i \neq \min J, \max J$.

Moreover, if equality holds, the map $(\min, \max)$ is an isomorphism.
\end{lem}
\begin{proof}
By assumption, successive subsets of $J$ appearing in $\vJ $ differ by exactly one element, hence
\begin{equation} \label{eq:compare_number_of_elements}
|J| - |I| + 2 \geq  |J^\leq_{I}| + | J^\geq_{I} |,
\end{equation}
with $I$ contributing two elements to the union of $J^\leq_{I} $ and $  J^\geq_{I}  $, and each subsequent element of $\vJ$ at most one. Note that this inequality is strict if and only if each subset of $J$ appearing in  $\vJ$ is obtained  by adding an element  which is either larger than the maximum of the preceding subset  or smaller than the minimum.

This inequality allows us the compare the dimensions of the source and target in Equations \eqref{eq:projection_min-0} and  \eqref{eq:projection_max-0}.  We begin by noting that the conventions fixed in Remark \ref{rem:convention-nors}  imply that the dimensions of the moduli spaces we are considering are
\begin{align}
\dim \Adams_{\vJ}  & = |J| - |I| - 1 \\
  \dim \Adams_{J^\geq_{I}} & = \max( |J^\geq_{I}| -2, 0) \\
  \dim \Adams_{J^\leq_{I}} & = \max( |J^\leq_{I}| -2, 0) .
\end{align}
We now re-write Equation \eqref{eq:compare_number_of_elements} as:
\begin{equation} \label{eq:compare_number_of_elements_dim}
\dim \Adams_{\vJ}   \geq   |J^\leq_{I}| + | J^\geq_{I} | - 3,
\end{equation}
from which we conclude that equality  of dimensions can only hold if one of $J^\geq_{I}$ or $J^\leq_{I}$ is a singleton. We now consider the three cases in reverse order:

{\bf Case (iii):} If both $J^\geq_{I}$ or $J^\leq_{I}$ are singletons, then the dimension of the corresponding moduli spaces are both $0$, so equality of dimension requires that $\Adams_{\vJ}$ be $0$-dimensional, hence a point, which corresponds to $I = J \setminus \{ i \}$. Surjectivity is obvious.

{\bf Case (ii):} If $J^\leq_{I}$ is a singleton, but $J^\geq_{I}$ is not, then Equation \eqref{eq:compare_number_of_elements_dim} becomes:
\begin{equation}
  \dim \Adams_{\vJ}   \geq   \dim \Adams_{ J^\geq_{I}}
\end{equation}
with equality holding whenever the inequality in Equation \eqref{eq:compare_number_of_elements} is strict. This corresponds to $\max$ being injective, so that surjectivity follows.

{\bf Case (i):} Entirely analogous to the previous case.
\end{proof}

\subsubsection{Stratification of the boundary of $\Adams_{\vK, \inp}$ }
\label{sec:strat-bound-adams_vk}
Given a sequence $\vJ$ of nested subsets of an ordered set $K$, with maximal element $J$ and minimal element $J_0$, and an  element $I \in \vJ$, we first introduce the notation $\vJ[\geq]_I$  and  $\vJ[\leq]_I$ as before for the nested collection of sets preceding and succeeding $I$.  By the construction of the previous section, we have maps:
\begin{align}\label{eq:projection_min}
  \min \co \Adams_{ \vJ[\geq]_{I}} & \to \Adams_{J^{\leq}_{I}} \\ \label{eq:projection_max}
    \max \co \Adams_{ \vJ[\geq]_{I} } & \to \Adams_{J^{\geq}_{I}}.
\end{align}
We now have the necessary notation to describe the restriction of the universal curve to the boundary strata of $\Adams_{\vK, \inp} $: denoting by $I_0$ the minimal element of $\vI$, the boundary stratum  of $ \Adams_{\vK; \inp}  $ labelled by  $K \in \vI \subset \vJ$ is
\begin{equation} \label{eq:decomposition_boundary_adams_vK}
\Adams_{\vI \subset  \vJ; \inp}  \cong \Adams_{\vI \subset  \vJ[\geq]_{I_0}}  \times  \Adams_{\vJ[\leq]_{I_0};\inp},
\end{equation}
where $\Adams_{\vI \subset  \vJ[\geq]_{I_0}} $  is the cube on $\vJ[\geq]_{I_0} \setminus  \vI$. When  $\vI = \{ I\}$ is a singleton, the pair $\{ I \} \subset  \vJ[\geq]_{I} $ labels the top dimensional stratum of $ \Adams_{\vJ[\geq]_{I}} $.  Figure \ref{fig:continuation_family_homotopy} illustrates the fibres of the universal curve over this stratum; we shall presently explain how to describe the components of this restricted universal curve as pullbacks, after introducing the relevant maps.
\begin{figure}
\centering
\begin{tikzpicture}
\begin{scope}                          
\ConoutBig{-6*\mw pt}{0}{\max K};
\CtoEBig{0}{0}{\max I}{\min I};
\ConinBig{6*\mw pt}{0}{\min K};
\end{scope}
\end{tikzpicture}

\caption{A fibre of the universal curve over  $ \Adams_{\vK;\inp}$ restricted to $ \Adams_{\vJ[\leq]_I;\inp}  \times \Adams_{\vJ[\geq]_I}$. The curve in the middle is $\Ubar_{\vJ[\leq]_I;\inp}$, while the curves on the left and the right are respectively the pullback of the universal curves over $ \Adams_{\min \vJ[\geq]_I}$  and $\Adams_{\max \vJ[\geq]_I}$.}
\label{fig:continuation_family_homotopy}
\end{figure}

On the first factor of the right hand side in Equation \eqref{eq:decomposition_boundary_adams_vK}, consider the product map
\begin{equation} \label{eq:boundary_stratum_inp_map_0}
\xymatrix{  \Adams_{\vI \subset  \vJ[\geq]_{I_0}}   \ar[rr]^-{(\min, \max)} & & \Adams_{\min \vI \subset  \min \vJ[\geq]_{I_0}} \times  \Adams_{\max \vI \subset  \max \vJ[\geq]_{I_0}} \subset  \Adams_{K^{\leq}_{I_0}} \times \Adams_{K^{\geq}_{I_0}} ,}
\end{equation}
using the fact that $K \in \vJ$ to derive the second inclusion.  On the other hand, the second factor in Equation \eqref{eq:decomposition_boundary_adams_vK} maps by $\mu_{\inp}$ to
\begin{equation}\label{eq:boundary_stratum_inp_map_2}
  \Adams_{\max \vJ[\leq]_{I_0}, \min \vJ[\leq]_{I_0}} \subset \Adams_{{I_0}^{\geq}_{J_0},{I_0}^{\leq}_{J_0} } \equiv \Adams_{{I_0}^{\inp}_{J_0}}.
\end{equation}

The product of the right hand sides in Equations \eqref{eq:boundary_stratum_inp_map_0} and \eqref{eq:boundary_stratum_inp_map_2} is a stratum of $\Adams_{K^\inp_{K_0}}$, and these map fit in a commutative diagram:
\begin{equation}
  \label{eq:inp_moduli_project_boundary_compatible}
  \xymatrix{  \Adams_{\vI \subset  \vJ[\geq]_{I_0}}  \times  \Adams_{\vJ[\leq]_{I_0};\inp}  \ar[r]^{\cong} \ar[d] & \Adams_{\vI \subset \vJ; \inp} \ar[r]
& \Adams_{\vK; \inp} \ar[d] \\
 \Adams_{\min \vI \subset  \min \vJ[\geq]_{I_0}} \times  \Adams_{\max \vI \subset  \max \vJ[\geq]_{I_0}}  \times  \Adams_{{I_0}^{\inp}_{J_0}} \ar[r] &  \Adams_{K^{\geq}_{I_0}} \times \Adams_{{I_0}^{\inp}_{J_0}} \times \Adams_{K^{\leq}_{I_0}} \ar[r]  & \Adams_{K^\inp_{K_0}}. }
\end{equation}

We now specialise the above discussion to the codimension $1$ strata of $\Adams_{\vK;\inp} $. These come in two types, both labelled by an element $I \in \vK \setminus \{ K\}$. The  first corresponds to  $\vI = \{ K \} $ and $\vJ = \vK \setminus \{ I \}$, and in this case Diagram \eqref{eq:inp_moduli_project_boundary_compatible} reduces to
 \begin{equation}
  \xymatrix{ \Adams_{\vJ ;\inp}  \cong  \Adams_{\{ K\}  \subset \vK \setminus \{ I\}; \inp} \ar[r] \ar[d]
& \Adams_{\vK; \inp} \ar[d] \\
  \Adams_{{K}^{\inp}_{J_0}} \ar[r]  & \Adams_{K^\inp_{K_0}}. }
\end{equation} 

The second case corresponds to $\vI = \{ I \subset K \}$ and $\vJ = \vK$, and in this case, we find that the projection $  \Adams_{\{ I \subset K\}  \subset \vK; \inp}  \to  \Adams_{K^\inp_{K_0}}$ factors through the map
\begin{equation} \label{eq:projection_min_max}
  \Adams_{\vK[\geq]_{I}}  \to  \Adams_{\min \vK[\geq]_{I}} \times  \Adams_{\max \vK[\geq]_{I}},
\end{equation}
which is derived from Equation \eqref{eq:boundary_stratum_inp_map_0} by using the fact that $\{ I,K\} \subset \vK[\geq]_{I}$ labels the top dimensional stratum of $ \Adams_{\vK[\geq]_{I}}$.

Using Equations \eqref{eq:boundary_adams_prism_0}-\eqref{eq:boundary_adams_prism_4} we obtain an explicit description of the fibres over boundary strata of $\Adams_{\vK; \inp} $ (see Figure \ref{fig:continuation_family_homotopy} for an example):
\begin{lem}
The boundary of  $ \Adams_{\vK;\inp}$ is covered by the following codimension $1$ strata:
\begin{align} \label{eq:inp_adams_skip}
&  \bigcup_{I \in \vK \setminus\{K\}}  \Adams_{\vK \setminus\{I\};\inp} \\ \label{eq:inp_adams_stop}
& \bigcup_{I \in \vK \setminus\{K\}}  \Adams_{\vK[\geq]_I}  \times \Adams_{\vK[\leq]_I;\inp} .
\end{align} 
The restriction of $\Ubar_{\vK;\inp} $ to the first stratum is naturally isomorphic to $\Ubar_{\vK \setminus \{I\};\inp} $, and the restriction to the second is given by
\begin{align} \label{eq:inp_adams_stop_U}
&   \left(  \mathrm{max}^* \Ubar_{\max \vK[\geq]_I} \coprod \mathrm{min}^* \Ubar_{\min \vK[\geq]_I} \right)  \times \Adams_{\vK[\leq]_I;\inp} \coprod  \Adams_{\vK[\geq]_I}  \times \Ubar_{\vK[\leq]_I;\inp} .
\end{align}
\end{lem}
\begin{proof}
Equation \eqref{eq:inp_adams_skip} corresponds the case when the coordinate labelled by $I$ vanishes, and the other type of boundary stratum to the case this coordinate is $\infty$. In Equation \eqref{eq:inp_adams_stop_U}, the projection from $K^\inp_{K_0}$ to $K$ can be used in order to identify $\max \vK[\geq]_I \times \{- \}$, as a subset of $K^\inp_{K_0}$ with $ \max \vK[\geq]_I$, and similarly for $ \min \vK[\geq]_I$.
\end{proof}

\subsection{Strips with one output marked point} \label{sec:strips-with-one}

The construction of moduli spaces with outputs is entirely analogous to that of  moduli spaces with inputs, replacing the partially ordered set $K^\inp_{K_0}$ by $K^\out_{K_0} $ in the construction, and a few other minor changes. 
\begin{figure}
\centering
\begin{tikzpicture}
\coordinate (0) at (0,0);
\coordinate (1) at (4,0);
\coordinate (2) at (4,4);
\coordinate (3) at (0,4);

\foreach \x in {0,...,3} \fill (\x) circle (8*\lw);

\draw[line width=2*\lw] (0)--(1)--(2)--(3)--cycle;

\ExtractCoordinate{$(0)$};
\EtoCtworight{\XCoord+4.5*\mw}{\YCoord +4*\mw}{0}{0};
\EtoConeright{\XCoord - 0.5*\mw}{\YCoord +4*\mw}{0}{0};
\EtoC{\XCoord-.5*\mw}{\YCoord - 0.5*\mw}{0}{0};  
\ExtractCoordinate{$(1)$};
\EtoC{\XCoord+1.5*\mw}{\YCoord -.5*\mw}{0}{2}; 
\Conin{\XCoord+2.75*\mw}{\YCoord -.5*\mw}{0};
\EtoC{\XCoord+1.5*\mw}{\YCoord +4*\mw}{0}{2}; 
\Conintwo{\XCoord+2.75*\mw}{\YCoord +4*\mw}{0};
\ExtractCoordinate{$(2)$};
\EtoC{\XCoord+1.5*\mw}{\YCoord+0.5*\mw}{0}{2};
\Conin{\XCoord+2.75*\mw}{\YCoord+0.5*\mw}{1}; 
\Conin{\XCoord+4*\mw}{\YCoord+0.5*\mw}{0}; 
\ExtractCoordinate{$(2)!.5!(3)$};
\EtoConeright{\XCoord-.75*\mw}{\YCoord+0.5*\mw}{0}{1};
\Conin{\XCoord+.5*\mw}{\YCoord+0.5*\mw}{0}; 
\ExtractCoordinate{$(3)$};
\EtoC{\XCoord-1.75*\mw}{\YCoord+0.5*\mw}{0}{1};
\Conin{\XCoord-.5*\mw}{\YCoord+0.5*\mw}{0}; 

\ExtractCoordinate{$(0)!.5!(1)$};
\EtoConeright{\XCoord+.5*\mw}{\YCoord - 0.5*\mw}{0}{0};
\end{tikzpicture}

\caption{The moduli space $\Adams_{\vK;\out}$, with vertices labelled by fibres of $\Ubar_{\vK;\out}$ ($\vK = 0\subset 01 \subset 012$).}
\label{fig:EtoC012}
\end{figure}

\begin{defin}
 The \emph{Adams moduli space with one output}, $ \Adams_{\vK;\out}$ is the product
 \begin{equation}
  \Adams_{\vK;\out}  \equiv [0,\infty]^{\vK \setminus K_0 }. 
 \end{equation}
\end{defin}

The cells of $   \Adams_{\vK;\out} $ are given by pairs of subsets $
  K_0 \in \vI \subset \vJ \subset \vK$,
and the union of those strata for which $\vI = K_0$ correspond to the open subset $ \Adamsop_{\vK;\out}  $.

There is a  natural map
\begin{equation} \label{eq:barycentric_to_out}
\mu_\out \co   \Adams_{\vK;\out} \to \Adams_{K^\out_{K_0}} 
\end{equation}
which assigns to  a sequence $\vJ \in \vK$  the subset $
 \min \vJ \times \{ - \} \coprod \max \vJ  \times \{+ \}$ of $K^\out_{K_0} $.
\begin{defin}
The \emph{universal curve over $\Adams_{\vK; \out}$} is the projection map  \begin{equation}
  \Ubar_{\vK;\out} \equiv \mu^{*}_{\out}( \Ubar_{K^\out_{K_0}}) \to \Adams_{\vK, \out}.
\end{equation}
\end{defin}

By pullback, Equation \eqref{eq:positive_end_moduli_prism} determines families of strip like ends $\epsilon_-$ and $\epsilon_+$, and a family of distinguished components which are trivialised by a map we still denote $\iota$. 

As in Equation \eqref{eq:inp_moduli_project_boundary_compatible},  the boundary strata of $\Adams_{\vK; \out}  $ and their images under $\mu_\out$ fit in a commutative diagram:
\begin{equation} \label{eq:out_moduli_project_boundary_compatible}
  \xymatrix{ \Adams_{\vI \subset \vJ; \out} \cong   \Adams_{\vJ[\geq]_I;\out}  \times \Adams_{\vI \subset  \vJ[\leq]_I}  \ar[r] \ar[d] & 
 \Adams_{\vK; \out} \ar[d] \\
 \Adams_{I^{\leq}_{K_0}} \times \Adams_{{J}^{\out}_{I}} \times \Adams_{I^{\geq}_{K_0}} \ar[r]  & \Adams_{K^\out_{K_0}} .}
\end{equation}

The boundary of $ \Adams_{\vK;\out}$ admits a natural decomposition as a union of codimension $1$ strata
\begin{align} \label{eq:out_adams_skip}
&  \bigcup_{I \in \vK \setminus {K_0}}  \Adams_{\vK \setminus{ I};\out} \\ \label{eq:out_adams_stop}
& \bigcup_{I \in \vK \setminus {K_0}} \Adams_{\vK[\geq]_I;\out}  \times \Adams_{\vK[\leq]_I}.
\end{align}
The fibres of $\Ubar_{\vK;\out} $ are given by $ \Ubar_{\vK \setminus{I};\out} $ over the first type of stratum, and
\begin{align} \label{eq:out_adams_skip_U}
& \Adams_{\vK[\geq]_I;\out}  \times \left(  \mathrm{max}^* \Ubar_{\max \vK[\leq]_I} \coprod \mathrm{min}^* \Ubar_{\min \vK[\leq]_I} \right) \coprod \Ubar_{\vK[\geq]_I;\out}  \times \Adams_{\vK[\leq]_I}
\end{align}
over the second kind.

\section{Lagrangian Floer theory} \label{sec:floer-theory}
We now return to the setting of Sections \ref{sec:areas-strips-flux} and \ref{sec:floer-equation}. In particular, we recall that  $\scrJ$ denotes the space of tame almost complex structures, and that all Lagrangians $L$ that we consider are equipped with an almost complex structure $J_L$ for which they bound no holomorphic disc, and are graded with respect to the quadratic complex volume form on $X$ induced by a density on $\Q$. In particular, whenever $L$ and $L'$ are both graded Lagrangians, and $x \in L \cap L'$ is a transverse intersection point, there is a well-defined Maslov index $
\deg(x) \in \bZ  $ as explained in \cite [Section (12b)]{seidel-Book}.

If $L$ and $L'$ are graded Lagrangians which are transverse, and which both satisfy Condition \eqref{eq:everything_is_good_with_curves}, pick a family $J_{t}$ of almost complex structures such that $J_{0} = J_L $ and $J_{1} = J_{L'} $, and which is constant in a neighbourhood of the point $t= 1/2$. To each pair $x, y \in L \cap L'$, there corresponds a moduli space $\Mbar(x,y) $ of holomorphic strips with boundary conditions given by $L$ along $\bR \times \{ 0 \}$ and $L'$ along $\bR \times \{ 1 \}$.  Equation \eqref{eq:dimension_moduli_strips} holds in this case as does the analogue of the decomposition of the boundary of the moduli space given by  Equation \eqref{eq:boundary_Floer_moduli_space}.

\subsection{Continuation maps} \label{sec:continuation-maps}

Let $K$ be a totally ordered set. 
\begin{defin} \label{def:compatibly_fam_continuation}
  A \emph{consistent family of continuation data} parametrised by $  \Ubar_{K} $ is a map
  \begin{equation}
    \Phi_K = (\phi_K, J_K, \psi_K) \co    \Ubar_{K} \to \scrH \times \scrJ \times \scrD
  \end{equation}
such that (i) $\Phi_K$ is constant along each end of a fibre of $\Ubar_K$,  (ii) the maps $\phi_K$ and $\psi_K$  are obtained by gluing, and $J_K$  by perturbed gluing, (iii) $\psi_k(z)$ preserves the image of $L$ under $\phi_k(z) $ and (iv) for each $ z  \in \Ubar^{\{1\}}_{K} $
\begin{equation}
J_K(z) = \left( \psi_K(z) \circ \phi_K(z) \right)_* J_L.
\end{equation}
\end{defin}
\begin{rem}
Only the restrictions of $\phi_K$ and $\psi_K$ to $\Ubar^{\{1\}}_{K}  $ shall be used.
\end{rem}
The assumption that $\phi_K$ is obtained by gluing implies that we have well-defined Hamiltonian diffeomorphisms $\{ \phi_i \}_{i \in K}$. To simplify the notation, we let $L_i = \phi_i(L)$.

Given a consistent family one can define, for each $\vr \in \Adams_{K} $, a holomorphic curve equation with moving Lagrangian boundary conditions:
\begin{align} \label{eq:map_fibre_universal_continuation_space}
u \co  \Ubar_{\vr}  \to X & \qquad \partial_{s} u(z) = J_K(z) \partial_{t}u(z) \\  \label{eq:map_fibre_universal_continuation_space-4}
u(z)  \in F_{q}\textrm{ if } z \in \Ubar_{\vr}^{\{0\}}  & \qquad u(z)  \in  \phi_K(z) L \textrm{ if } z \in  \Ubar_{\vr}^{\{1\}}.
\end{align}
\begin{figure}
\centering
\begin{tikzpicture}
\draw[line width=4*\lw] (0, 1*\mw pt)--(- 8*\mw pt, 1*\mw pt);
\draw[line width=4*\lw] (0, -1*\mw pt)--(- 8*\mw pt, -1*\mw pt);
\fill (-4*\mw pt, 0) circle (8*\lw);
\node[left] at (-8*\mw pt, 0)  {$x$};
\node[right] at (0, 0)  {$y$};
\node[below]  at (-4*\mw pt,-1*\mw pt )  {$F_q$};
\node[above]  at (-4*\mw pt,1*\mw pt )  {$  \phi_K(z) L$};
\node[above]  at (-8*\mw pt,1*\mw pt )  {$  L_{\max K}$};
\node[above]  at (0,1*\mw pt )  {$  L_{\min K}$};
\end{tikzpicture}
\caption{The moduli space $\cM_{q,K}(x,y)$. The presence of one interior marked point is shorthand that allows us to distinguish moduli spaces of continuation maps from solutions to Floer's equation. A more precise figure would show one fewer  marked point than $|K|$.}
\label{fig:moduli_cont}
\end{figure}

Given a pair of points $y \in L_{\min K} \cap F_{q}$ and $x \in L_{\max K} \cap F_{q}$ , denote by 
\begin{equation}
\cM_{q,K}(x,y) \to \Adams_{K}
\end{equation}
the moduli space of solutions to Equations \eqref{eq:map_fibre_universal_continuation_space}-\eqref{eq:map_fibre_universal_continuation_space-4}, which in addition satisfy the asymptotic conditions:
\begin{equation}
\lim_{s \to -\infty} u \circ \epsilon_-(s,t) = x  \qquad \lim_{s \to +\infty} u \circ \epsilon_+(s,t) = y  .
\end{equation}

For a generic family $J_{K}$, the Gromov-Floer compactification $\Mbar_{q,K}(x,y)$ is a manifold  with boundary such that
\begin{align}
\dim_{\bR} \Mbar_{q,K}(x,y) & = \deg(x) - \deg(y)+ \dim(\Adams_{K}) \\
& = \deg(x) - \deg(y)+ |K| -2.
\end{align}
The generalisation of Equation \eqref{eq:boundary_Floer_moduli_space} is that the codimension $1$ strata of the boundary come in two families:
\begin{align}  \label{eq:boundary_continuation_strip_out}
& \coprod_{i \in K  } \coprod_{z \in L_{i} \cap F_{q}}  \Mbar_{q,K^\geq_{i}}(x,z) \times \Mbar_{q,K^\leq_{i}}(z,y) \\   \label{eq:boundary_continuation_skip}
& \coprod_{i \in  K \setminus \{\min K, \max K \}}  \Mbar_{q,K \setminus \{ i \}}(x,y) .
\end{align}
Breaking of Floer strips is incorporated in the first case, corresponding to $i = \max K$ or $ i = \min K$. So the first strata project to the interior of $\Adams_{K}$ or to the boundary stratum  where the coordinate corresponding to $i$ equals $\infty$, and the second to the case where this coordinate vanishes.

\begin{rem}
Since the choice of data in Definition \ref{def:compatibly_fam_continuation} takes place for a fixed $K$, there is a slight abuse of notation in Equations \eqref{eq:boundary_continuation_strip_out} and \eqref{eq:boundary_continuation_skip}, given that the moduli spaces which appear are defined with respect to the Floer data restricted from $\Ubar_{K} $. This issue will be addressed in Section \ref{sec:floer-theory-conv}, where such data will be chosen inductively so that there is no ambiguity in the description of the boundary strata. 
\end{rem}

\subsection{Continuation with a distinguished marker}
\label{sec:continuation-prism}

Let $L^+$ and $L^-$ be graded Lagrangians which are transverse, and which both satisfy Condition \eqref{eq:everything_is_good_with_curves} for almost complex structures $J_{L^+}$ and $J_{L^-}$.  Let $K_-,K_+$ be subsets of $K \times \{ - , + \}$ as in Section \ref{sec:adams-paths-prism}. 

\begin{defin} \label{def:compatibly_fam_continuation_prism}
  A \emph{consistent family of continuation data with a distinguished marker} parametrised by $  \Ubar_{K_-, K_+} $ is a map
  \begin{equation}
    \Phi_{K_-,K_+} = (\phi_{K_-,K_+}, J_{K_-,K_+}, \psi_{K_-,K_+}) \co    \Ubar_{K_-,K_+} \to \scrH \times \scrJ \times \scrD
  \end{equation}
such that (i) the restriction to every end is constant, (ii) $J_{K_-,K_+}$ is obtained by perturbed gluing and the other maps by gluing in a neighbourhood of each  boundary stratum, (iii) $ \psi_{K_-,K_+}(z) $ preserves the image of $L^{+}$ (respectively $L^-$) under $\phi_{K_-,K_+}(z)  $ whenever $z \in \Ubar_{K_-,K_+}^{ z_{1} <}$ as defined in Equation \eqref{eq:subset_boundary_after_marked_point} (respectively $z \in \Ubar_{K_-,K_+}^{ z_{1} >} $),    and (iv) 
\begin{equation}
J_{K_-,K_+}(z) = \begin{cases} \left( \psi_{K_-,K_+}(z) \circ \phi_{K_-,K_+}(z) \right)_* J_{L^+} & \textrm{ if }  z \in     \Ubar_{K_-,K_+}^{ z_{1} <} \\
\left( \psi_{K_-,K_+}(z) \circ \phi_{K_-,K_+}(z) \right)_* J_{L^-} & \textrm{ if }  z \in     \Ubar_{K_-,K_+}^{ z_{1} >}.
\end{cases}
\end{equation}
\end{defin}

We obtain, for each $\vr \in \Adams_{K_-,K_+} $, a holomorphic curve equation $ \partial_{s} u(z)  = J(z) \partial_{t}u(z)$ with moving Lagrangian boundary conditions
\begin{align} \label{eq:map_to_X_inp-0}
& u \co  \Ubar_{\vr}  \to X  && \qquad u(z)  \in F_{q}\textrm{ if } z \in \Ubar_{\vr}^{\{0\}} \\ \label{eq:map_to_X_inp-5}
&u(z)  \in  \phi(z) L^- \textrm{ if } z \in  \Ubar_{\vr}^{z_{1}>} && \qquad 
u(z)  \in  \phi(z) L^+ \textrm{ if } z \in  \Ubar_{\vr}^{z_{1}<}.
\end{align}

\subsection{Continuation with one input}
\label{sec:homotopies-with-one}

Let $K_-$ and $K_+$ be subsets of $K$ such that $
  \max K_+ \leq \min K_-$, and let $L$ and $L'$ be graded Lagrangians satisfying Condition \eqref{eq:everything_is_good_with_curves} for almost complex structures $J_L$ and $J_{L'}$. In the setting of the previous section, set $L^+ = L$ and $L^- = L'$.
\begin{figure}
\centering
\begin{tikzpicture}
\draw[line width=4*\lw] (0, 1*\mw pt)--(- 12*\mw pt, 1*\mw pt);
\draw[line width=4*\lw] (0, -1*\mw pt)--(- 12*\mw pt, -1*\mw pt);
\draw[line width=4*\lw] (-6*\mw pt, 1*\mw pt+1/8*\mw pt)--(-6*\mw pt, 1*\mw pt-1/8*\mw pt);
\node[left] at (-12*\mw pt, 0)  {$x'$};
\node[right] at (0, 0)  {$x$};
\node[below]  at (-6*\mw pt,-1*\mw pt )  {$F_q$};
\node[above]  at (-6*\mw pt,1*\mw pt )  {$  x_\inp$};
\node[above]  at (-12*\mw pt,1*\mw pt )  {$  L'_{\max K_-}$};
\node[above]  at (0,1*\mw pt )  {$  L_{\min K_+}$};
\end{tikzpicture}
\caption{The moduli space $\cM_{q,K_-,K_+}(x'; x_\inp, x)$}
\label{fig:moduli_cont}
\end{figure}
 Given points $x' \in L'_{\max K_-} \cap F_{q}$,  $x \in L_{ \min K_+} \cap F_{q}$ and $x_{\inp} \in L \cap L'$ , define
\begin{equation} 
\cM_{q, K_-, K_+}(x'; x_{\inp} , x) \to \Adams_{K_-,K_+}
\end{equation} 
to be the moduli space of solutions to Equations \eqref{eq:map_to_X_inp-0}-\eqref{eq:map_to_X_inp-5}, which in addition satisfy the conditions:
\begin{equation}
 \lim_{s \to -\infty} u \circ \epsilon_-(s,t) = x'  \qquad  u(z_{\inp}) = x_{\inp} \qquad \lim_{s \to +\infty} u \circ \epsilon_+(s,t) = x.
\end{equation}

For a generic family of parametrised Floer data, the Gromov-Floer compactification $\Mbar_{q,  K_-, K_+} ( x'; x_{\inp}, x)$ is a manifold  with boundary such that
\begin{align}
\dim_{\bR} \Mbar_{q, K_-, K_+}(x'; x_{\inp}, x) & =  \deg(x')  - \deg(x_{\inp}) - \deg(x)+ |K_- | + |K_+ | -2.
\end{align}

If $\vK$ is a nested sequence of subsets of $K$, with minimal element $K_0$, let 
\begin{equation} 
\Mbar_{q,\vK;\inp}(x'; x_{\inp} , x) \to \Adams_{\vK;\inp}
\end{equation} 
be the pullback of $ \Mbar_{q, K^\inp_{K_0}}(x'; x_{\inp}, x) $ by the map from $\Adams_{\vK;\inp} $ to $\Adams_{K^\inp_{K_0}}$ (see  Equation \eqref{eq:K_I_inp}). For generic Floer data, this is a manifold  with boundary such that
\begin{align} \label{eq:dimension_moduli_space_inp}
\dim_{\bR} \Mbar_{q, \vK;\inp}(x'; x_{\inp}, x) & = \deg(x')  - \deg(x_{\inp}) - \deg(x)+ |\vK| -1.
\end{align}

\begin{figure}
\centering
\begin{tikzpicture}
\begin{scope}                          
\ConoutBig{-6*\mw pt}{0}{\max K};
\CtoEBigL{0}{0}{\max J}{\min J}{L'}{L};
\ConinBig{6*\mw pt}{0}{\min K};
\end{scope}
\end{tikzpicture}

\caption{A schematic picture of the holomorphic curve problem for curves over $ \Adams_{\vK[\leq]_J;\inp}  \times \Adams_{\vK[\geq]_J}$.}
\label{fig:boundary_moduli_input}
\end{figure}
The boundary decomposition of $\Adams_{\vK;\inp}$ leads to the following decomposition of the boundary of $\Mbar_{q, \vK;\inp}(x'; x_{\inp}, x)$:
\begin{align}
& \coprod_{y' \in  L'_{\max K} \cap F_{q}} \Mbar_{q}(x',y') \times  \Mbar_{q, \vK;\inp}(y'; x_{\inp}, x) \\
& \coprod_{y \in  L_{\min K} \cap F_{q}}   \Mbar_{q,\vK;\inp}( x'; x_{\inp}, y) \times \Mbar_{q}(y,x) \\
& \coprod_{y_{\inp} \in  L \cap L'} \Mbar_{q, \vK;\inp}(x'; y_{\inp}, x) \times \Mbar(y_{\inp},x_{\inp}) \\ \label{eq:stratum_boundary_input_skip}
& \coprod_{I \in \vK \setminus  K }  \Mbar_{q, \vK \setminus I;\inp}(x'; x_{\inp}, x)\\ \label{eq:stratum_boundary_input_pullback_max_min_non_compatible}
& \coprod_{I \in \vK \setminus  K } \coprod_{y \in L_{\min I} \cap F_{q}} \coprod_{y' \in L'_{\max I} \cap F_{q}} \mathrm{max}^* \Mbar_{q,\max \vK[\geq]_I}(x',y') \times_{ \Adams_{\vK[\geq]_I}} \mathrm{min}^* \Mbar_{q,\min \vK[\geq]_I}(y,x) \\ \notag
& \qquad \times \Mbar_{q,\vK[\leq]_I;\inp}( y'; x_{\inp}, y). 
\end{align}

\begin{lem} \label{lem:good_boundary_in} 
Assume that $\vK$ consists of $|K|$ elements, and that $\dim_{\bR} \Mbar_{q, \vK;\inp}(x'; x_{\inp}, x) = 1$. If all Floer data are regular, the only contributions to the stratum in Equation \eqref{eq:stratum_boundary_input_pullback_max_min_non_compatible} are given by $I$ such that (i) $\max I = \max K$ and $\min$ is injective on $\vK[\geq]_I $, in which case the boundary contribution is: 
\begin{equation} \label{eq:boundary_moduli_0_input_max}
 \coprod_{y \in L_{\min  I} \cap F_{q}}   \Mbar_{q,\vK[\leq]_I;\inp}(x'; x_{\inp}, y) \times \Mbar_{q,\min \vK[\geq]_I}(y,x).
\end{equation}
or (ii)  $\min I = \min K$ and $\max$ is  injective on $\vK[\geq]_I $:
\begin{equation} \label{eq:boundary_moduli_0_input_min}
 \coprod_{y' \in L'_{\max  I} \cap F_{q}} \Mbar_{q,\max \vK[\geq]_I}(x',y') \times \Mbar_{q,\vK[\leq]_I;\inp}( y'; x_{\inp}, x), 
\end{equation}
or (iii)  $I = K \setminus \{ i\}$, for $i \neq \min K, \max K$, in which case the corresponding stratum is
\begin{equation}
  \label{eq:boundary_moduli_0_input_skip}
   \Mbar_{q,\vK[\leq]_I ;\inp}( x'; x_{\inp}, x).
\end{equation}
\end{lem}
\begin{proof}
Under the assumption on $\vK$, successive elements of this sequence differ by exactly one element of $K$, hence the same property holds for $\vK[\geq]_I$. The result now follows from Lemma \ref{lem:min_max_maps_deg}, which asserts that, whenever the conditions above do not hold,  the stratum in Equation \eqref{eq:stratum_boundary_input_pullback_max_min_non_compatible} is obtained by pulling back a moduli space of curves parametrised by a manifold of dimension lower than the dimension of the stratum. Assuming regularity, this moduli space has negative virtual dimension, hence is empty.
\end{proof}

\subsection{Continuation with one output}
\label{sec:homotopies-with-one-out}
Let $K = K_- \cup K_+$ be decomposition of $K$ such that  $  \min K_+ = \max K_-$. Set $L^- = L$ and $L^+ = L'$.

Given  points $x \in L_{\max K_-} \cap F_{q}$,  $x'\in L'_{ \min K_+} \cap F_{q}$, and $x_{\out} \in L \cap L'$ , denote by
\begin{equation} 
\Mbar_{q,K_-,K_+}(x, x_{\out} ; x') \to \Adams_{K_-,K_+}
\end{equation} 
the compactified moduli space of solutions to the Cauchy Riemann equation determined by such data, with asymptotic conditions
\begin{equation}
 \lim_{s \to -\infty} u \circ \epsilon_-(s,t) = x  \qquad  u(z_{\out}) = x_{\out} \qquad \lim_{s \to +\infty} u \circ \epsilon_+(s,t) = x'.
\end{equation}

If $\vK$ is a sequence of nested subsets with minimal element $K_0$ and maximal element $K$, let $\Mbar_{q,\vK;\out}( x, x_{\out}; x')$ be the pullback of $\Mbar_{q,K^\out_{K_0}}(x, x_{\out} ; x')  $ under the projection map from $\Adams_{\vK;\out}$ to  $\Adams_{K^\out_{K_0}} $. Assuming that all Floer data are chosen generically, this  is a manifold  with boundary such that
\begin{align}
\dim_{\bR} \Mbar_{q, \vK;\out}(x, x_{\out}; x') & = \deg(x)  + \deg(x_{\out}) - \deg(x')+ |\vK| - n-1.
\end{align}
\begin{figure}
\centering
\begin{tikzpicture}
\begin{scope}                          
\ConoutBig{-6*\mw pt}{0}{\min K_0};
\EtoCBigL{0}{0}{\min I}{\max I}{L}{L'};
\ConinBig{6*\mw pt}{0}{\max K_0};
\end{scope}
\end{tikzpicture}

\caption{A schematic picture of the holomorphic curve problem for curves over $\Adams_{\vK[\geq]_I;\out}  \times \Adams_{\vK[\leq]_I} $.}
\label{fig:boundary_moduli_output}
\end{figure}
The boundary decomposition of $\Adams_{\vK;\out}$ leads to the following decomposition of the boundary of $\Mbar_{q, \vK;\out}(x, x_{\out}; x')$ 
\begin{align}
& \coprod_{y' \in  L'_{\max K_0} \cap F_{q}}  \Mbar_{q,\vK;\out}( x, x_{\out}; y') \times \Mbar_{q}(y',x') \\
& \coprod_{y \in  L_{\min K_0} \cap F_{q}}  \Mbar_{q}(x,y) \times  \Mbar_{q, \vK;\out}(y, x_{\out}; x') \\
& \coprod_{y_{\out} \in  L \cap L'} \Mbar(x_{\out},y_{\out}) \times \Mbar_{q, \vK;\out}(x, y_{\out}; x') \\
& \coprod_{I \in \vK \setminus  K_0 }  \Mbar_{q, \vK \setminus I;\out}(x, x_{\out}; x')\\ \label{eq:stratum_boundary_output_pullback_max_min}
& \coprod_{I \in \vK \setminus  K_0 } \coprod_{y \in L_{\min I} \cap F_{q}} \coprod_{y' \in L'_{\max I} \cap F_{q}} \mathrm{min}^* \Mbar_{q,\min \vK[\leq]_I}(y',x') \times_{\Adams_{\vK[\leq]_I}} \mathrm{max}^* \Mbar_{q,\max \vK[\leq]_I}(x,y)  \\ \notag
& \qquad \times \Mbar_{q,\vK[\geq]_I;\out}( y, x_{\out}; y'). 
\end{align}
We also have the analogue of Lemma \ref{lem:good_boundary_in}:
\begin{lem}
Assume that $\vK$ consists of $|K|$ elements, and that $\dim_{\bR} \Mbar_{q, \vK;\out}(x, x_{\out}; x') = 0$. The only contribution to the stratum in Equation \eqref{eq:stratum_boundary_output_pullback_max_min} is given by $I$ such that (i) $I = K_0 \cup \{ i\}$, for $i \neq \min K, \max K$, in which case the corresponding stratum is
\begin{equation}
  \label{eq:boundary_moduli_0_output_skip}
 \Mbar_{q,\vK[\geq]_I;\out}(x, x_{\out}; x'),
\end{equation} 
or (ii) $\max I = \max K_0$, and $\min$ is injective on $\vK[\leq]_I$ :
\begin{equation} \label{eq:boundary_moduli_0_output_min}
 \coprod_{y \in L_{\min I} \cap F_{q}} \Mbar_{q,\min \vK[\leq]_I}(x,y) \times \Mbar_{q,\vK[\geq]_I;\out}( y, x_{\out}; x'), 
\end{equation}
or (iii) $\min I = \min K_0$, and $\max$ is injective on $\vK[\leq]_I$ 
\begin{equation} \label{eq:boundary_moduli_0_output_max}
 \coprod_{y' \in L'_{\max I} \cap F_{q}}   \Mbar_{q,\vK[\geq]_I;\out}(x, x_{\out}; y') \times \Mbar_{q,\max \vK[\leq]_I}(y',x').
\end{equation}\qed
\end{lem}

\section{Floer theory and convergence} \label{sec:floer-theory-conv}

\subsection{Uniformly small choices of perturbations} \label{sec:unif-small-choic}

Let $L$ and $L'$ be Lagrangians which are tautologically unobstructed with respect to a pair of tame almost complex structures $J_L$ and $J_{L'}$. Fix an embedded path $\{J_{t}\}_{t=0}^{1}$ from $J_L$ to $J_{L'}$. Pick a sequence of contractible neighbourhoods in $\scrJ$ of this path
\begin{equation}
  \scrJ^{0} \subset \scrJ^{1} \subset \cdots \subset \scrJ^{n+3}
\end{equation}
so that the closure of $\scrJ^{i}$ lies in the interior of $\scrJ^{i+1}$.

Let $\scrH^0$ be a contractible neighbourhood of the identity in the space of Hamiltonian diffeomorphisms such that
\begin{equation}
  \phi_{*}(J_t)  \in \scrJ^0
\end{equation}
for all $\phi \in \scrH^0$. Let $\scrL$ and $\scrL'$ denote the families of Lagrangians obtained by applying elements of $\scrH^0$ to $L$ and $L'$.

Let $\scrD$ be a neighbourhood of the identity in $\Diff(X)$ such that $\scrD(\scrJ^i) \subset \scrJ^{i+1}$, where the action is by pushforward as in Equation \eqref{eq:push_J_forward}. Given a subset $W \in \Q^2$,  define
\begin{equation}
  \scrD_{W}(\scrL) \subset W \times \scrL \times \scrD
\end{equation}
to be the subset consisting of triples $((q,p),\phi L, \psi)$ such that 
\begin{align}
  \psi(F_q) & = F_p &  \psi( \phi L) &= \phi L,
\end{align}
and similarly for $\scrL'$. Dropping the first condition yields a family $\scrD(\scrL)$ over $\scrL$, and dropping the second,  a family $\scrD_W $ over $W$. For a fixed element of $\scrL$,  write 
\begin{align}
\scrD(\phi L)  & = \{ \psi \in \scrD | \psi(\phi L) = \phi L \} \\
\scrD_{W}(\phi L) & = \{ \psi \in \scrD_W | \psi(\phi L) = \phi L \} 
\end{align}
for the fibres of $\scrD(\scrL)$ and $   \scrD_{W}(\scrL) $ over $\phi L \in \scrL$.

For all $q \in \Q$, pick $\phi_q \in \scrH^0$ so that $\phi_{q}L$ is transverse to $F_{q}$. Since $\Q$ is compact, there is  a finite cover $\scrU$ of $\Q$ and maps $ \phi_{U}, \phi'_U \in \scrH^0$ for all $U \in \scrU$, so that $\phi_{U}(L)$ and $\phi'_{U}(L') $ are transverse to $F_{q}$ if $q \in U$. For a pair $(p,q) \in U^2$, this implies the existence of diffeomorphisms mapping $q$ to $p$ and preserving $\phi_U(L)$ or  $\phi'_{U}(L') $. Such diffeomorphisms may not act appropriately on the spaces $\scrJ^k$ of almost complex structure, but,  if the elements of $\scrU$ are sufficiently small, there is a sequence of neighbourhoods 
\begin{equation} \label{eq:nested_diffeo}
  \Id = \scrD^{0} \subset \scrD^{1} \subset \cdots \subset \scrD^{2^{n+1}} \subset \scrD
\end{equation}
which are invariant under inversion so that 
\begin{equation} \label{eq:composition_of_diffeo_nested}
  \scrD^{i} \circ \scrD^{j} \subset \scrD^{i+j}
\end{equation}
and we have acyclic fibrations $  \scrD^{k}(\scrL) \to \scrL $,  $  \scrD^{k}(\scrL') \to \scrL' $, and
\begin{equation}
    \xymatrix{\scrD^k_{U^2}(\phi_{U}L) \ar[dr]& \scrD^k_{U^2} \ar[d] & \scrD^k_{U^2}(\phi_{U}L') \ar[dl] \\
& U^2. & }
\end{equation}

We now apply the construction of Section \ref{sec:covers-q} with the above cover in mind: 
\begin{lem}  \label{cor:cover_including_open_stars}
There is a partially ordered set $\Sigma$, labelling a simplicial triangulation of $\Q$ and a cover $\{ P_i \}_{i \in \Sigma}$ of $\Q$ by integral affine polygons which refine $\scrU$, such that  Equations \eqref{eq:nested_inclusion_polygons} and \eqref{eq:Q-obtained-by-gluing} hold, and such that the open star of each cell $\sigma_I$ for $I \subset \Sigma$ is contained in  $ P_I$.
\end{lem}
 \begin{proof}
 Lemma \ref{lem:cover_P_contain_sigma} provides a cover satisfying all the properties except the inclusion of the open star. To remedy this, let $B \Sigma$ denote the barycentric subdivision of $\Sigma$, which is a partially ordered set whose elements are the totally ordered subsets of $\Sigma$; we set the ordering on $B \Sigma$ to be given by reverse inclusion. For $I \in B \Sigma$ define $BP_{I}$ to be $P_{\min I}$. To reduce the clutter, we write
 \begin{equation}
   BP_{i_0\cdots i_k} \equiv BP_{\{i_0, \ldots, i_k\}}.
 \end{equation}
The polyhedra $BP_{I}$ satisfy Equation \eqref{eq:nested_inclusion_polygons} since $\min J \leq \min I$ whenever $I \subset J$ (this justifies the ordering by reverse inclusion). To check condition \eqref{eq:Q-obtained-by-gluing}, consider the natural map
\begin{equation} \label{eq:barycentric_mod_equivalence}
\bigcup_{I \in B \Sigma} BP_I / \!\!\sim \; \to \bigcup_{i \in \Sigma} P_i /\!\!\sim   \; = \Q
\end{equation}
associated to the assignment $I \to \min I$. Observe that, for each $I \in B \Sigma$, the above  map restricts to an isomorphism 
\begin{equation}
\bigcup_{J \subset I} BP_J /\!\!\sim \; \to  \bigcup_{\min I \leq i} P_i /\!\!\sim  \; = P_{\min I},
\end{equation}
hence the subsets $BP_{i} = P_{i}$ cover the left hand side of Equation \eqref{eq:barycentric_mod_equivalence}. For a pair $i < j$, the inclusions 
\begin{equation}
  BP_{i} \subset BP_{ij} \supset BP_{j}  
\end{equation}
imply that Equation \eqref{eq:barycentric_mod_equivalence} is injective.  We conclude as desired that it is therefore a homeomorphism.

Iterating the barycentric subdivision construction yields open stars whose diameter goes to $0$, while the corresponding cover does not shrink because it consists of repetitions of the original cover. There is therefore a finite iteration with the property that the open star of every vertex is contained in the corresponding polygon.
 \end{proof}

For each $i$, pick $\phi_{i}, \phi'_i \in \scrH^0$ arbitrarily among those Hamiltonian diffeomorphisms $\phi_{U}$ and $\phi'_U$ where $P_i \subset U$.  It is immediate that $\scrD^{i}_{P_i^2}$, $\scrD^{i}_{P_i^2}(\phi_i L)$ and   $\scrD^{i}_{P_i^2}(\phi_i L')$ are acyclic fibrations over $P_i^2$. In addition, choose generic families $J_{i} , J'_i\co [0,1] \to \scrJ^{1} $ so that
\begin{equation} \label{eq:almost_complex_structure_vertex}
J_{i,0} = (\phi_i)_{*} J_L \textrm{ and } J'_{i,0} = (\phi'_i)_{*} J_{L'}
\end{equation}
and all moduli spaces of Floer trajectories $\Mbar_{q_i}(x,y)$ and $\Mbar_{q_i}(x',y')$ defined with respect to these families of almost complex structures are regular if $x,y \in L_i \cap F_{i}$ and $x',y' \in L'_i \cap F_{i}$. We write
\begin{equation} \label{eq:Floer_data_element_Sigma}
  \Phi_{i} = (\phi_i, J_{i}, \Id  ) \co [0,1] \to \scrH^0 \times \scrJ \times \scrD,
\end{equation}
where the first and last maps are constant, and similarly for $\Phi'_i$. Moreover, pick sections
\begin{equation}
  \psi_{i} \co P_{i} \to \scrD^{1}_{ \{q_{i} \} \times P_{i}}(L_i).
\end{equation}

\begin{rem}
In order not to increase the awkwardness of the notation, we shall adopt the following conventions: let $Z$ be a topological space.  Given maps
\begin{equation}
  (\Phi, \Psi) \co Z \to \scrD \times \scrH \times \scrJ \times \scrD,  
\end{equation}
with $\Phi = (\phi,J,\psi)$, use pushforward and composition (pointwise in $Z$)  to define
\begin{equation}
 \Psi_*(\Phi) \equiv ( \phi, \Psi_* J, \Psi \circ \psi) \co Z \to  \scrH \times \scrJ \times \scrD.
\end{equation}
Also, write $\Psi^{-1}$ for the pointwise inverse of $\Psi$.
\end{rem}

\subsection{Locally constant families of continuation maps} \label{sec:locally-const-famil} \label{sec:continuation-maps-at-vertex}
Let $I$ be a totally ordered subset of $\Sigma$,  and recall the notation from Equation \eqref{eq:conventions_short_notation} for the intersection $P_I$ and element $q_I \in P_I$. Pick  a map
\begin{equation} \label{eq:diffeo_pair_cont}
\Psi_{I} \co P_{I} \times \Ubar_{I} \to \scrD^{2^{|I|-1}}
\end{equation}
which is the identity on $\{ q_I \} \times \Ubar_I$, and continuation data
\begin{equation}
  \Phi_I  = (\phi_I, J_I, \psi_I) \co \Ubar_{I} \to \scrH^0 \times \scrJ^{|I|} \times \scrD^{2^{|I|-1}}.
\end{equation}

\begin{rem} \label{rem:explain_construction_continuation}
The  remainder of this section is likely to be less incomprehensible if the reader keeps in mind that the holomorphic curve problem on the moduli space $\Ubar_{I}$ is defined with Lagrangian boundary conditions $F_{{I}}$ over the boundary labelled $0$. The map $\Psi_I$ is introduced to transport this holomorphic curve problem to nearby fibres.

A potentially helpful reference is  Section 3.5 of \cite{A-ICM} which essentially covers the case where the set $I$ consists of two elements. In this case, $\Ubar_{I} $ is a strip, so that Equation \eqref{eq:diffeo_pair_cont} is the choice of a family of diffeomorphisms of $X$ parametrised by the product of $P_I$ with a strip. As a warning to the reader, we note that we required in Equation (3.50) of \cite{A-ICM} that $\Psi_I$ be constant in a certain region of the strip, which ensured that it preserved the image of $L$ under a moving Hamiltonian isotopy. In the present account, this is replaced by Equation  \eqref{eq:condition_diffeo_I_to_K_along_1} which is more flexible.
\end{rem}
In addition to the conditions imposed on $\Phi_I$ in Definition \ref{def:compatibly_fam_continuation}, assume that the restrictions of $(\Phi_I, \Psi_{I})$ agree with
\begingroup
\allowdisplaybreaks
  \begin{align} \label{eq:cont_data_end_+}
& ({\psi_{\min I}(q_I)}_*\Phi_{\min I}, \psi_{\min I} \circ \psi^{-1}_{\min I}(q_I)) && \textrm{ along the end }  \epsilon_+ \\ \label{eq:cont_data_end_-}
& (\Phi_{\max I}, \psi_{\max I})  & &\textrm{ along the end }  \epsilon_- \\ \label{eq:drop_element_same_data}
& (\Phi_{I \setminus i} ,\Psi_{I  \setminus i }) &&\textrm{ on } \Ubar_{I \setminus i}  \\ \label{eq:cont_data_stop_i_minus}
& ( \Phi_{I^\geq_{i}} , \Psi_{I^\geq_{i} }) &&\textrm{ on }  \Ubar_{I^\geq_{i}} \times \Adams_{I^\leq_{i}}  \\ \label{eq:cont_data_stop_i_plus}
& \left({\Psi_{I^\leq_{i} }(q_I)}_{*}  \Phi_{I^\leq_{i}}  , \Psi_{I^\leq_{i}} \circ\Psi_{I^\leq_{i} }^{-1}(q_I)  \right)  &&\textrm{ on } \Adams_{I^\geq_{i}} \times   \Ubar_{I^\leq_{i}}.
  \end{align}
\begin{rem} \label{rem:interpret_funny_composition}
The last factor in Equation \eqref{eq:cont_data_end_+} should be interpreted as follows: $\Psi_I$  is a map which depends on $q \in P_I = P_{\max I}$, whereas $\psi_{\min I}$ depends on $q \in P_{\min I}$. By assumption, $P_{\max I} \subset P_{\min I}$, so it makes sense to require that
\begin{equation*}
\Psi_I(q, z) =  \psi_{\min I}(q) \circ \psi^{-1}_{\min I}(q_I)
\end{equation*}
whenever $z$ lies along the end $\epsilon_+$ of $\Ubar_{I} $. Equation \eqref{eq:cont_data_stop_i_plus} is to be interpreted in the same way.
\end{rem}
Moreover, on the boundary of each fibre, 
  \begin{align} \label{eq:condition_diffeo_I_to_K_along_1}
\Psi_{I }\left( q,z \right) & \in  \scrD^{2^{|I|-1}}(\phi_{I,z}(L)) \textrm{ if } z \in \Ubar^{\{1\}}_{I} \\ \label{eq:condition_diffeo_I_to_K_along_0}
\Psi_{I}\left(q,z  \right) & \in  \scrD^{2^{|I|-1}}_{q_{I}, q} \textrm{ if } z \in \Ubar^{\{0\}}_{I} .
  \end{align}
\endgroup

\begin{lem} \label{lem:construction_by_induction}
There are choices of maps $(\Phi_I, \Psi_{I})$  which are obtained by perturbed gluing in a neighbourhood of every boundary stratum so that Equations \eqref{eq:cont_data_end_+}-\eqref{eq:condition_diffeo_I_to_K_along_0} are satisfied.
\end{lem}
\begin{proof}
Proceed by induction on the number of elements of $I$. In the base case, $|I|=2$ so $\Adams_{I}$ is a point, $\Ubar_{I} = \Strip$, and $\Ubar^{\{1\}}_{I} \cong \bR$ . Pick a map $\phi_{I}$ which agrees with $\phi_{\min I}$ and $\phi_{\max I}$ at the two ends. Then pick $\psi_{I}$ agreeing with $\Id$ near the negative end and with $ \psi_{\min I}(q_I)$ near the positive end, and so that $
  \psi_{I} \left( \phi_{I}(z) L \right) = \phi_{I}(z) L$ for all $z \in \Ubar^{\{1\}}_{I}$. Note that both values at the endpoints lie in $\scrD^{1}$, so such a path may be chosen in $\scrD^{1}$ by the assumption that the forgetful map from $\scrD^{k}(\scrL)$ to $\scrL$ is an acylic fibration. Also choose a family of almost complex structures $J_{I}$ whose restriction to the negative end is $ J_{\max I}$ and to the positive end is $ \left( \psi_{\min I} (q_I)  \right)_{*} J_{\min I} $, and which agrees with $(\psi_{I} \circ  \phi_{I})_{*} J$ along $\Ubar^{\{1\}}_{I}$. Since all these almost complex structures lie in $\scrJ^{2}$, the image of $ J_{I}$ may be required to also lie in $\scrJ^{2}$ since this space was assumed to be contractible. This completes the construction of $\Phi_I$ for $2$-element subsets $I$.

Pick  a map $\Psi_{I }$ on $P_I \times \Ubar_{I}$ subject to the condition that along the ends, Equations \eqref{eq:cont_data_end_+} and \eqref{eq:cont_data_end_-} hold.  These constraints imply that the image of $\Psi_{I}$ along the ends lies in $\scrD^{2}$; extend it to a map from $P_I \times \Ubar_{I}$ to $\scrD^{2} $, with the requirement that Equations \eqref{eq:condition_diffeo_I_to_K_along_1} and \eqref{eq:condition_diffeo_I_to_K_along_0} hold along the two boundaries of the strip.  This completes the base case.

Given an ordered subset $K$ of $\Sigma$, assume by induction that  continuation data and families of diffeomorphisms for all subsets $I$ of $K$ which satisfy Equations \eqref{eq:cont_data_end_+}-\eqref{eq:condition_diffeo_I_to_K_along_0} have been chosen. Define $(\Phi_{K}, \Psi_K)$ along the boundary strata of $\Ubar_{K}$ using Equation \eqref{eq:drop_element_same_data}-\eqref{eq:cont_data_stop_i_plus}. The inductive hypothesis for $(\Phi_I, \Psi_I)$ implies that this construction, which is a priori defined only on each separate condimension $1$ cell of $\Ubar_K$, restricts to the same map on each codimension $2$ cell, hence defines a map from the boundary of $\Ubar_{K}$.  Moreover, the conditions imposed on the ends in Equations \eqref{eq:cont_data_end_+}-\eqref{eq:cont_data_end_-} show that the restriction of these data to an end of a component of a fibre in $\Ubar_{K}$  which is labelled by $i \in K$  is given by $ \left( \psi_{i} (q_K) _{*}   \Phi_{i}, \psi_{i} \circ \psi_i^{-1}(q_K) \right)$. 
 
One can therefore use gluing to extend these data from the boundary of $\Ubar_{K}$ to its interior as follows: first pick the extension $\phi_K$ by gluing with image in $\scrH^0$. Then pick the extension $\psi_{K}$ by gluing which preserves the Lagrangian boundary conditions, and whose restrictions to the positive and negative ends are respectively given by $\psi_{\min K}(q_K) $ and the identity.  By induction, the images of $\psi_{I}$ and   $\Psi_{I } $ both lie in $\scrD^{2^{|I|-1}} \subset  \scrD^{2^{|K|-2}}$, so such an extension may be chosen to have image in $\scrD^{2^{|K|-1}}$.  The essential condition here is Equation \eqref{eq:cont_data_stop_i_plus} in which the two maps, which by induction lie in $\scrD^{2^{|K|-2}}  $,  are composed.  

The next step is to extend $ J_K$ by perturbed gluing subject to the condition that it agree with the pushforward of $J$ under  $\psi_{K} \circ \phi_K$ on  $\Ubar^{\{ 1 \}}_{K} $, with the pushforward of  $J_{\min K}$ under $\psi_{\min K}(q_K)$ along the positive end, and with  $J_{\max(K)}$ along the positive end. The same argument as above shows that the image of $ J_{K}$ may be chosen to lie in $\scrJ^{|K|}$:  Equation \eqref{eq:cont_data_stop_i_plus} again imposes the main constraint, since the almost complex structure in this region is obtained pushing $J_{I_i^{\leq}}$ forward by ${\Psi_{I_i^{\leq}}}_{*}$. By induction, this pair lies in $\scrJ^{|K|-1} \times \scrD^{2^{|K|-2}}$, which is contained in $\scrJ^{|K|-1} \times \scrD  $, so the pushforward lies in $\scrJ^{|K|} $.

Finally, by the inductive hypothesis, the image of $ \Psi_{I}   $ lies in $\scrD^{2^{|I|-1}} \subset \scrD^{2^{|K|-2}}$. By Equation \eqref{eq:drop_element_same_data}-\eqref{eq:cont_data_stop_i_plus} the map  defined on the boundary strata of $\Ubar_{K}$ has image in $ \scrD^{2^{|K|-1}}$. Since this space is contractible, choose an extension to a map  $\Psi_{K}$  defined on $\Ubar_{K}$.  Construct this map by gluing in a neighbourhood of the boundary strata,  requiring in addition that  Equations \eqref{eq:cont_data_end_+} and \eqref{eq:cont_data_end_-}  hold along the ends and Equations \eqref{eq:condition_diffeo_I_to_K_along_1} and \eqref{eq:condition_diffeo_I_to_K_along_0} along the boundary; these properties can be achieved given the assumptions on the acyclicity of $  \scrD^{2^{|K|-1}}(\scrL)$ and   $  \scrD^{2^{|K|-1}}_{P^2_K}$. This completes the construction of the data $(\Phi_I , \Psi_{I }) $ by induction on $|I|$.
\end{proof}

Denote the pushforward $ {\Psi_K}_*\Phi_K$ by:
\begin{align} \label{eq:locally_constant_continuation}
\Phi^P_K  \co P_K \times \Ubar_{K} & \to\scrH  \times \scrJ^{|K|+1}  \times \scrD^{2^{|K|}}.
\end{align}

The conditions imposed on $\phi_{K}$ in the previous section imply that, for each $q \in P_K$,  $\Phi^P_K(q)$ is a compatible family of continuation data in the sense of Definition \ref{def:compatibly_fam_continuation}.  For each $k \in K$, the diffeomorphism $ \psi_{k}(q) $ maps $F_{k}$ to $F_{q}$, and preserves $\phi_{k} L$, hence maps a point $x_{k} \in F_{k} \cap \phi_{k} L$ to 
\begin{equation} \label{eq:push_intersect_diffeo}
  x_{k}(q) \equiv   \psi_{k}(q) x_{k} \in F_{q} \cap  \phi_{k} L.
\end{equation}

\begin{lem}
Given $y \in F_{{\min K}} \cap \phi_{\min K} L$ and $x \in F_{{\max K}} \cap \phi_{\max K} L$, composition with $\Psi_K(q)$ yields a homeomorphism
\begin{equation}
    \xymatrix{    \Mbar_{q,K}(x(q), y(q)) \ar[r]^{\cong} & \Mbar_{q_K,K}(x(q_K), y(q_K)) }.
\end{equation} \qed
\end{lem}

When there is no ambiguity, we shall write $\Mbar_{K}(x, y) $ for this moduli space. Fix generic choices of continuation data $\Phi_K$, for which all such moduli spaces are regular.

\subsection{Locally constant continuation maps with a distinguished marker: input}
\label{sec:locally-const-cont-inp} \label{sec:cont-with-one}

Start by assuming that data $(\Phi_I, \Psi_{I })$ and $(\Phi'_I, \Psi'_{I })$ for the Lagrangians $L$ and $L'$ have been chosen for all totally ordered subsets $I$ of $\Sigma$. As in Section \ref{sec:homotopies-with-one-out}, pick a decomposition $I = I_- \cup I_+ $ with the property that  that $\max I_+ = \min I_-$.

\begin{rem}
To make sense of this section, the reader should consult Figure \ref{fig:boundary_moduli_input}. The holomorphic curve problem in the presence of an input along the boundary has Lagrangian boundary conditions $F_{\max I}$ along the $t= 0$ boundary (using the fact that $\max I = \max I_-$).
\end{rem}
\begingroup
\allowdisplaybreaks
Pick a continuation datum with inputs $\Phi_{I_-,I_+} $, and a family of diffeomorphisms $\Psi_{I_-,I_+ }  $ which agree with the identity on  $\{q _I \} \times  \Ubar_{I_-,I_+}$ and are obtained by gluing 
\begin{align} \label{eq:Hamiltonian_cont_in}
\Phi_{I_-,I_+} \co \Ubar_{I_-,I_+} & \to \scrH^0  \times \scrJ^{|I_-|+|I_+|-1}  \times \scrD^{2^{|I_-|+|I_+|-2}} \\
 \Psi_{I_-,I_+ } \co P_{I}  \times \Ubar_{I_-,I_+} & \to \scrD^{2^{|I_-|+|I_+|-2}} .
\end{align}
The restrictions of these maps to subsets of $\Ubar_{I_-,I_+} $ are required to agree with
\begin{align} \label{eq:continuation_inp_end}
& ({\psi_{\min I}(q_I)}_*\Phi_{\min I}, \psi_{\min I} \circ \psi^{-1}_{\min I}(q_I)) && \textrm{along the end } \epsilon_+  \\
& ( \Phi'_{\max I} , \psi'_{\max I})&& \textrm{along the end }  \epsilon_- \\ \label{eq:continuation_inp_skip_-}
&  (  \Phi_{I_- \setminus \{ i\},I_+  }, \Psi_{I_- \setminus \{ i\}, I_+}) && \textrm{on } \Ubar_{I_- \setminus \{ i\},I_+  }   \\
& (  \Phi_{I_-,I_+  \setminus \{ i\}  }, \Psi_{I_-,I_+  \setminus \{ i\}} ) && \textrm{on }  \Ubar_{I_- ,I_+ \setminus \{ i\}  } \\ \label{eq:continuation_inp_something}
&  \left( \Phi^{P}_{ I_{+,i}^\leq }(q_I),  \Psi_{ I_{+,i}^\leq } \circ  \left(\Psi_{ I_{+,i}^\leq }(q_I) \right)^{-1} \right) && \textrm{on } \Adams_{I_-,I_{+,i}^\geq} \times  \Ubar_{I_{+,i}^\leq}   \\
& (\Phi_{I_-,I_{+,i}^\geq} , \Psi_{(I_-,I_{+,i}^\geq) })   && \textrm{on }\Ubar_{I_-,I_{+,i}^\geq} \times  \Adams_{I_{+,i}^\leq} \\
&   (\Phi'_{  I_{-,i}^\geq } ,   \Psi'_{ I_{-,i}^\geq } ) && \textrm{on }\Ubar_{I_{-,i}^\geq} \times \Adams_{I_{-,i}^\leq , I_+ }  \\ \label{eq:restrict_J_in_stop}
&  \left( \left( \Psi_{(I_{-,i}^\leq , I_+) }(q_I)\right)_* \Phi_{I_{-,i}^\leq , I_+ } , \Psi_{(I_{-,i}^\leq , I_+) } \circ \left( \Psi_{(I_{-,i}^\leq , I_+) }(q_I) \right)^{-1} \right) && \textrm{on } \Adams_{I_{-,i}^\geq} \times \Ubar_{I_{-,i}^\leq , I_+ } ,
  \end{align}
where we interpret Equation \eqref{eq:continuation_inp_end}, \eqref{eq:continuation_inp_something}, and  \eqref{eq:restrict_J_in_stop} in accordance with Remark \ref{rem:interpret_funny_composition}. Moreover, on the boundary of each fibre, the following conditions hold: 
  \begin{align}
\Psi_{I_-,I_+ }\left(q, z \right) & \in  \scrD^{2^{|I_-|+|I_+|-2}}(\phi_{I_-,I_+,z}(L)) \textrm{ if } z \in \Ubar^{z_{\inp}<}_{I_-,I_+} \\ 
\Psi_{I_-,I_+ }\left( q,z \right) & \in  \scrD^{2^{|I_-|+|I_+|-2}}(\phi_{I_-,I_+,z}(L')) \textrm{ if } z \in \Ubar^{z_{\inp}> }_{I_-,I_+} \\ 
\Psi_{I_-,I_+ }\left(q, z  \right) & \in  \scrD^{2^{|I_-|+|I_+|-2}}_{q_{\max I_-}, q} \textrm{ if } z \in \Ubar^{\{0\}}_{I_-,I_+} .
\end{align}
The existence of such data again follows from an inductive argument on the number of elements of $I_-$ and $I_+$, as in the proof of Lemma \ref{lem:construction_by_induction}. 
\endgroup

The pushforward of $\Phi_{I_-,I_+}$ by $\Psi_{ I_-,I_+}$ defines a map
\begin{align} \label{eq:J_inp_polygon}  \Phi^P_{I_-,I_+ } \co P_{\max I_-} \times \Ubar_{I_-,I_+} & \to \scrH^0 \times \scrJ^{n+2} \times \scrD^{2^{n+1}}.
\end{align}

As in Section \ref{sec:locally-const-famil}, given $x \in \phi_{\min I_+} L \cap F_{{\min I_+}}$, $x_{\inp} \in L \cap L'$, and $x' \in \phi_{\max I_-} L' \cap F_{{\max I_-}}$, evaluation of $\Phi^P_{I_-,I_+} $ at $q$  defines a moduli space $
\Mbar_{q,I_-,I_+ }( x'(q); x_{\inp}, x(q))$, where $x'(q)$ and $x(q)$ are the images of $x'$ and $x$ under $\psi'_{\max I_-}(q)$ and $ \psi_{\min I_+}(q)  $.

By construction, composition with $\Psi_{I_-,I_+}(q)  $  yields a homeomorphism 
\begin{equation} \label{eq:moduli_space_independent_basepoint_inp}
    \Mbar_{q_{\max I_-},I_-,I_+}(x'(q_{\max I_-}); x_{\inp}, x(q_{\max I_-}))  \cong \Mbar_{q,I_-,I_+}(x'(q); x_{\inp}, x(q)).
\end{equation}

We choose the continuation data generically so that all such moduli spaces are regular.

\subsection{Locally constant continuation maps with a distinguished marker: output} \label{sec:locally-const-famil-output}\label{sec:cont-with-one-out}

As in Section \ref{sec:cont-with-one-out}, assume that $I_-$ and $I_+$ are totally ordered subsets of $\Sigma$ such that $  \min I_+ = \max I_-$, with $I = I_- \cup I_+$ also totally ordered.
\begin{rem}
  Figure \ref{fig:boundary_moduli_output}, and the fact that the holomorphic curve problem on the moduli space $\Ubar_{I_-,I_+}$ in the presence of an output along the boundary is defined with Lagrangian boundary conditions $F_{{\min I_+}}$ may be helpful in understanding the construction.
\end{rem}
\begingroup
\allowdisplaybreaks
Pick a continuation datum with outputs $\Phi_{I_-,I_+} $, and a family of diffeomorphisms $\Psi_{I_-,I_+ }  $ which are obtained by gluing and agree with the identity on $\{q _{\min I_+} \} \times  \Ubar_{I_-,I_+}$: 
\begin{align} \label{eq:Hamiltonian_cont_out}
\Phi_{I_-,I_+} \co \Ubar_{I_-,I_+} & \to \scrH^0  \times \scrJ^{|I_-|+|I_+|-1} \times \scrD^{2^{|I_-|+|I_+|-2}}   \\
 \Psi_{I_-,I_+ } \co P_{ \min I_+}  \times \Ubar_{I_-,I_+} & \to \scrD^{2^{|I_-|+|I_+|-2}} 
\end{align}
The restrictions of these maps to subsets of $\Ubar_{I_-,I_+} $ are required to be
\begin{align}
& {\left( \Phi'_{\min I_+}, \psi'_{\min I_+} \right)}&& \textrm{along the end } \epsilon_+  \\
&( \left( \psi_{\max I_-} (q_{ \min I_+}) \right)_{*} \Phi_{\max I_-} ,\psi_{\max I_-} \circ \psi_{\max I_-}^{-1} (q_{\min I_+}))&& \textrm{along the end }  \epsilon_-  \\ \label{eq:cont_out_skip_-}
&   (  \Phi_{I_- \setminus \{ i\},I_+  }, \Psi_{I_- \setminus \{ i\},I_+})   && \textrm{on }    \Ubar_{I_- \setminus \{ i\},I_+  }  \\
&(  \Phi_{I_-,I_+  \setminus \{ i\}  }, \Psi_{I_-,I_+  \setminus \{ i\}} ) && \textrm{on }   \Ubar_{I_- ,I_+ \setminus \{ i\}  }  &&\\ \label{eq:restrict_data_out_+}
&   \left( \Phi'^P_{I_{+,i}^\leq}(q_{\min I_+}) ,  \Psi'_{I_{+,i}^\leq} \circ \left( \Psi'_{I_{+,i}^\leq}(q_{\min I_+}) \right)^{-1} \right)  && \textrm{on }  \Adams_{I_-,I_{+,i}^\geq} \times  \Ubar_{I_{+,i}^\leq}  \\
&{\scriptstyle (( \Psi_{(I_-,I_{+,i}^\geq ) }(q_{\min I_+}) )_* \Phi_{I_-,I_{+,i}^\geq}, \Psi_{(I_-,I_{+,i}^\geq) }\circ  \Psi^{-1}_{(I_-,I_{+,i}^\geq) }( q_{\min I_+}) ) }  && \textrm{on }   \Ubar_{I_-,I_{+,i}^\geq} \times  \Adams_{I_{+,i}^\leq}  \\
&   (\Phi^P_{I_{-,i}^\geq }(q_{\min I_+}) ,  \Psi_{I_{-,i}^\geq } \circ  \Psi^{-1}_{I_{-,i}^\geq} ( q_{\min I_+} )  )  && \textrm{on }  \Ubar_{I_{-,i}^\geq} \times \Adams_{I_{-,i}^\leq , I_+ }  &&\\  \label{eq:restrict_J_out_stop}
&   \left(  \Phi_{I_{-,i}^\leq , I_+ }, \Psi_{(I_{-,i}^\leq , I_+) }\right)   && \textrm{on }     \Adams_{I_{-,i}^\geq} \times \Ubar_{I_{-,i}^\leq , I_+ }
  \end{align}
On the boundary of each fibre:
  \begin{align}
\Psi_{I_-,I_+ }\left(q, z \right) & \in  \scrD^{2^{|I_-|+|I_+|-2}}(\phi_{I_-,I_+,z}(L')) \textrm{ if } z \in \Ubar^{z_{\out}<}_{I_-,I_+} \\ 
\Psi_{I_-,I_+ }\left(q, z \right) & \in  \scrD^{2^{|I_-|+|I_+|-2}}(\phi_{I_-,I_+,z}(L)) \textrm{ if } z \in \Ubar^{z_{\out}> }_{I_-,I_+} \\ 
\Psi_{I_-,I_+ }\left(q, z  \right) & \in  \scrD^{2^{|I_-|+|I_+|-2}}_{q_{\min I_+}, q} \textrm{ if } z \in \Ubar^{\{0\}}_{I_-,I_+}.
  \end{align}
\endgroup
It is crucial at this stage that $P_{I}$ was required to include the closure of the open star of the cell corresponding to $I$, as expressions of the form $\Psi_I(q_{i}) $, with $i$ an arbitrary element of $I$, must be defined.

The pushforward of  $\Phi_{I_-,I_+}$ by $\Psi_{ I_-,I_+}$ defines a map
\begin{align} \label{eq:J_out_polygon}
\Phi^P_{I_-,I_+ } \co P_{\min  I_+} \times \Ubar_{I_-,I_+} & \to \scrH^0 \times \scrJ^{n+2} \times \scrD^{2^{n+1}}.
\end{align}

Given $x \in \phi_{\max I_-} L \cap F_{{\max I_-}}$, $x_{\out} \in L \cap L'$, and $x' \in \phi_{\min I_+} L' \cap F_{{\min I_+}}$, the data  $\Phi^P_{I_-,I_+ }  (q)$ yield a moduli space $
 \Mbar_{q,I_-,I_+}( x(q) ,x_{\out}; x'(q))$, 
for each point $q \in P_{\max I_+} $, where $x(q)$ and $x'(q)$ are the images of $x$ and $x'$ under $ \psi_{\max I_-}(q) $  and $\psi'_{\min I_+}(q)$. Composition with $ \Psi^P_{I_-,I_+}(q)  $  yields a homeomorphism
\begin{equation} \label{eq:diffeo_moduli_varying_points}
    \Mbar_{q_{\min I_+},I_-,I_+ }(x (q_{\min I}), x_{\out}; x'(q_{\min I}))  \cong \Mbar_{q,I_-,I_+ }(x(q), x_{\out}; x'(q)) .
\end{equation}

\subsection{Moduli spaces associated to cells of $\Sigma$}
\label{sec:moduli-spac-assoc}

Combining the discussions of Section \ref{sec:strips-with-one-positive} and \ref{sec:homotopies-with-one}, yields a moduli space
\begin{equation}
   \Mbar_{q,\vK; \inp}(x'(q); x_{\inp}, x(q)),
\end{equation}
whenever $q \in P_{\max \vK}$, by taking the fibre product of $  \Mbar_{q,K^\inp_{K_0}}(x'(q); x_{\inp}, x(q))$ with $\Adams_{\vK; \inp}$ over $\Adams_{K^\inp_{K_0}} $.  Given a totally ordered subset $K$ of $\Sigma$, define
\begin{equation} \label{eq:definition_moduli_in_cells}
   \Mbar_{q,K; \inp}(x'(q); x_{\inp}, x(q)) = \bigcup_{\substack{|\vK| = |K| \\  K=\max \vK}}  \Mbar_{q,\vK; \inp}(x'(q); x_{\inp}, x(q)),
\end{equation}
where the moduli spaces on the right are glued along the boundary strata obtained by omitting an element of $\vK$.  This moduli space is again independent of $q$ by applying the homeomorphism in Equation \eqref{eq:moduli_space_independent_basepoint_inp}, so  it is denoted
$  \Mbar_{K; \inp}(x'; x_{\inp}, x)$.

This moduli space is parametrised by a manifold of dimension $|K| -1 $, obtained by taking the union of the spaces $\Adams_{\vK; \inp}$, for $\vK$ a maximal length sequence with largest element $K$, glued along the boundary strata in Equation \eqref{eq:inp_adams_skip}. Assuming that all its strata are defined using regular continuation data, the dimension of the interior is
\begin{equation}
  \deg(x')  - \deg(x_{\inp}) - \deg(x)+ |K| -1.
\end{equation}

\begin{lem}
If the dimension is $1$, then $\Mbar_{K; \inp}(x'; x_{\inp}, x) $ is a $1$-dimensional manifold with boundary
\begin{align} \label{eq:boundary_moduli_inp_Floer_diff}
& \coprod_{y_{\inp} \in  L \cap L'} \Mbar_{K;\inp}(x'; y_{\inp}, x) \times \Mbar(y_{\inp},x_{\inp})  \\ \label{eq:boundary_moduli_inp_cont_input}
& \coprod_{ i \in K  } \coprod_{y \in \phi_i L \cap F_i}  \Mbar_{K^{\geq}_i;\inp}( x'; x_{\inp}, y) \times  \Mbar_{K^\leq_i}(y,x)\\
& \coprod_{i \in K } \coprod_{y' \in \phi_i L' \cap F_i}  \Mbar_{K^{\geq}_i}(x',y') \times  \Mbar_{K^\leq_i;\inp}( y'; x_{\inp}, y)  \\ \label{eq:stratum_boundary_input_pullback_max_min}
& \coprod_{i \in K \setminus \{\min K, \max K \} } \Mbar_{K \setminus \{i\}; \inp}(x'; x_{\inp}, x)
\end{align}
\end{lem}
\begin{proof}
The description of the boundary of the moduli space follows from Lemma \ref{lem:good_boundary_in}. The key point is that $\Mbar_{K; \inp}(x'; x_{\inp}, x) $  is obtained by gluing the moduli spaces $ \Mbar_{\vK;\inp}(x'; x_{\inp} , x) $ along the strata given in Equation \eqref{eq:stratum_boundary_input_skip}, so these correspond to the boundary. The strata in Equations \eqref{eq:boundary_moduli_0_input_max}-\eqref{eq:boundary_moduli_0_input_skip} correspond to Equations \eqref{eq:boundary_moduli_inp_cont_input}-\eqref{eq:stratum_boundary_input_pullback_max_min}.
\end{proof}

There are parallel results for moduli spaces with outputs:
\begin{equation}
   \Mbar_{q, \vK; \out}(x(q) ,x_{\out} ; x'(q))\end{equation}
is the fibre product of $  \Mbar_{q,K^\out_{K_0}}(x'(q); x_{\out}, x(q))$ with $\Adams_{\vK; \out}$ over $\Adams_{K^\out_{K_0}} $, and is defined whenever $q \in P_{\min \vK}$. Given a totally ordered subset  $I$ of $\Sigma$,  form the union
\begin{equation} \label{eq:moduli_out_cells}
   \Mbar_{q,I; \out}(x(q) ,x_{\out} ; x'(q)) = \bigcup_{I=\min \vK} \Mbar_{q, \vK; \out}(x(q) ,x_{\out} ; x'(q)).
\end{equation}
As this moduli space is independent of $q$,  denote it $
   \Mbar_{I; \out}(x ,x_{\out} ; x')$.

For generic data, the interior is a manifold of dimension
\begin{equation}
 \deg(x)  + \deg(x_{\out}) - \deg(x')+ n + 1-|I|.
\end{equation}
\begin{lem}
If this dimension is $1$, then $ \Mbar_{I; \out}(x ,x_{\out} ; x') $ is a $1$-dimensional manifold with boundary
\begin{align} \label{eq:boundary_moduli_out_strip}
& \coprod_{y_{\out} \in  L \cap L'} \Mbar(x_{\out},y_{\out}) \times \Mbar_{I;\out}(x, y_{\out}; x') \\
& \coprod_{\max J = \min I} \coprod_{y \in L_{\min J} \cap F_{{\min J} }}  \Mbar_{J}(x,y) \times \Mbar_{J \cup I;\out}( y, x_{\out}; x')\\ 
& \coprod_{\min J = \max I } \coprod_{y'  \in L_{\max J} \cap F_{{\max J} }}  \Mbar_{I \cup J;\out}(x, x_{\out}; y') \times \Mbar_{J}(y',x') \\ \label{eq:boundary_moduli_out_skip} 
& \coprod_{\substack{ I \cup \{ j\} \\ \min I < j < \max I, j \notin I } }   \Mbar_{I \cup \{ j\} ;\out}(x, x_{\out}, x'). \qed
\end{align} 
\end{lem}

\section{Moduli space of degenerate annuli}
\label{sec:moduli-space-degen}

\subsection{Dual and pair subdivisions}

As in Sections \ref{sec:covers-q} and  \ref{sec:unif-small-choic}, let $\Sigma$ denote a partially ordered set labelling a simplicial triangulation of $Q$. In particular, every maximal totally ordered subset of $\Sigma$ consists of $n+1$ elements. Given a totally ordered subset $J$ of $\Sigma$, denote by $\sigma_{J}$ the corresponding simplex included in $Q$. The barycenters of the top dimensional simplices are the vertices of the dual polyhedral subdivision (which is not in general a triangulation), corresponding to the partially order set $\Sigma^{\vee}$ which is obtained from  $\Sigma$ by reversing the order; write $\check{\sigma}_{J}$ for the cell dual to $\sigma_{J}$. This cell can be realised as the polyhedron associated to the partially ordered set $  \{ K | J \subset K \} $ equipped with the ordering which reverses inclusions (i.e. $J$ is the unique maximal element of this ordering, and corresponds to the top dimensional face). In particular,
\begin{equation}
  \partial \check{\sigma}_{I} = \bigcup_{\substack{I \subset J \\ |J| = |I|+1}}  \check{\sigma}_{J}.
\end{equation}

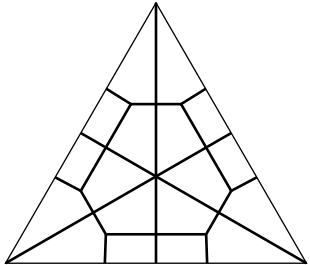
\begin{figure}
\begin{tikzpicture}[scale=0.50]
\coordinate (0) at (-4,0);
\coordinate (1) at (4,0);
\coordinate (2) at (0,6.9282);

\foreach \x in {0,1,2} \fill (\x) circle (2*\lw);
\draw[line width=2*\lw] (0)--(1)--(2)--cycle;

\coordinate (01) at ($(0) !.5! (1)$);
\coordinate (12) at ($(2) !.5! (1)$);
\coordinate (02) at ($(0) !.5! (2)$);
\coordinate (012) at ($1/3*(0)+1/3*(1)+1/3*(2)$);

\coordinate (0-01) at ($(0) !.66! (01)$);
\coordinate (0-02) at ($(0) !.66! (02)$);
\coordinate (1-01) at ($(1) !.66! (01)$);
\coordinate (2-02) at ($(2) !.66! (02)$);
\coordinate (1-12) at ($(1) !.66! (12)$);
\coordinate (2-12) at ($(2) !.66! (12)$);
\coordinate (0-01-012) at ($1/3*(0)+1/3*(01)+1/3*(012)$);
\coordinate (0-02-012) at ($1/3*(0)+1/3*(02)+1/3*(012)$);
\coordinate (1-01-012) at ($1/3*(1)+1/3*(01)+1/3*(012)$);
\coordinate (2-02-012) at ($1/3*(2)+1/3*(02)+1/3*(012)$);
\coordinate (1-12-012) at ($1/3*(1)+1/3*(12)+1/3*(012)$);
\coordinate (2-12-012) at ($1/3*(2)+1/3*(12)+1/3*(012)$);


\draw[line width=4*\lw] (0-01-012)--(0-02-012)--(2-02-012)--(2-12-012)--(1-12-012)--(1-01-012)--cycle;
\draw[line width=4*\lw] (0-01-012)--(0-01);
\draw[line width=4*\lw] (0-02-012)--(0-02);
\draw[line width=4*\lw] (1-01-012)--(1-01);
\draw[line width=4*\lw] (2-02-012)--(2-02);
\draw[line width=4*\lw] (1-12-012)--(1-12);
\draw[line width=4*\lw] (2-12-012)--(2-12);

\coordinate (P0-012) at ($1/2*(0-01-012)+1/2*(0-02-012)$);
\coordinate (P1-012) at ($1/2*(1-01-012)+1/2*(1-12-012)$);
\coordinate (P2-012) at ($1/2*(2-12-012)+1/2*(2-02-012)$);
\coordinate (P01-012) at ($1/2*(0-01-012)+1/2*(1-01-012)$);
\coordinate (P12-012) at ($1/2*(1-12-012)+1/2*(2-12-012)$);
\coordinate (P02-012) at ($1/2*(0-02-012)+1/2*(2-02-012)$);

\draw[line width=4*\lw] (0)--(012);
\draw[line width=4*\lw] (01)--(012);
\draw[line width=4*\lw] (02)--(012);
\draw[line width=4*\lw] (1)--(012);
\draw[line width=4*\lw] (12)--(012);
\draw[line width=4*\lw] (2)--(012);
\end{tikzpicture}

\caption{The pairs barycentric subdivision of a simplex.}
\label{fig:dual_pairs_barycentric}
\end{figure}

The minimal common refinement of a triangulation and its dual subdivision is called the pairs subdivision (as far as the author knows, the terminology is due to Denis Sullivan, though the concept has appeared earlier \cite{PS}). Combinatorially,  the pairs subdivision $P\Sigma$ of $\Sigma$ is the partially ordered set whose elements are nested pairs $I \subset K $ of totally ordered subsets of $\Sigma$. The cells associated to totally ordered subsets of $P \Sigma$ correspond to pairs of intersecting cells associated to totally ordered subsets of $\Sigma$ and $\Sigma^{\vee}$. Moreover, these cells are cubes: this can be seen algebraically by noting that the maximal elements of $P \Sigma$ are pairs $\{ i\} \subset K$, where $K$ has length $n+1$: the corresponding cell of $Q$ is associated to the set of elements preceding this pair, which consists of all pairs $I \subset J$ contained in $K$ and containing $i$. This is exactly the totally ordered set associated to the cube on the set $K \setminus \{ i\}$ as such a pair $ I \subset J  $ with $ i \in I$ and $J \subset K$ corresponds to the face in which all coordinates labelled by $I \setminus \{ i\}$ equal $1$, while those corresponding to $K \setminus J$ vanish.

There are also geometric proofs that the intersection of dual cells is a cube: embed the $n$-simplex $\Delta_n$ in $\bR^{n}$ as the convex hull of $(-1, \ldots, -1)$, and the set of points with one coordinate equal to $n$, and all others equal to $-1$. Let $\Pi \subset \bR^n$ denote the simplicial complex whose unique vertex is the origin, whose $1$-skeleton consists of the negative coordinate axes and the ray spanned by $(1, \ldots, 1)$, and whose top-dimensional cells are the octants spanned by any choice of $n$ of these rays. As subcomplexes of $\bR^{n}$, $\Delta_n$ and $\Pi $ are dual, and the intersection of the simplex with the octant spanned by the coordinate axes is obviously a cube. Another proof appears in Shtan{\cprime}ko and Shtrogrin \cite[Proof of main theorem]{SS-cubic}.

\subsection{The pairs barycentric subdivision}
One motivation for introducing the pairs subdivision is the following: assume that one has moduli spaces $Z_{I; \inp}$ parametrised by cells of a triangulation of $\Q$, and $Z_{I, \out}$ parametrised by cells of the dual subdivision, and one wanted to interpret the products $ Z_{I; \inp} \times Z_{I; \out}$ as moduli spaces parametrised by cells of a subdivision of $\Q$. The \emph{dual pairs subdivision} is then the natural choice, as its top-dimensional simplices are naturally products of cells and their duals (this fact is mentioned only for motivation, and will not be used). In Section \ref{sec:moduli-spac-assoc}, we constructed moduli spaces $ \Mbar_{I; \inp} $ and $ \Mbar_{I; \out}$ labelled by cells of the triangulation of $\Q$. It turns out that these spaces do not quite fit within this framework, but this can be remedied by a further refinement.

Let $B \Sigma$ denote the barycentric subdivision of $\Sigma$ whose elements are totally ordered subsets, and whose partial ordering is given by inclusion. Denote a totally ordered subset of $B \Sigma$ by $\vI$, and the corresponding simplex by $\sigma_{\vI}$.

Let $P B \Sigma$ denote the pairs subdivision of the barycentric subdivision, which will henceforth be called the \emph{pairs barycentric subdivison}. Cells correspond to pairs $
  \vI \subset \vJ $ where $\vJ$ is a totally ordered subset of $\Sigma$. Write $\sigma_{\vI \subset \vJ}$ for the corresponding cube embedded in $\Q$. We shall presently see that the cells of $P B \Sigma $ naturally parametrise the products of moduli spaces $ \Mbar_{\vK; \inp}  $ and $ \Mbar_{\vK; \out} $ and their boundary strata.

Let $I_0$ and $J_0$ (respectively $I$ and $J$) denote the minimal (respectively maximal) elements of $\vI$ and $\vJ$.  Recalling the definition of the moduli space $\Adams_K$ and its variants as a product of copies of $[0,+\infty]$ (see e.g. Equation \eqref{eq:Adams_moduli_product}), an identification of the interval $[0,1]$ with $[0, \infty]$ yields natural maps
\begin{equation} \label{eq:inp_pairs}
  \xymatrix{ \Adams_{\vJ[\geq]_{I}; \out}& \sigma_{\vI \subset \vJ} \ar[r] \ar[l] &\Adams_{\vJ[\leq]_{I_0}; \inp} }
\end{equation}
where the ordered sets $\vJ[\leq]_{I_0}$ and  $\vJ[\geq]_{I}$ respectively consist of elements of $\vJ$ which are contained in $I_0$ and which contain $I$.   At the level of partially ordered sets the first map takes a pair of subsets $\vI[1] \subset \vJ[1] $ to their intersections with $\vJ[\geq]_{I}$, and the second to their intersection with $\vJ[\leq]_{I_0}$:
\begin{equation}
 \xymatrix{ \vI[1,\geq]_{I} \subset \vJ[1,\geq]_{I} & \vI[1] \subset \vJ[1] \ar[l] \ar[r] & \vI[1,\leq]_{I_0} \subset \vJ[1,\leq]_{I_0} .}
\end{equation}

\begin{lem}
  If $\vI = \{ I\}$  the product map is a bijection:
\begin{equation} \label{eq:map_pairs_barycentric_moduli}
  \sigma_{\vI \subset \vJ} \to \Adams_{\vJ[\geq]_{I}; \out}  \times  \Adams_{\vJ[\leq]_{I}; \inp}.
\end{equation}
\end{lem}
\begin{proof}
By construction, $ \sigma_{\vI \subset \vJ} $ is the cube on $\vJ \setminus I$, whereas $ \Adams_{\vJ[\leq]_{I}; \inp}$ and  $\Adams_{\vJ[\geq]_{I}; \out}$ are respectively the cubes on $ \vJ[\leq]_{I} \setminus I$ and $\vJ[\geq]_{I} \setminus I$ (we have used the fact that $I_0=I$,). By assumption, every element of $\vJ$ may be uniquely written as the union of an element of $\vJ[\leq]_{I}$ and $ \vJ[\geq]_{I}  $. Equation  \eqref{eq:map_pairs_barycentric_moduli}  realises the induced product decomposition
\begin{equation}
  [0,1]^{\vJ \setminus I  } \cong  [0,1]^{ \vJ[\leq]_{I} \setminus I  } \times [0,1]^{\vJ[\geq]_{I} \setminus I }.
\end{equation}
\end{proof}
\begin{figure}
\centering
  \begin{tikzpicture}
\node at (5*\mw pt, 0)  {$z_\inp$};
    \CtoEBent{0}{0}{\min I}{\max I};
\node[blue] at (1*\mw pt, 0)  {$\Adams_{I^{\inp}_{J_0}}$};   
\node[red] at (-3*\mw pt, 0)  {$ \Adams_{J^{\out}_{I}}  $};
   \EtoCBent{-2*\mw}{0};
\node at (-7*\mw pt, 0)  {$z_\out$};
  \end{tikzpicture}
\caption{A fibre of $\Ann^{\infty}(\sigma_{\{ I\} \subset \vJ}) $, with curves labelled by the moduli space from which they are obtained by pullback.}
\label{fig:fibre_infinite_annulus_broken-top-dim}
\end{figure}
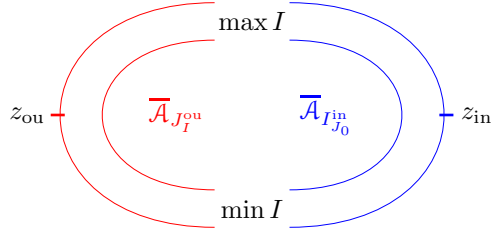
Composing Equation \eqref{eq:inp_pairs} with the maps defined in Equations \eqref{eq:barycentric_to_inp} and \eqref{eq:barycentric_to_out} yields maps
\begin{equation} \label{eq:cell_pair_mu_in_out}
\xymatrix{ \Adams_{J^{\out}_{I}} & \ar[l]_{ \mu_{\out} } \ar[r]^{\mu_{\inp}}  \sigma_{\vI \subset \vJ} & \Adams_{{I_0}^{\inp}_{J_0}} .}
\end{equation}
Taking the union of the pullbacks of the universal curves on these spaces, we obtain a moduli space of Riemann surfaces over $\sigma_{\{ I \} \subset \vJ} $ which we denote $ \Ann^{\infty}(\sigma_{\{ I\} \subset \vJ}) $, and call the \emph{moduli space of degenerate annuli} parametrised by $ \sigma_{\{ I \} \subset \vJ} $; see Figure \ref{fig:fibre_infinite_annulus_broken-top-dim}.    

Note that all top-dimensional strata of the pairs barycentric subdivision are of this form. In the next section, we shall show that this construction is appropriately compatible across boundary strata, i.e. that we obtain a \emph{moduli space of degenerate annuli} parametrised by $ \Q $.
\subsection{Degenerate annuli over a cell of the pairs barycentric subdivision}
Given a pair $\vI \subset \vJ$, define
\begin{align}
  \min_{\vI}(\vJ) & = \{ \min J_1 | I_0 \subset J_1 \subset I, \textrm{ and } J_1 \in \vJ \} \\
\max_{\vI}(\vJ) & = \{ \max J_1 | I_0 \subset J_1 \subset I, \textrm{ and } J_1 \in \vJ \}.
\end{align}

For each cell of the pairs barycentric subdivision, this defines maps
\begin{equation} \label{eq:min_pairs}
\xymatrix{ \Adams_{I^{\geq}_{I_0}} & \sigma_{\vI \subset \vJ}  \ar[r]^{\min} \ar[l]_{\max}& \Adams_{I^{\leq}_{I_0}} }\\ 
\end{equation}
which, at the level of partially ordered sets, are given by
\begin{equation}
\xymatrix{  \max_{\vI}(\vI_1) \subset  \max_{\vI}(\vJ_1)  & \ar@{|->}[r] \ar@{|->}[l]  \vI_1 \subset \vJ_1 & \min_{\vI}(\vI_1) \subset  \min_{\vI}(\vJ_1)}.
\end{equation}
Define  the universal curve $ \Ann^{\infty}_{\vI \subset \vJ} $ over
 \begin{equation}
\Adams^{\infty}_{\vI \subset \vJ} \equiv    \Adams_{J^{\out}_{I}} \times \Adams_{I^{\geq}_{I_0}} \times \Adams_{I^{\leq}_{I_0}} \times \Adams_{{I_0}^{\inp}_{J_0}}
 \end{equation}
to be the union of the pullbacks of the universal curves of each factor.

\begin{defin} \label{def:universal_deg_annuli_pairs_subdiv}
 The \emph{universal family of degenerate annuli} $\Ann^{\infty}(\sigma_{\vI \subset \vJ})$ over $\sigma_{\vI \subset \vJ}  $  is the pullback of $\Ann^{\infty}_{\vI \subset \vJ} $  by the product of $\mu_\out$, $\max$, $\min$, and $\mu_\inp$.
\end{defin}
The reader should consult Figure \ref{fig:fibre_infinite_annulus_broken} which shows how to interpret the fibre of $\Ann^{\infty}(\sigma_{\vI \subset \vJ}) $ as a degenerate annulus for an example with $\vI$ consisting of four nested sets. 

Given a pair of strata $ \sigma_{\vI[1] \subset \vJ[1]} \subset \sigma_{\vI \subset \vJ} $, there is a commutative diagram for the restrictions of $\mu_\inp$ and $\mu_\out$:
\begin{equation} \label{eq:diagram_moduli_space_inp_out_cubes}
  \xymatrix{\Adams_{I^{1,\inp}_{0,J^1_0}}  & \ar@{->>}[l] \Adams_{I^{\geq}_{0,I^1_0}} \times \Adams_{I^{1,\inp}_{0,J^1_0}}  \times \Adams_{I^{\leq}_{0,I^1_0}} \ar@{^{(}->}[r] & \Adams_{I^{\inp}_{0,J_0}} \\
& \sigma_{\vI[1] \subset \vJ[1]} \ar@{^{(}->}[r] \ar[d] \ar[u] \ar[ul]^{\mu_\inp} \ar[dl]_{\mu_\out} &  \sigma_{\vI \subset \vJ}  \ar[u]^{\mu_\inp} \ar[d]_{\mu_\out} \\
\Adams_{J^{1,\out}_{I^1}}  & \Adams_{I^{1,\leq}_{I}} \times \Adams_{J^{1,\out}_{I^1}}  \times \Adams_{I^{1,\geq}_{I}} \ar@{^{(}->}[r] \ar@{->>}[l] &  \Adams_{J^{\out}_{I}}.
}
\end{equation}
To help the reader navigate the above diagram, note that the upward pointing arrow labelled $\mu_\inp$ is the map appearing in Equation \eqref{eq:cell_pair_mu_in_out}. Since the middle and top horizontal maps are inclusions, the unlabelled upward pointing arrow is the restriction of $\mu_\inp $ to the boundary stratum $\sigma_{\vI \subset \vJ} $. This stratum is equipped with its own map $\mu_\inp$ coming from Equation \eqref{eq:cell_pair_mu_in_out}, which is the left-upward pointing arrow. The bottom half of the diagram is determined in the same way, \emph{mutatis mutandis.}

On the other hand, the maps $\min$ and $\max$ fit into the diagram
\begin{equation} \label{eq:diagram_moduli_space_max_min_cubes}
  \xymatrix{\Adams_{I^{1,\geq}_{I^1_0}} & \Adams_{I^{\geq}_{0,I^1_0}}  \times   \Adams_{I^{\geq}_{I_0}}  \times \Adams_{I^{1,\geq}_{I}} \ar@{->>}[r] \ar@{_{(}->}[l] & \Adams_{I^{\geq}_{I_0}} \\
& \sigma_{\vI[1] \subset \vJ[1]} \ar@{^{(}->}[r] \ar[d] \ar[u]  \ar[ul]^{\max} \ar[dl]_{\min} &  \sigma_{\vI \subset \vJ}  \ar[u]^{\max} \ar[d]_{\min} \\
\Adams_{I^{1,\leq}_{I^1_0}}  &   \Adams_{I^{1,\leq}_{I}}  \times \Adams_{I^{\leq}_{I_0}}   \times  \ar@{->>}[r] \ar@{_{(}->}[l] \Adams_{I^{\leq}_{0,I^1_0}}  & \Adams_{I^{\leq}_{I_0}} .
}
\end{equation}
One can navigate this diagram in essentially the same way as the previous one: the left-upward pointing map is the map $\max$ from Equation \eqref{eq:min_pairs}, applied to the boundary stratum $ \sigma_{\vI[1] \subset \vJ[1]} $. Since the top left pointing arrow is an inclusion, the commutativity of this diagram is equivalent to the claim that this map factors through the substratum $ \Adams_{I^{\geq}_{0,I^1_0}}  \times   \Adams_{I^{\geq}_{I_0}}  \times \Adams_{I^{1,\geq}_{I}}  $ of the target. This stratum projects to its middle factor, which is the image of the map $\max$ defined on $ \sigma_{\vI \subset \vJ}$. The commutativity of the top square should then be interpreted as the assertion that the map $\max$ defined on $ \sigma_{\vI \subset \vJ} $, when restricted to its boundary stratum $\sigma_{\vI[1] \subset \vJ[1]} $, factors through the image of the map $\max$ defined on this stratum. Similar considerations apply to the bottom part of the diagram.
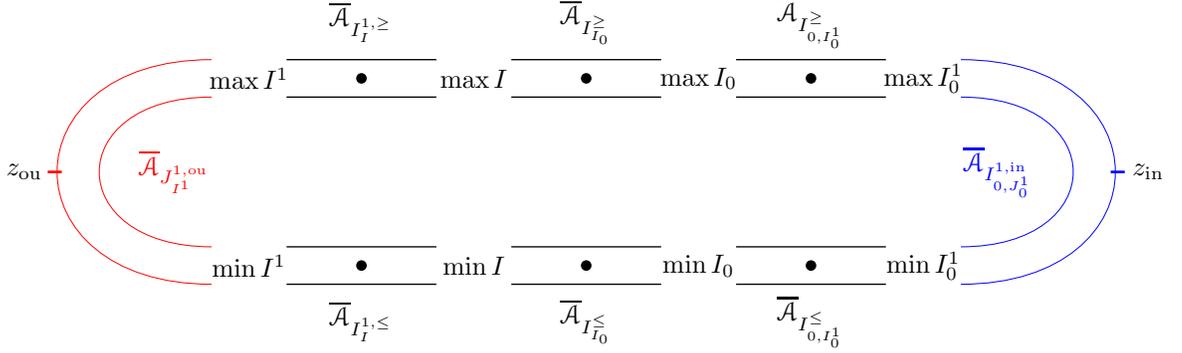
\begin{figure}
\centering
  \begin{tikzpicture}
\node at (5*\mw pt, 0)  {$z_\inp$};
    \CtoEBent{0}{0}{\min I^1_0}{\max I^1_0};
\node[blue] at (1*\mw pt, 0)  {$\Adams_{I^{1,\inp}_{0,J^1_0}}$};
\ConoutBig{-2*\mw pt }{2.5*\mw pt}{\max I_0};
\node at (-4*\mw pt, 4*\mw pt) {$ \Adams_{I^{\geq}_{0,I^1_0}} $}; 
\ConoutBig{-2*\mw pt }{-2.5*\mw pt}{\min I_0};
\node at (-4*\mw pt, -4*\mw pt) {$ \Adams_{I^{\leq}_{0,I^1_0}} $}; 
\ConoutBig{-8*\mw pt }{2.5*\mw pt}{\max I};
\node at (-10*\mw pt, 4*\mw pt) {$ \Adams_{I^{\geq}_{I_0}} $};
\ConoutBig{-8*\mw pt }{-2.5*\mw pt}{\min I};
\node at (-10*\mw pt, -4*\mw pt) {$ \Adams_{I^{\leq}_{I_0}} $};
\ConoutBig{-14*\mw pt }{2.5*\mw pt}{\max I^1};
\node at (-16*\mw pt, 4*\mw pt) {$   \Adams_{I^{1,\geq}_{I}} $};
\ConoutBig{-14*\mw pt }{-2.5*\mw pt}{\min I^1};
\node at (-16*\mw pt, -4*\mw pt) {$\Adams_{I^{1,\leq}_{I}}  $};
\node[red] at (-21*\mw pt, 0)  {$ \Adams_{J^{1,\out}_{I^1}}  $};
   \EtoCBent{-20*\mw}{0};
\node at (-25*\mw pt, 0)  {$z_\out$};
  \end{tikzpicture}
\caption{A fibre of $\Ann^{\infty}(\sigma_{\vI[1] \subset \vJ[1]}) $ for $\vI[1] = \{I^1_0 \subset I_0 \subset I \subset I^1 \}$, with curves labelled by the moduli space from which they are obtained by pullback.}
\label{fig:fibre_infinite_annulus_broken}
\end{figure}
To complete our analysis of the comparison between the fibres of $\Ann^{\infty}(\sigma_{\vI[1] \subset \vJ[1]}) $ over different strata, define
\begin{equation}
 \partial_{\vI[1] \subset \vJ[1]} \Adams^{\infty}_{\vI \subset \vJ} \equiv  \Adams_{J^{1,\out}_{I^1}}  \times \Adams_{I^{1,\geq}_{I}} \times \Adams_{I^{\geq}_{I_0}} \times \Adams_{I^{\geq}_{0,I^1_0}}\times  \Adams_{I^{\leq}_{0,I^1_0}} \times\Adams_{I^{\leq}_{I_0}} \times   \Adams_{I^{1,\leq}_{I}} \times \Adams_{I^{1,\inp}_{0,J^1_0}};
\end{equation}
inspection of Figure \ref{fig:fibre_infinite_annulus_broken} may help in decoding the notation.
\begin{lem}
The products of the maps  $\mu_\out$, $\max$, $\min$, and $\mu_\inp$ fit in a commutative diagram:
\begin{equation} \label{eq:decomposition_annuli_broken_over_boundary_cells}
  \xymatrix{ 
& \sigma_{\vI[1] \subset \vJ[1]} \ar@{^{(}->}[r] \ar[d] \ar[dl]  &   \sigma_{\vI \subset \vJ}   \ar[d]  \\
\Adams^{\infty}_{\vI[1] \subset \vJ[1]}   &  \partial_{\vI[1] \subset \vJ[1]} \Adams^{\infty}_{\vI \subset \vJ}  \ar@{^{(}->}[r] \ar@{_{(}->}[l] & \Adams^{\infty}_{\vI \subset \vJ}. 
}
\end{equation}
\end{lem}
\begin{proof}
The vertical maps are those arising in Definition \ref{def:universal_deg_annuli_pairs_subdiv}. The bottom horizontal maps are natural inclusions, which we make explicit for the benefit of the reader: for the left pointing arrow, we factor $ \partial_{\vI[1] \subset \vJ[1]} \Adams^{\infty}_{\vI \subset \vJ}  $ as
\begin{equation}
    \Adams_{J^{1,\out}_{I^1}}  \times \left( \Adams_{I^{1,\geq}_{I}} \times \Adams_{I^{\geq}_{I_0}} \times \Adams_{I^{\geq}_{0,I^1_0}} \right) \times  \left( \Adams_{I^{\leq}_{0,I^1_0}} \times\Adams_{I^{\leq}_{I_0}} \times   \Adams_{I^{1,\leq}_{I}} \right) \times \Adams_{I^{1,\inp}_{0,J^1_0}}.
\end{equation}
The terms at the two ends appear in Diagram \eqref{eq:diagram_moduli_space_inp_out_cubes}, and are factors of $\Adams^{\infty}_{\vI[1] \subset \vJ[1]}  $. Up to reordering, the two ``blocks'' of three terms appear in Diagram \eqref{eq:diagram_moduli_space_max_min_cubes}, and include in the remaining two factors of $\Adams^{\infty}_{\vI[1] \subset \vJ[1]}$.

For the right pointing arrow,   we factor $ \partial_{\vI[1] \subset \vJ[1]} \Adams^{\infty}_{\vI \subset \vJ}  $ as
\begin{equation} \label{eq:factorisation_boundary_annuli_into_bigger}
       \underbrace{ \Adams_{J^{1,\out}_{I^1}}  \times  \Adams_{I^{1,\geq}_{I}}}  \times \Adams_{I^{\geq}_{I_0}} \times  \overbrace{ \Adams_{I^{\geq}_{0,I^1_0}}  \times \Adams_{I^{\leq}_{0,I^1_0}} } \times  \Adams_{I^{\leq}_{I_0}} \times  \underbrace{  \Adams_{I^{1,\leq}_{I}} } \times \overbrace{ \Adams_{I^{1,\inp}_{0,J^1_0}}}.
\end{equation}
The three terms above underbraces, as well as those below overbraces, appear in  Diagram \eqref{eq:diagram_moduli_space_inp_out_cubes} and include into two of the factors of $\Adams^{\infty}_{\vI \subset \vJ} $. The remaining factors of $\Adams^{\infty}_{\vI \subset \vJ} $ are those which are not labelled at all in Equation \eqref{eq:factorisation_boundary_annuli_into_bigger}. 

Having described all the maps in Diagram \eqref{eq:factorisation_boundary_annuli_into_bigger}, its commutativity is an immediate consequence of the commutativity of Diagrams \eqref{eq:diagram_moduli_space_inp_out_cubes} and \eqref{eq:diagram_moduli_space_max_min_cubes}.
\end{proof}

By construction, the universal curves over the left and right products of Adams path spaces in the bottom row of Diagram \eqref{eq:decomposition_annuli_broken_over_boundary_cells}  restrict to the natural universal curve in the product in the middle. In particular:
\begin{lem} \label{lem:deg_annuli_agree}
  There is a natural identification
  \begin{equation}
    \Ann^{\infty}(\sigma_{\vI[1] \subset \vJ[1]}) \cong  \Ann^{\infty}(\sigma_{\vI \subset \vJ})| \sigma_{\vI[1] \subset \vJ[1]}. \qed
  \end{equation} 
\end{lem}

The above result implies the existence of a family of degenerate annuli over $\Q$
\begin{equation}
   \Ann^{\infty}(\Q) \to \Q
\end{equation}
whose restriction to each top-dimensional cell of the pairs barycentric subdivision labelled by $\{ I  \} \subset \vJ  $ admits a natural isomorphism
\begin{equation}
  \label{eq:annuli_over_top_cells}
  \Ann^{\infty}(\sigma_{\{ I  \} \subset \vJ} ) \cong \mu^*_\out    \Ubar_{J^{\out}_{I}} \cup \mu^*_\inp  \Ubar_{{I}^{\inp}_{J_0}}.
\end{equation}
The boundaries of the fibres will be denoted
\begin{align}
   \Ann^{\{ 0 \} , \infty}_{\Q}  =\mu^*_\out    \Ubar^{\{ 0 \} }_{J^{\out}_{I}} \cup \mu^*_\inp  \Ubar^{\{ 0 \} }_{{I}^{\inp}_{J_0}}  & \textrm{ and }
   \Ann^{\{ 1 \} , \infty}_{\Q}  =\mu_\out^*   \Ubar^{\{ 1 \} }_{J^{\out}_{I}}    \cup \mu_\inp^* \Ubar^{\{1\}}_{{I}^{\inp}_{J_0}} ,
\end{align}
with the restriction to $q \in \Q$ denoted $  \Ann^{\{ 0 \} , \infty}_{q} $ and $\Ann^{\{ 1 \} , \infty}_{q}   $.

The marked points $z_{\inp}$ and $z_{\out}$ on the fibres of $ \Ubar^{\{1\}}_{{I}^{\inp}_{J_0}}$ and $ \Ubar^{\{ 1 \} }_{J^{\out}_{I}} $ over $\Adams_{{I}^{\inp}_{J_0}}  $ and  $  \Adams_{J^{\out}_{I}} $ yield marked points on $  \Ann^{\{ 1 \} , \infty}_{q}  $ for every $q \in \Q$ which induce a decomposition
\begin{equation}
   \Ann^{\{ 1 \} , \infty}_{\Q} =  \Ann^{z_{\inp} < z_{\out} , \infty}_{\Q} \cup \Ann^{z_{\out} < z_{\inp} , \infty}_{\Q}
\end{equation}
with the first component being the union of the pullbacks of $  \Ubar^{z_\inp < }_{{I}^{\inp}_{J_0}} $ and  $ \Ubar^{z_{\out} > }_{J^{\out}_{I}}  $ while the second component is the union of the pullbacks of $ \Ubar^{z_\inp > }_{{I}^{\inp}_{J_0}}  $ and  $ \Ubar^{z_{\out} < }_{J^{\out}_{I}}  $.

\subsection{Floer data on degenerate annuli} \label{sec:floer-data-degen} On each top-dimensional cell of the pairs barycentric subdivision, consider the map $\Phi^{\infty}_{ \{ I\} \subset \vJ }   $ defined as the composition
\begin{equation} \label{eq:map_degenerate_annuli_floer}
  \xymatrix{  \Ann^{ \infty}(\sigma_{\vI \subset \vJ}) = \mu_\inp^*  \Ubar_{{I}^{\inp}_{J_0}}  \cup \mu_\out^* \Ubar_{J^{\out}_{I}}   \ar@{^{(}->}[d] &   \scrH^0 \times \scrJ^{n+3} \times \scrD^{2^{n+2}} \\
\sigma_{\{I\} \subset \vJ} \times \Ubar_{{I}^{\inp}_{J_0}}   \cup \sigma_{\{I\} \subset \vJ} \times \Ubar_{J^{\out}_{I}}  \ar@{^{(}->}[r]  & P_{\max I} \times \Ubar_{{I}^{\inp}_{J_0}}    \cup   P_{\max J} \times  \Ubar_{J^{\out}_{I}}  \ar[u]^{ (\Phi^{P}_{{I}^{\inp}_{J_0} }, \Phi^{P}_{J^{\out}_{I} )}  }.  }
\end{equation}
The horizontal arrow is induced by the inclusions of $ \sigma_{\{I\} \subset \vJ} $  in $ P_{\max I}$ and $  P_{\max J}  $, arising from the fact that $\sigma_{\{I\} \subset \vJ} $, as a cell of the pairs barycentric subdivision is contained in the open star of the vertices labelled by $ \max I $ and $\max J$, and our assumption in Section \ref{sec:floer-theory-conv} that the open star of every vertex be contained in the corresponding polygon. To describe the bottom pointing arrow, recall that $ \mu_\inp^*  \Ubar_{{I}^{\inp}_{J_0}} $ is the space over $\sigma_{\{I\} \subset \vJ}  $ obtained by pulling back $  \Ubar_{{I}^{\inp}_{J_0}}  \to  \Adams_{{I}^{\inp}_{J_0}} $.  The restriction of the left vertical map to the first component is then the product of the projection to the base and the natural map to $  \Ubar_{{I}^{\inp}_{J_0}}  $. The second component of the left vertical map is given in the same way.

Over a codimension $1$ stratum, the restrictions of the maps $ \Phi^{\infty}_{ \{ I\} \subset \vJ } $  agree under the natural identifications; on the first component, this follows from Equations \eqref{eq:continuation_inp_skip_-}-\eqref{eq:restrict_J_in_stop}, while on the second this follows from the analogous conditions imposed in Equations \eqref{eq:cont_out_skip_-}-\eqref{eq:restrict_J_out_stop}. Taking the union over all top-dimensional cells, we obtain a map
\begin{equation}
  \Phi^{\infty}_{\Q} = (\phi^{\infty}_{\Q}, J^{\infty}_{\Q}, \psi^{\infty}_\Q)  \co \Ann^{\infty}(\Q)  \to \scrH^0 \times  \scrJ^{n+3} \times \scrD^{2^{n+2}}.
\end{equation}

Given intersection points $x_\inp ,x_\out \in L \cap L'$, define
\begin{equation}
  \cM^{\infty}_{\vI \subset \vJ }(x_\out; x_\inp)
\end{equation}
to be the moduli space of finite energy maps from a fibre of $\Ann^{\infty}(\sigma_{\vI \subset \vJ })$ to $X$ such that
\begin{align} 
 \partial_{s} u(z) & = J^{\infty}_{\Q}(z) \partial_{t}u(z) & u(z) & \in  \phi^{\infty}_{\Q}(z) (L) \textrm{ if } z \in  \Ann^{z_{\inp} < z_{\out}, \infty}(\sigma_{\vI \subset \vJ } ) \\ 
u(z) & \in F_{q}\textrm{ if } z \in \Ann^{\{0\}, \infty}(q)  & u(z) & \in  \phi^{\infty}_{\Q}(z) (L') \textrm{ if } z \in  \Ann^{z_{\out} < z_{\inp}, \infty }( \sigma_{\vI \subset \vJ }),
\end{align}
and which converge to $x_\inp$ and $x_\out$ at the ends, and with compatible convergence to intersection points along ends (i.e. each node of the corresponding nodal Riemann surface is labelled by an intersection point of Lagrangians).

\begin{rem}
The above moduli space represents a significant change in perspective from previous constructions wherein the labels on any moduli space determined a basepoint in $\Q$  whose corresponding fibre served as the boundary condition for all elements of the moduli space. The results of Section \ref{sec:floer-theory-conv} ensure that we can choose any fibre in a fixed neighbourhood of the basepoint, and obtain a bijective correspondence between the moduli spaces defined for different fibres. The new point of view is that, using the fact that the moduli spaces we consider are \emph{parametrised moduli spaces,} we can choose any map from the parameter space to the space of fibres over which the given Floer data are defined, and introduce a new moduli space, with elements in bijective correspondence to the old as a fibre product.

For specificity, let us consider the case of the moduli space $\Mbar_{I}(x, y) $. The maps $ \Psi_I $ and $\Phi_I$ chosen in Section \ref{sec:locally-const-famil} allow us to define Floer data $  { \Psi_I} _*(\Phi_I) $ parametrised by  $  P_{I} \times \Ubar_{I} $. If we fix a point  $q \in P_I$, then these data, together with the Lagrangian boundary condition $F_q$ along $   \Ubar^{\{0\}}_{I} $, yield moduli spaces $\Mbar_{q,I}(x, y)  $, which are intertwined by homeomorphisms induced by $\Psi_I$. In this example, the new point of view is to instead pick a map  $ \Adams_I \to P_I$, and consider the composition of $ { \Psi_I} _*(\Phi_I)  $ with the map
\begin{equation}
   \Ubar_{I} \to  \Adams_I \times \Ubar_I \to P_I \times \Ubar_I.
\end{equation}
This gives a new family of Floer data on $\Ubar_I$, which again yields a moduli space  homeomorphic to $\Mbar_{I}(x, y) $, but now with  the boundary  along $   \Ubar^{\{0\}}_{I} $ prescribed according to the map $  \Adams_I \to P_I $ (i.e. over the line $  \Ubar^{\{0\}}_{\vr} $ over a point $\vr \in \Adams_I $, the Lagrangian boundary condition  is the fibre at the image of $\vr$ in $P_I$).
\end{rem}

\begin{lem} \label{lem:moduli_space_deg_annuli_product}
There is a natural homeomorphism
\begin{equation}
    \cM^{\infty}_{\{I\} \subset \vJ }(x_\out; x_\inp) \cong  \coprod_{\substack{ x \in F_I \cap L_I \\ x' \in F_I \cap L'_I}}  \cM_{q_I, \vJ[\geq]_{I}; \out} (x, x_\out,  x') \times  \cM_{q_I, \vJ[\leq]_{I}; \inp} (x', x_\out,  x).
\end{equation}
\end{lem}
\begin{proof}
By Equation \eqref{eq:map_pairs_barycentric_moduli}, $  \sigma_{\vI \subset \vJ} $  splits as a product of the spaces parametrising $  \cM_{q_I, \vJ[\geq]_{I}; \out}(x, x_\out,  x') $ and $\cM_{q_I, \vJ[\leq]_{I}; \inp}(x', x_\out,  x)$, and the universal curve $ \Ann^{\infty}(\sigma_{\{ I  \} \subset \vJ} )$ is obtained by pulling back the domains of these moduli spaces. Since the Cauchy-Riemann equations defining $ \cM^{\infty}_{\{I\} \subset \vJ } (x_\out; x_\inp)  $ are also obtained by pullback, we obtain an identification of moduli spaces of solutions after applying the family of homeomorphisms from Equation \eqref{eq:diffeo_moduli_varying_points} to the constituent moduli spaces of continuation maps.
\end{proof}

The virtual dimension of $  \cM^{\infty}_{\vI \subset \vJ }(x_\out; x_\inp)  $ is 
\begin{equation}
  \dim \sigma_{\vI \subset \vJ } - n +  \deg x_\out - \deg x_\inp.
\end{equation} 
Since the Floer data on each component of a fibre of $\Ann^{\infty}(\sigma_{\vI \subset \vJ }) $ were chosen generically in Sections \ref{sec:locally-const-famil}, \ref{sec:locally-const-cont-inp}, and \ref{sec:locally-const-famil-output}, this moduli space is empty whenever the virtual dimension is negative.  Assuming that $\deg x_\out = \deg x_\inp $, all moduli spaces are therefore empty, except those for which $\dim  \sigma_{\vI \subset \vJ } = n$, i.e. $\vI = \{ I\}$, and  $\vJ$ has length $n+1$. Define
\begin{equation}
   \cM^{\infty}_{I}(x_\out; x_\inp) \equiv \bigcup_{ I \in \vI  \subset \vJ }   \cM^{\infty}_{\vI \subset \vJ }(x_\out; x_\inp).
\end{equation}
Our regularity assumptions imply that this space has virtual dimension $\deg x_\out - \deg x_\inp$, and that it is stratified by the submanifolds  $\cM^{\infty}_{\vI \subset \vJ }(x_\out; x_\inp)$ whose codimension is equal to the codimension of $\sigma_{\vI \subset \vJ }$. Whenever  the virtual dimension vanishes, we conclude that the only contributions to this moduli space come from cells of the form $ \sigma_{\{ I \} \subset \vK}$, since the moduli spaces parametrised by all other cells have negative virtual dimension. In this case, there is a straightforward description of the moduli space:
\begin{lem} \label{lem:decomposition_degenerate_annuli_0-dim}
  If $\deg x_\out = \deg x_\inp $, there is a natural homeomorphism 
  \begin{equation}  \cM^{\infty}_{I}(x_\out; x_\inp) \cong \coprod_{\substack{x \in F_{I} \cap L \\ x' \in F_{I}  \cap L'}}  \coprod_{\gamma \in H_1(F_{q_I}, \bZ) }\Mbar_{q_I,   I; \out}(x, x_\out; x') \times \Mbar_{q_I, I; \inp}(x'; x_\inp, x).
\end{equation}
\end{lem}
\begin{proof}
  We apply Lemma \ref{lem:moduli_space_deg_annuli_product} to each such cell, and note that, according to Equations \eqref{eq:definition_moduli_in_cells} and \eqref{eq:moduli_out_cells}, we have
  \begin{equation}
        \Mbar_{q_I,   I; \out}(x, x_\out; x') \times \Mbar_{q_I, I; \inp}(x'; x_\inp, x) = \coprod_{  \substack{ I=\max \vI  \\ I = \min \vJ}} \Mbar_{q,\vI; \inp}(x'(q); x_{\inp}, x(q)) \times  \Mbar_{q, \vJ; \out}(x(q) ,x_{\out} ; x'(q)).
  \end{equation}
The result follows from the fact that the assignment $\vK = \vI \cup \vJ$ yields a bijective correspondence between pairs $(\vI,\vJ)$ such that $\max \vI = I = \min \vJ$ and sequences $\vK$ containing $I$.\end{proof}

We now restrict attention to the space of annuli whose restriction to the boundary fibre is a null-homologous circle:
\begin{equation}
  \cM^{\infty}_{[0],I}(x_\out ; x_\inp) \subset  \cM^{\infty}_{I}(x_\out; x_\inp).
\end{equation}
Denote by $\Mbar^{\infty}_{[0],I}(x_\out ; x_\inp) $ the Gromov-Floer compactification.  

In order to describe this decomposition in terms of the constituent curves, consider the natural decomposition 
\begin{align}
  \cM_{q_I,  I ; \out}(x, x_\out,  x') & \equiv \coprod_{\gamma \in H_1(F_{q_I}, \bZ)  }  \cM_{q_I,    I, \gamma; \out}(x, x_\out,  x')   \\
 \cM_{q_I,I; \inp}(x'; x_\inp, x) & \equiv \coprod_{\gamma \in H_1(F_{q_I}, \bZ)  } \cM_{q_I, I, \gamma; \inp}(x'; x_\inp, x),
\end{align}
where the components of the right hand side labelled by $\gamma$ consist of curves $u$ whose boundary along $q_I$ yields a circle in $F_{q_I}$ representing the homology class $\gamma$ after concatenation with the paths from $x$ and $x'$ to the basepoint on $F_I$, which we fix as in Section \ref{sec:areas-strips-flux}.

The following result is now an immediate consequence of Lemma \ref{lem:decomposition_degenerate_annuli_0-dim}: 
\begin{lem}
  If $\deg x_\out = \deg x_\inp $,  $\cM^{\infty}_{[0],I}(x_\out; x_\inp) $  is a $0$-dimensional manifold, which agrees with its Gromov-Floer compactification, and which is naturally homeomorphic to
  \begin{equation} \coprod_{\substack{x \in F_{I} \cap L \\ x' \in F_{I}  \cap L'}}  \coprod_{\gamma \in H_1(F_{q_I}, \bZ) }\Mbar_{q_I,   I, \gamma; \out}(x, x_\out; x') \times \Mbar_{q_I, I, \gamma; \inp}(x'; x_\inp, x).
  \end{equation} \qed
\end{lem}
Write $\cM^{\infty}_{[0]}(x_\out; x_\inp)  $ for the union of the moduli spaces $ \cM^{\infty}_{[0],I}(x_\out; x_\inp) $ over all cells $I$ of $\Sigma$.

\subsection{Gluing description of Floer data on degenerate annuli}  \label{sec:gluing-descr-floer} 
As in Lemma \ref{lem:deg_annuli_agree}, consider an inclusion of cells of the pairs barycentric subdvisision $
\sigma_{\vI[1] \subset \vJ[1]} \subset \sigma_{\vI \subset \vJ}$.
Composing the maps $\mu_\inp$ and $\mu_\out$ with projection to the factors labelled by elements which do not lie in  $ \min \vI[1]$ and $\max \vI[1]$  yields a map
\begin{equation} \label{eq:gluing_parameters_pairs_barycentric}
\sigma_{\vI \subset \vJ} \to [0,\infty]^{\min \vI[1] \setminus \min \vI} \times [0,\infty]^{\max \vI[1] \setminus \max \vI}.
\end{equation}
This map provides a gluing description of the restriction of $\Ann^{\infty}(\sigma_{\vI \subset \vJ})$ to a neighbourhood of $\sigma_{\vI[1] \subset \vJ[1]}$. Indeed, composing the map from $ \sigma_{\vI \subset \vJ} $ to $\Adams^{\infty}_{\vI \subset \vJ}$ with the maps in Equations \eqref{eq:project_map_adams} for each factor of the target defines a map $
\sigma_{\vI \subset \vJ} \to  \Adams^{\infty}_{\vI[1] \subset \vJ[1]}$.  After restricting to a neighbourhood $\nu_{\vI \subset \vJ} \sigma_{\vI[1] \subset \vJ[1]} $ of $  \sigma_{\vI[1] \subset \vJ[1]} $ in $ \sigma_{\vI \subset \vJ} $, we  denote the pullback of the universal curve under this map by
\begin{equation}
  \Ann^{\infty}_{\vI[1] \subset \vJ[1]}( \nu_{\vI \subset \vJ} \sigma_{\vI[1] \subset \vJ[1]} ) \to \nu_{\vI \subset \vJ} \sigma_{\vI[1] \subset \vJ[1]}.
\end{equation}

The gluing map in Equation \eqref{eq:gluing_maps_universal_curve_paths} yields a map of families of surfaces over $\nu_{\vI \subset \vJ} \sigma_{\vI[1] \subset \vJ[1]}  $ 
\begin{equation} \label{eq:gluing_map_for_annuli}
  \Ann^{\infty}_{ \vI[1] \subset \vJ[1] }(\nu_{\vI \subset \vJ} \sigma_{\vI[1] \subset \vJ[1]}) \to \Ann^{\infty}(\nu_{\vI \subset \vJ} \sigma_{\vI[1] \subset \vJ[1]}).
\end{equation}
Define the continuation data
\begin{equation} \label{eq:continuation_data_annuli_nbd_substratum}
 \Phi^{\infty}_{\vI[1] \subset \vJ[1] } = (\phi^{\infty}_{\vI[1] \subset \vJ[1] }, J^{\infty}_{\vI[1] \subset \vJ[1] }, \psi^{\infty}_{\vI[1] \subset \vJ[1]})  \co \Ann^{\infty}_{\vI[1] \subset \vJ[1]}(\sigma_{\vI \subset \vJ} ) \to \scrH^0 \times \scrJ^{n+3} \times \scrD^{2^{n+2}}
\end{equation}
to be the composition of $\Phi^{\infty}_\Q$ with the gluing map of Equation \eqref{eq:gluing_map_for_annuli}.

If $V_{\vI[1] \subset \vJ[1]}$ is a neighbourhood of $ \sigma_{\vI[1] \subset \vJ[1]} $ in $\Q$  whose intersection with every cell $ \sigma_{\vI \subset \vJ}$ lies in $\nu_{\vI \subset \vJ} \sigma_{\vI[1] \subset \vJ[1]} $, we define
\begin{equation} \label{eq:moduli_degenerate_annuli_nbd}
 \Ann^{\infty}_{\vI[1] \subset \vJ[1]  }(V_{\vI[1] \subset \vJ[1]}) \to V_{\vI[1] \subset \vJ[1]}
\end{equation}
to be the union of the restrictions
\begin{equation} 
   \Ann^{\infty}_{\vI[1] \subset \vJ[1]  }(\nu_{\vI \subset \vJ} \sigma_{\vI[1] \subset \vJ[1]}) | \nu_{\vI \subset \vJ} \sigma_{\vI[1] \subset \vJ[1]} \cap   V_{\vI[1] \subset \vJ[1]},
\end{equation}
glued along the natural identifications across strata of the pairs barycentric subdivision. The compatibility of the Floer data across strata implies that the maps in Equation \eqref{eq:continuation_data_annuli_nbd_substratum} yield a continuous family of Floer data
\begin{equation}
  \Phi^{\infty}_{\vI[1] \subset \vJ[1]} \co \Ann^{\infty}_{\vI[1] \subset \vJ[1]  }(V_{\vI[1] \subset \vJ[1]}) \to \scrH^0 \times \scrJ^{2^{n+2}} \times \scrD^{2^{n+2}}.
\end{equation}

\section{Cardy's relation}
\label{sec:cardys-relation}

The reader familiar with the use of Cardy's relation in the generation criterion for Fukaya categories \cite{A-generate} is likely already aware of what will happen in the rest of the paper: having constructed a moduli space of degenerate annuli parametrised by $q \in Q$, we shall realise this moduli space as a boundary of a moduli space of annuli of arbitrary modular parameter. This moduli space will have another boundary component, which will be shown to give rise to the identity map on Floer cohomology.

There is one remaining technical difficulty ahead, arising from the need to define compatible gluings of the moduli spaces of degenerate annuli defined over different cells of the pairs barycentric subdivision associated to $\Sigma$.  In the next few paragraphs, which have no content subsequently required in the paper, we attempt to give some intuition for the problem we shall face. The reader who would like to continue with the logical development can safely skip to Section \ref{sec:moduli-smooth-annuli}.

In essence, the difficulty has nothing to do with annuli, and can be illustrated more succinctly by considering a moduli space of continuation maps parametrised by a  $1$-dimensional manifold $\Q$, which is stratified as two closed intervals meeting at a point $q$. The reader should have in mind the example in 
which the  moduli spaces we shall next describe arise as unions of components in the moduli space of degenerate annuli.

\begin{figure}
\centering
\begin{tikzpicture}
\draw[line width=2*\lw] (-1,0)--(3,0);
\fill (1,0) circle (8*\lw);
\node[below] at (1,0) {$q$};

\begin{scope}[shift={(-.5,0)}]
\draw[line width=\lw] (-1/8,2.5+1/8-1/2 )--(-1/8 , 3.75+1/8+1/2);
\draw[line width=\lw] (1/8, 2.5+1/8 - 1/2)--(1/8, 3.75+1/8+1/2);
\fill (0, 2.5+1/8) circle (2*\lw);
\fill (0, 3.75+1/8) circle (2*\lw);
\node at (0,4.5) {${\scriptstyle l}$};

\node at (0,.75) {${\scriptstyle i}$};
\draw[line width=\lw] (-1/8,1.5-1/8-1/2 )--(-1/8 , 1.5-1/8+1/2);
\draw[line width=\lw] (1/8, 1.5-1/8 - 1/2)--(1/8, 1.5-1/8+1/2);
\fill (0, 1.5-1/8) circle (2*\lw);
\node at (0,2) {${\scriptstyle j}$};
\node at (0,3.25) {${\scriptstyle S}$};
\end{scope}

\node at (1,.75) {${\scriptstyle i}$};
\draw[line width=\lw] (1-1/8,1.5-1/8-1/2 )--(1-1/8 , 1.5-1/8+1/2);
\draw[line width=\lw] (1+1/8, 1.5-1/8 - 1/2)--(1+1/8, 1.5-1/8+1/2);
\fill (1, 1.5-1/8) circle (2*\lw);
\node at (1,2) {${\scriptstyle j}$};
\draw[line width=\lw] (1-1/8,2.5+1/8-1/2 )--(1-1/8 , 2.5+1/8+1/2);
\draw[line width=\lw] (1+1/8, 2.5+1/8 - 1/2)--(1+1/8, 2.5+1/8+1/2);
\fill (1, 2.5+1/8) circle (2*\lw);
\node at (1,3.25) {${\scriptstyle k}$};
\draw[line width=\lw] (1-1/8,3.75+1/8-1/2 )--(1-1/8 , 3.75+1/8+1/2);
\draw[line width=\lw] (1+1/8, 3.75+1/8 - 1/2)--(1+1/8, 3.75+1/8+1/2);
\fill (1, 3.75+1/8) circle (2*\lw);
\node at (1,4.5) {${\scriptstyle l}$};

\begin{scope}[shift={(.5,0)}]
\node at (2,.75) {${\scriptstyle i}$};
\draw[line width=\lw] (2-1/8,1.5-1/8-1/2 )--(2-1/8 , 2.5+1/8+1/2);
\draw[line width=\lw] (2+1/8, 1.5-1/8 - 1/2)--(2+1/8, 2.5+1/8+1/2);
\fill (2, 1.5-1/8) circle (2*\lw);
\fill (2, 2.5+1/8) circle (2*\lw);
\node at (2,2) {${\scriptstyle S}$};
\node at (2,3.25) {${\scriptstyle k}$};
\draw[line width=\lw] (2-1/8,3.75+1/8-1/2 )--(2-1/8 , 3.75+1/8+1/2);
\draw[line width=\lw] (2+1/8, 3.75+1/8 - 1/2)--(2+1/8, 3.75+1/8+1/2);
\fill (2, 3.75+1/8) circle (2*\lw);
\node at (2,4.5) {${\scriptstyle l}$};
\end{scope}
\begin{scope}[shift={(5,0)}]
\draw[line width=2*\lw] (-1,0)--(3,0);
\fill (1,0) circle (8*\lw);
\node[below] at (1,0) {$q$};

\begin{scope}[shift={(-.5,0)}]
\fill (0, 2.5+1/8) circle (2*\lw);
\fill (0, 3.75+1/8) circle (2*\lw);
\node at (0,4.5) {${\scriptstyle l}$};

\node at (0,.75) {${\scriptstyle i}$};
\draw[line width=\lw] (-1/8,1.5-1/8-1/2 )--(-1/8 , 3.75+1/8+1/2);
\draw[line width=\lw] (1/8, 1.5-1/8 - 1/2)--(1/8, 3.75+1/8+1/2);
\fill (0, 1.5-1/8) circle (2*\lw);
\node at (0,2) {${\scriptstyle R}$};
\node at (0,3.25) {${\scriptstyle S}$};
\end{scope}

\begin{scope}[shift={(.5,0)}]
\node at (2,.75) {${\scriptstyle i}$};
\draw[line width=\lw] (2-1/8,1.5-1/8-1/2 )--(2-1/8 , 3.75+1/8+1/2);
\draw[line width=\lw] (2+1/8, 1.5-1/8 - 1/2)--(2+1/8, 3.75+1/8+1/2);
\fill (2, 1.5-1/8) circle (2*\lw);
\fill (2, 2.5+1/8) circle (2*\lw);
\node at (2,2) {${\scriptstyle S}$};
\node at (2,3.25) {${\scriptstyle R}$};
\fill (2, 3.75+1/8) circle (2*\lw);
\node at (2,4.5) {${\scriptstyle l}$};
\end{scope}

\end{scope}

\begin{scope}[shift={(3.75,0)}]
\node at (2,.75) {${\scriptstyle i}$};
\draw[line width=\lw] (2-1/8,1.5-1/8-1/2 )--(2-1/8 , 2.5+1/8+1/2);
\draw[line width=\lw] (2+1/8, 1.5-1/8 - 1/2)--(2+1/8, 2.5+1/8+1/2);
\fill (2, 1.5-1/8) circle (2*\lw);
\fill (2, 2.5+1/8) circle (2*\lw);
\node at (2,2) {${\scriptstyle R}$};
\node at (2,3.25) {${\scriptstyle k}$};
\draw[line width=\lw] (2-1/8,3.75+1/8-1/2 )--(2-1/8 , 3.75+1/8+1/2);
\draw[line width=\lw] (2+1/8, 3.75+1/8 - 1/2)--(2+1/8, 3.75+1/8+1/2);
\fill (2, 3.75+1/8) circle (2*\lw);
\node at (2,4.5) {${\scriptstyle l}$};
\end{scope}

\node at (6,2.5+1/8) {$ \neq $};

\begin{scope}[shift={(6.25,0)}]
\draw[line width=\lw] (-1/8,2.5+1/8-1/2 )--(-1/8 , 3.75+1/8+1/2);
\draw[line width=\lw] (1/8, 2.5+1/8 - 1/2)--(1/8, 3.75+1/8+1/2);
\fill (0, 2.5+1/8) circle (2*\lw);
\fill (0, 3.75+1/8) circle (2*\lw);
\node at (0,4.5) {${\scriptstyle l}$};

\node at (0,.75) {${\scriptstyle i}$};
\draw[line width=\lw] (-1/8,1.5-1/8-1/2 )--(-1/8 , 1.5-1/8+1/2);
\draw[line width=\lw] (1/8, 1.5-1/8 - 1/2)--(1/8, 1.5-1/8+1/2);
\fill (0, 1.5-1/8) circle (2*\lw);
\node at (0,2) {${\scriptstyle j}$};
\node at (0,3.25) {${\scriptstyle R}$};
\end{scope}

\end{tikzpicture}

\caption{On the left is a continuous family of nodal Riemann surfaces with interior marked points over the interval. On the right is a naive  attempt to produce a family of smooth nodal surfaces by gluing, resulting in a discontinuous family.}
\label{fig:continuation_family_discontinuous}
\end{figure}
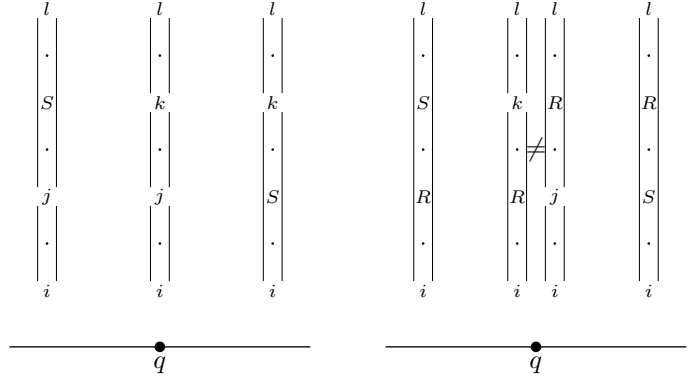

Assume that we have a totally ordered set $ i < j < k < l$, and that moduli space corresponding to $q$ is the product $\Adams_{kl} \times \Adams_{jk} \times \Adams_{ij}  $, which is simply a point, with universal curve shown in Figure \ref{fig:continuation_family_discontinuous}. We parametrise each of the intervals comprising $\Q$ by $[0, \infty]$, with $\infty$ corresponding to the common point. We identify the first interval with $\Adams_{jkl} \times \Adams_{ij}$ and the second with $\Adams_{kl} \times \Adams_{ijk}$. Some readers may be wondering at this point about which smooth structure we use on $[0, \infty]$; this is completely irrelevant, since the issue we shall encounter is one of continuity, not smoothness.  It should be clear at this stage that we have a continuous family of moduli spaces over $I$ as shown to the left of Figure \ref{fig:continuation_family_discontinuous}.

Our goal is to define a universal curve over $I \times [R,\infty]$ for some (positive) real number $R$, which restricts to the moduli space just described over $I \times \{ \infty \}$. Over each top dimensional stratum, there is a straightforward construction: there are natural embeddings
\begin{align}
\Adams_{kl} \times \Adams_{ijk} \times [R, \infty] & \to \Adams_{ijkl} \\
\Adams_{jkl} \times \Adams_{ij}  \times [R, \infty] & \to \Adams_{ijkl}
\end{align}
and the moduli space can then simply be defined by pullback. The problem is that this construction is not consistent at $q$ as shown in Figure \ref{fig:continuation_family_discontinuous}: the curve associated to the point $(q,S)$ lies in $ \Adams_{jkl} \times \Adams_{ij} $ for the first stratum, and in $ \Adams_{kl} \times \Adams_{ijl} $  for the second stratum. 

It is of course not too surprising that this naive attempt to construct an extension fails: we start with a stratified map $\Q \to \partial \Adams_{ijkl}$, and the gluing coordinates define embeddings into $ \Adams_{ijkl}$ of the product of $[R, \infty]$ with the two top-dimensional strata of $  \partial \Adams_{ijkl} $ which we are considering. However, these embeddings do not restrict to the same embedding at the intersection of these two strata, which is a point; indeed, we have described $\Adams_{ijkl}  $ as the product $[R, \infty] \times [R, \infty] $, and the two gluing coordinates correspond to the two factors. The problem can be summarised by noting that $\Q \times [R, \infty] $ is a manifold with boundary, whereas the gluing coordinates make $\Adams_{ijkl}$ into a manifold with corners. 

The solution to this problem is rather straightforward: for any real number $R$, there is certainly \emph{some} topological embedding $ \partial \Adams_{ijkl} \times [R, \infty] \to  \Adams_{ijkl} $, which can moreover be chosen to be given by the gluing coordinates away from a neighboorhood of the corner stratum (the size of the neighbourhood depends on $R$). The embedding can be given by a (varying) combination of the two gluing functions near the corner. For example, one can start by setting the embedding of the product of the corner with $ [R, \infty]  $ into $\Adams_{ijkl} $ to be the diagonal map into the pair of gluing coordinates, then cutoff one of the two factors as one moves towards the interior of the two strata. This is the basic idea behind Definition \ref{def:annulus_gluing}, which will be given in Section \ref{sec:contr-choic-gluing}.

\subsection{Moduli of smooth annuli over $\Q$}
\label{sec:moduli-smooth-annuli}

Let $\Ann \to [0, \infty)$ be the family of curves over the positive real line whose fibre at $S$ is the annulus
\begin{equation}
  \bR/ 4S \bZ \times [0,1] \cong \{ z \in \bC | e^{-\pi/2S} \leq |z| \leq 1\},  
\end{equation}
where the left hand side is equipped with coordinates $(s,t)$ and complex structure $j \partial_s = \partial_t$, and the identification is via the coordinates $e^{-2\pi \frac{1-t-i s}{4S}}$ (see Figure \ref{fig:two_models_for_annuli}).
\begin{figure}
  \centering
  \begin{tikzpicture}

\draw[line width=4*\lw] (-8*\mw pt,-1*\mw pt) -- ( 0*\mw pt, -1*\mw pt);
\draw[line width=4*\lw] (-8*\mw pt,1*\mw pt) -- ( 0*\mw pt, 1*\mw pt);
\draw[line width=4*\lw] (-8*\mw pt,-2*\mw pt) -- ( 0*\mw pt, -2*\mw pt);
\fill[blue, opacity=.5]  (-6*\mw pt,-1*\mw pt) rectangle (-2*\mw pt,1*\mw pt);
\fill[red, opacity=.25]  (-2*\mw pt,-1*\mw pt) rectangle (0*\mw pt,1*\mw pt);
\fill[red, opacity=.25]  (-8*\mw pt,-1*\mw pt) rectangle (-6*\mw pt,1*\mw pt);
\draw[line width=4*\lw] (-4*\mw pt,-2*\mw pt -1/8*\mw pt) -- (-4*\mw pt,-2*\mw pt +1/8*\mw pt);
\draw[line width=4*\lw][blue] (-4*\mw pt,1*\mw pt -1/8*\mw pt) -- (-4*\mw pt,1*\mw pt +1/8*\mw pt);
\draw[line width=4*\lw] (-8*\mw pt,-2*\mw pt -1/8*\mw pt) -- (-8*\mw pt,-2*\mw pt +1/8*\mw pt);
\draw[line width=4*\lw][red] (-8*\mw pt,1*\mw pt -1/8*\mw pt) -- (-8*\mw pt,1*\mw pt +1/8*\mw pt);
\draw[line width=4*\lw] (0*\mw pt,-2*\mw pt -1/8*\mw pt) -- (0*\mw pt,-2*\mw pt +1/8*\mw pt);
\node[below] at (-8*\mw pt, -2*\mw pt)  {$-2S$};
\node[below] at (-0*\mw pt, -2*\mw pt)  {$2S$};
\node[below] at (-4*\mw pt, -2*\mw pt)  {$0$};
\node[above] at (-4*\mw pt, 1*\mw pt)  {$z_\inp$};
\node[above] at (-8*\mw pt, 1*\mw pt)  {$z_\out$};

\fill[blue!50]  (5*\mw pt,-3*\mw pt) arc (-90:90:3*\mw pt) -- (5*\mw pt,2*\mw pt) arc (90:-90:2*\mw pt);
\draw[line width=4*\lw] (5*\mw pt,-3*\mw pt) arc (-90:90:3*\mw pt) ;
\draw[line width=4*\lw] (5*\mw pt,-2*\mw pt) arc (-90:90:2*\mw pt) ;
\draw[line width=4*\lw][blue] (8*\mw pt -1/8*\mw pt, 0*\mw pt)--(8*\mw pt+1/4*\mw pt, 0*\mw pt);
\node[right] at (8*\mw pt, 0*\mw pt) {$z_\inp$};

\fill[red!25]  (5*\mw pt,3*\mw pt) arc (90:270:3*\mw pt) -- (5*\mw pt,-2*\mw pt) arc (270:90:2*\mw pt);
\draw[line width=4*\lw] (5*\mw pt,3*\mw pt) arc (90:270:3*\mw pt) ;
\draw[line width=4*\lw] (5*\mw pt,2*\mw pt) arc (90:270:2*\mw pt) ;
\draw[line width=4*\lw][red] (2*\mw pt +1/8*\mw pt, 0*\mw pt)--( 2*\mw pt-1/4*\mw pt, 0*\mw pt);

\node[left] at ( 2*\mw pt, 0*\mw pt)  {$z_\out$} ;
\node at (5*\mw pt, 0)  {$|z| = e^{-\pi/2S}$};
\node[left] at (4*\mw pt,3*\mw pt)  {$|z|=1$};

\end{tikzpicture}
  \caption{ }
  \label{fig:two_models_for_annuli}
\end{figure}

For $S=0$, the fibre is the unit disc. Let $ \Ann^S$ denote the fibre over $S$. Given any subset $U \subset \Q$, let
\begin{equation}
 \Ann(U) \to U \times [0,\infty) 
\end{equation}
denote the pullback of this family of annuli via projection to the first factor, with $\Ann^S(U)$ the fibre over $S$. Fix the holomorphic embeddings
\begin{align}
  \iota_{\inp} , \iota_\out \co [-S,S] \times [0,1] & \to   \bR/ 4S \bZ \times [0,1]
\end{align}
which are the identity on the $[0,1]$ factor, and whose images are respectively the natural inclusions of $[-S,S] \times [0,1] $ and $[S,3S] \times [0,1] $. Note in particular that the image of $(0,1)$ under the first map is $(0,1)$, and under the second map is $(2S,1)$; in the model of the annulus embedded in the unit disc, these marked points are $\pm 1$.  Denote the corresponding
 sections of $\Ann(\Q) $  by
\begin{align}
z_{\inp}, z_\out \co \Q & \to \Ann(\Q) .
\end{align}
We write $\Ann^{\{0\}}(\Q)$ for the family of circles over $\Ann(\Q) $  corresponding to $\bR/4S \bZ \times \{0\}$, $\Ann^{z_{\inp} < z_{\out}}(\Q)$ for the family of intervals $ (0,2S) \times \{ 1 \}  $, and $\Ann^{z_{\out} < z_{\inp} }(\Q) $  for the family of intervals $(-2S,0) \times \{ 1 \} $. Let  $\Ann^{\{1\}}(\Q)$ denote the union of these intervals.

Fixing an identification of the complement of $\pm 1$ in the unit disc with the strip $B$, mapping $\pm 1$ to $\pm \infty$, defines positive and negative strip-like ends on neighbourhoods of $\{  z_\inp, z_\out \} $. Choose a family of strip-like ends
\begin{align}
  \epsilon_\pm \co \Q \times (0,\infty] \times B_\pm & \to \Ann(Q)
\end{align}
covering neighbourhoods of $z_\inp$ and $z_\out$, which agree with these strip-like ends for $S$ close to $0$, and agree with the ends obtain by gluing whenever $S$ is sufficiently large.

\subsection{Contractible choices of gluing}
\label{sec:contr-choic-gluing}
Let $\sigma_{\vI[1] \subset \vJ[1]} \subset \sigma_{\vI \subset \vJ}$ be  cells of the pairs barycentric subdivision, $\nu_{\vI \subset \vJ}  \sigma_{\vI[1] \subset \vJ[1]}$ the  neighbourhood fixed in Section \ref{sec:gluing-descr-floer}, and $\opnu_{\vI \subset \vJ}  \sigma_{\vI[1] \subset \vJ[1]} $ its interior.

Fix a cover $\{ V_{\vI \subset \vJ} \}$ of $\Q$ by open subsets whose closures $\overline{V}_{\vI \subset \vJ}$ satisfy 
\begin{align}
\overline{V}_{\vI[1] \subset \vJ[1]} & \subset \bigcup_{\vI \subset \vI[1] \subset \vJ[1] \subset \vJ} \opnu_{\vI \subset \vJ}  \sigma_{\vI[1] \subset \vJ[1]} \\
\overline{V}_{\vI[1] \subset \vJ[1]} \cap  \overline{V}_{\vI \subset \vJ} & = \emptyset \textrm{ unless the corresponding cells are nested.} 
\end{align}
As illustrated in Figure \ref{fig:cover_nbd_cells}, the intuition is that $  V_{\vI \subset \vJ} $ is an open neighbourhood in $\Q$ of the complement in $\sigma_{\vI \subset \vJ}$ of a neighbourhood of the boundary. In particular, the second condition precludes $  \sigma_{\vI \subset \vJ} $  being contained in $\overline{V}_{\vI \subset \vJ}$ for all cells, since this would imply the existence of intersections for elements of the cover labelled by cells which are adjacent but not nested (i.e. share a common cell in their boundary).

\begin{figure}
  \centering
    \begin{tikzpicture}
\begin{scope}[shift={(-10,0)}]
    \clip (-2,-2) rectangle (2,2);

\coordinate  (0) at (0,0);
\coordinate  (1) at (0,-4);
\coordinate   (2) at (-4,0);
\coordinate (3) at (4,4);

\draw [draw=black, fill=green, fill opacity=0.8]  (0) circle (1);

\foreach \x in {1,2,3} \draw[line width=2*\lw] (0)--(\x);
\node at (-.5, 1.25)  {$V_{\vJ \subset \vJ } $};
\end{scope}
\begin{scope}[shift={(-5,0)}]
    \clip (-2,-2) rectangle (2,2);

\coordinate  (0) at (0,0);
\coordinate  (1) at (0,-4);
\coordinate   (2) at (-4,0);
\coordinate (3) at (4,4);

\draw [draw=black, fill=green, fill opacity=0.8]  (0) circle (1);
\node at (0, 1.75)  {$V_{\{  i \subset ijk \} \subset \vJ } $};
\draw [draw=black, fill=blue, fill opacity=0.4]  (-2,1) .. controls (0,1) and (0,-1) .. (-2,-1);
\node at (-1.1, 1.25)  {$V_{\{ ij \subset ijk\} \subset \vJ } $};
\draw [draw=black, fill=blue, fill opacity=0.4]  (1,-2) .. controls (1,0) and (-1,0) .. (-1,-2);
\node at (-1.25, -1.5)  {$V_{\{ i \subset ij\} \subset \vJ } $};
\draw [draw=black, fill=blue, fill opacity=0.4]  (1.293,2.707) .. controls (-.707,.707) and (.707,-.707) .. (2.707,1.293);

\foreach \x in {1,2,3} \draw[line width=2*\lw] (0)--(\x);
\end{scope}
\begin{scope}
    \clip (-2,-2) rectangle (2,2);

\coordinate  (0) at (0,0);
\coordinate  (1) at (0,-4);
\coordinate   (2) at (-4,0);
\coordinate (3) at (4,4);

\draw [draw=black, fill=green, fill opacity=0.8]  (0) circle (1);
\draw [draw=black, fill=blue, fill opacity=0.4]  (-2,1) .. controls (0,1) and (0,-1) .. (-2,-1);
\draw [draw=black, fill=blue, fill opacity=0.4]  (1,-2) .. controls (1,0) and (-1,0) .. (-1,-2);
\draw [draw=black, fill=blue, fill opacity=0.4]  (1.293,2.707) .. controls (-.707,.707) and (.707,-.707) .. (2.707,1.293);
\draw [draw=black, fill=red, fill opacity=0.2]  (-3,3) -- (-2,.5) .. controls (0,0) .. (2,2.5);
\draw [draw=black, fill=red, fill opacity=0.2]  (3,-3) -- (.5,-2) .. controls (0,0) .. (2.5,2);
\draw [draw=black, fill=red, fill opacity=0.2]  (-3,-3) -- (-.5,-2) .. controls (0,0) .. (-2, -.5);
\node at (-1.25, 1.5)  {$V_{\{ ijk \} \subset \vJ } $};
\node at (-1.35, -1.5)  {$V_{\{ ij \} \subset \vJ } $};
\node at (1.3, -.85)  {$V_{\{ i \} \subset \vJ } $};
\foreach \x in {1,2,3} \draw[line width=2*\lw] (0)--(\x);
\end{scope}
\end{tikzpicture} 
 \caption{Constructing the cover $\{ V_{\vI \subset \vJ} \}$ by induction near the barycenter of the simplex labelled by $\vJ = \{ i \subset ij \subset ijk \}$. }
  \label{fig:cover_nbd_cells}
\end{figure}
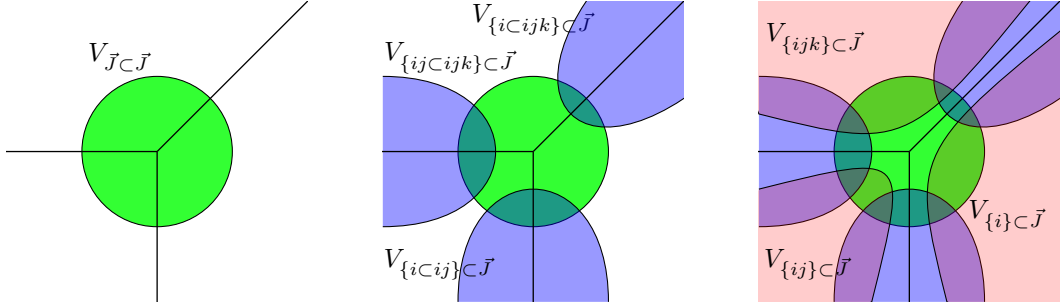

\begin{defin} \label{def:annulus_gluing}
An \emph{annulus gluing function} on $ V_{\vI \subset \vJ}$ is a map
\begin{equation} \label{eq:gluing_map_nbd_annulus}
g^{\Ann} \co  V_{\vI \subset \vJ} \times (0,+\infty] \to (0,\infty]^{\min \vI} \times (0,\infty]^{\max \vI}
\end{equation}
which is smooth on $V_{\vI \subset \vJ} \times (0,+\infty)$ , and such that the sum of the coordinates in each of the factors of the right agrees with the projection to $ (0,+\infty]$. 
\end{defin}

Since addition of coordinates defines a smooth fibre bundle with contractible fibres $
  (0,\infty)^{d} \to (0,\infty)$,
which extends to an acylic fibration $  (0,\infty]^{d} \to (0,\infty]$,
the space of choices of $g^{\Ann}$ extending any given choice on a subset of the domain is contractible.

Restricting the domain of Equation \eqref{eq:gluing_map_nbd_annulus} to $\infty$ in the second factor, Equation \eqref{eq:gluing_parameters_pairs_barycentric} yields a map on the intersection of $ V_{\vI[1] \subset \vJ[1]} $ with each stratum
\begin{multline}
   V_{\vI[1] \subset \vJ[1]} \cap  \nu_{\vI \subset \vJ}  \sigma_{\vI[1] \subset \vJ[1]} \substack{g_{\vI \subset \vI[1]} \\ \longrightarrow}  (0,\infty]^{\min \vI[1] \setminus \min \vI} \times (0,\infty]^{\max \vI[1] \setminus \max \vI} \\  \to  (0,\infty]^{\min  \vI[1] \setminus \min \vI} \times \{ \infty \}^{ \min \vI} \times (0,\infty]^{\max \vI[1] \setminus \max \vI} \times \{ \infty \}^{ \max \vI},
 \end{multline}
where we remind the reader that $\min \vI$ is the subset of $\Sigma$ consisting of all minimal elements of the nested subsets which comprise $\vI$, and $\max \vI$ is the set of maximal elements. 

We obtain a continuous map
\begin{equation}
 g^{\Ann^\infty} \co  V_{\vI \subset \vJ} \times \{ \infty \}  \to (0,\infty]^{\min \vI} \times (0,\infty]^{\max \vI}.
\end{equation}

If $ \sigma_{\vI[1] \subset \vJ[1]} \subset \sigma_{\vI \subset \vJ} $, the product of  Equation \eqref{eq:gluing_parameters_pairs_barycentric} with the choice of a map $g^{\Ann}$ on $ V_{\vI \subset \vJ} $  yields a map
\begin{equation} \label{eq:nested_pair_gluing_induced}
  V_{\vI \subset \vJ} \cap  V_{\vI[1] \subset \vJ[1]} \times (0,+\infty]  \to  (0,\infty]^{\min \vI[1]} \times (0,\infty]^{\max \vI[1]}.
\end{equation}
Since the space of choices is contractible, such a map may be extended to a map on $V_{\vI[1] \subset \vJ[1]} $. Proceeding by descending induction on the dimension of $\sigma_{\vI \subset \vJ}  $, one shows:
\begin{lem}
  There is a choice of annulus gluing maps $g^{\Ann}  $ restricting to $g^{\Ann^\infty}  $ on the boundary, and which extends the map in Equation \eqref{eq:nested_pair_gluing_induced} for every nested pair. \qed
\end{lem}

\subsection{Gluing maps from degenerate to smooth annuli}

\begin{figure}
\centering
  \begin{tikzpicture}
\node at (5*\mw pt, 0)  {$z_\inp$};
\CtoEBent{0}{0}{\min I_0}{\max I_0};
\begin{scope}
\clip (0,0) -- (1,2) -- (3,2) -- (2,-1.5) --cycle;
\fill[blue, opacity=.5] (0, 2*\mw pt) .. controls (4*\mw pt, 2*\mw pt) and  (4*\mw pt, -2*\mw pt) .. (0, -2*\mw pt) -- (0, -3* \mw pt) .. controls (5.5*\mw pt,- 3*\mw pt) and (5.5*\mw pt, 3*\mw pt) .. (0, 3*\mw pt);
\end{scope}
\ConoutBig{-2*\mw pt }{2.5*\mw pt}{\max I};
\fill[gray,opacity=.125] (-2*\mw pt - .25 *\mw pt ,2.5*\mw pt+1/2*\mw pt) rectangle (-2*\mw pt - 2.75*\mw pt, 2.5*\mw pt  -1/2*\mw pt);
\ConoutBig{-2*\mw pt }{-2.5*\mw pt}{\min I};
\fill[gray,opacity=.125] (-2*\mw pt - .75 *\mw pt ,-2.5*\mw pt+1/2*\mw pt) rectangle (-2*\mw pt - 3.5*\mw pt, -2.5*\mw pt  -1/2*\mw pt);
   \EtoCBent{-8*\mw}{0};
\begin{scope}
\clip (-8*\mw pt,0) -- (-5*\mw pt-8*\mw pt,2*\mw pt) -- (-5*\mw pt-8*\mw pt,-2*\mw pt) -- (-10.5*\mw pt,-3.5*\mw pt) --cycle;
\fill[red, opacity=.25] (-8*\mw pt, 2*\mw pt) .. controls (-8*\mw pt-4*\mw pt, 2*\mw pt) and  (-8*\mw pt-4*\mw pt, -2*\mw pt) .. (-8*\mw pt, -2*\mw pt) -- (-8*\mw pt, -3* \mw pt) .. controls (-8*\mw pt-5.5*\mw pt,- 3*\mw pt) and (-8*\mw pt-5.5*\mw pt, 3*\mw pt) .. (-8*\mw pt, 3*\mw pt);
\end{scope}
\node at (-13*\mw pt, 0)  {$z_\out$};
  \end{tikzpicture}
\caption{A fibre of $\Ann^{\infty}(\sigma_{\vI \subset \vJ}) $ for $\vI = \{I_0 \subset I \}$, with shaded regions determined by the function $ g^{\Ann} $.}
\label{fig:annulus_gluing_regions-0}
\end{figure}

Let $ V_{\vI \subset \vJ} $ be the open set associated to $\sigma_{\vI \subset \vJ} $ in the previous section. Recall that the components of the fibres of $\Ann^{\infty}_{\vI \subset \vJ}(V_{\vI \subset \vJ}) $ over a point in $V_{\vI \subset \vJ} $ are identified with strips, and have ends labelled by the pairs $(\min I_0, \max I_0)$, $(\min I, \max I)$, or successive elements of either $\min \vI$ or $\max \vI$.

Given an element $i \in \max \vI \cup \max \vJ$, we write $g^{\Ann}_{i} $ for the corresponding component of the gluing function $g^{\Ann} $.  Since the closure of $V_{\vI \subset \vJ} $ does not intersect any cell which does not include $\sigma_{\vI \subset \vJ}$,  the Floer data are constant outside the union of the finite strips
\begin{align}
&[-g^{\Ann}_{\max I_0}(v,S) ,g^{\Ann}_{\min I_0}(v,S)] \times [0,1]  \\
&[-g^{\Ann}_{\min I}(v,S) ,g^{\Ann}_{\max I}(v,S)] \times [0,1] \\
&[- g^{\Ann}_{j}(v,S), g^{\Ann}_{i}(v,S)] \times [0,1] 
\end{align}
whenever $S$ is large enough, where  $i<j$  are successive elements of $ \min \vI$  or $\max \vI $  in the last equation (see Figure \ref{fig:annulus_gluing_regions-0}).

Restricting the embeddings $\iota_\inp$ and $\iota_\out$ defines maps
\begin{align} \label{eq:inp-interval-ann-glue}
[-g^{\Ann}_{\max I_0}(v,S) ,g^{\Ann}_{\min I_0}(v,S)] \times [0,1] & \to  \bR/4S\bZ \times [0,1] \\ \label{eq:out-interval-ann-glue}
[-g^{\Ann}_{\min I}(v,S) ,g^{\Ann}_{\max I}(v,S)] \times [0,1]  &  \to \bR/4S\bZ \times [0,1], 
\end{align}
where the image of the second map is $ [2S-g^{\Ann}_{\min I}(v,S) ,2S+g^{\Ann}_{\max I}(v,S)] \times [0,1]  $.

To define the embeddings corresponding to the other components of $ \Ann^{\infty}_{\vI \subset \vJ}(V_{\vI \subset \vJ}) $, let
\begin{equation}
  R_j(v,S) = 2 \sum_{\substack{ k \in \min \vI \\ j \leq k} } g^{\Ann}_{k}(v,S)
\end{equation}
for $j \in \min \vI$ and define the embedding in $ \bR/4S\bZ \times [0,1]$
\begin{multline} \label{eq:min-intervals-ann-glue}
  [- g^{\Ann}_{j}(v,S), g^{\Ann}_{i}(v,S)] \times [0,1] \to   [R_j(v,S) - g^{\Ann}_{j}(v,S),R_j(v,S)+ g^{\Ann}_{j}(v,S)] \times [0,1] 
\end{multline}
for successive elements $ \min I_0 \neq i <j$ of $\min \vI$. Since $R_{\min I}(v,S) =2 S  $  by assumption, these embeddings intersect only on the boundary, and their images cover the finite strip
\begin{equation} \label{eq:min-intervals-union-ann-glue}
  [ g^{\Ann}_{\min I_0}(v,S), 2S-g^{\Ann}_{\min I}(v,S)] \times [0,1] \subset \bR/4S\bZ \times [0,1].
\end{equation}

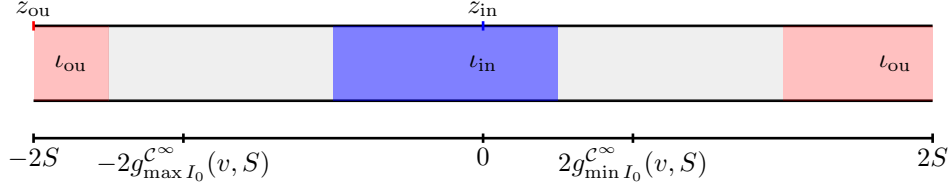
\begin{figure}
  \centering
  \begin{tikzpicture}

\draw[line width=4*\lw] (-12*\mw pt,-1*\mw pt) -- ( 12*\mw pt, -1*\mw pt);
\draw[line width=4*\lw] (-12*\mw pt,1*\mw pt) -- ( 12*\mw pt, 1*\mw pt);
\draw[line width=4*\lw] (-12*\mw pt,-2*\mw pt) -- ( 12*\mw pt, -2*\mw pt);
\fill[blue, opacity=.5]  (-4*\mw pt,-1*\mw pt) rectangle (2*\mw pt,1*\mw pt);
\fill[red, opacity=.25]  (-12*\mw pt,-1*\mw pt) rectangle (-10*\mw pt,1*\mw pt);
\fill[red, opacity=.25]  (12*\mw pt,-1*\mw pt) rectangle (8*\mw pt,1*\mw pt);
\fill[gray, opacity=.125]  (-10*\mw pt,-1*\mw pt) rectangle (-4*\mw pt,1*\mw pt);
\fill[gray, opacity=.125]  (2*\mw pt,-1*\mw pt) rectangle (8*\mw pt,1*\mw pt);
\draw[line width=4*\lw] (-12*\mw pt,-2*\mw pt -1/8*\mw pt) -- (-12*\mw pt,-2*\mw pt +1/8*\mw pt);
\draw[line width=4*\lw] (-8*\mw pt,-2*\mw pt -1/8*\mw pt) -- (-8*\mw pt,-2*\mw pt +1/8*\mw pt);
\draw[line width=4*\lw] (0*\mw pt,-2*\mw pt -1/8*\mw pt) -- (0*\mw pt,-2*\mw pt +1/8*\mw pt);
\draw[line width=4*\lw] (4*\mw pt,-2*\mw pt -1/8*\mw pt) -- (4*\mw pt,-2*\mw pt +1/8*\mw pt);
\draw[line width=4*\lw] (12*\mw pt,-2*\mw pt -1/8*\mw pt) -- (12*\mw pt,-2*\mw pt +1/8*\mw pt);
\draw[line width=4*\lw][blue] (0*\mw pt,1*\mw pt -1/8*\mw pt) -- (0*\mw pt,1*\mw pt +1/8*\mw pt);
\draw[line width=4*\lw][red] (-12*\mw pt,1*\mw pt -1/8*\mw pt) -- (-12*\mw pt,1*\mw pt +1/8*\mw pt);

\node at (-11*\mw pt, 0*\mw pt)  {$\iota_\out$};
\node at (0*\mw pt, 0*\mw pt)  {$\iota_\inp$};
\node at (11*\mw pt, 0*\mw pt)  {$\iota_\out$};

\node[below] at (-12*\mw pt, -2*\mw pt)  {$-2S$};
\node[below] at (-8*\mw pt, -2*\mw pt)  {$-2g^{\Ann^\infty}_{\max I_0}(v,S)$};
\node[below] at (4*\mw pt, -2*\mw pt)  {$2g^{\Ann^\infty}_{\min I_0}(v,S)$};
\node[below] at (0*\mw pt, -2*\mw pt)  {$0$};
\node[below] at (12*\mw pt, -2*\mw pt)  {$2S$};
\node[above] at (0*\mw pt, 1*\mw pt)  {$z_\inp$};
\node[above] at (-12*\mw pt, 1*\mw pt)  {$z_\out$};

\end{tikzpicture}
  \caption{Decomposition of the annulus for gluing parameter $S$ into finite strips.}
  \label{fig:gluing-strips-to-annuli-2}
\end{figure}

Similarly, if $j \in \max \vI$, define
\begin{equation}
  R_j(v,S) = 2 \sum_{\substack{ k \in \max \vI \\ j \geq k} } g^{\Ann}_{k}(v,S),
\end{equation}
and consider the embedding
\begin{multline} \label{eq:max-intervals-ann-glue}
  [- g^{\Ann}_{j}(v,S), g^{\Ann}_{i}(v,S)] \times [0,1] \to   [-R_i(v,S) - g^{\Ann}_{j}(v,S),-R_i(v,S)+ g^{\Ann}_{i}(v,S)] \times [0,1]
\end{multline}
The union of these strips is
\begin{equation} \label{eq:max-intervals-union-ann-glue}
  [ -2S+g^{\Ann}_{\max I}(v,S), -g^{\Ann}_{\max I_0}(v,S)] \times [0,1] \subset \bR/4S\bZ \times [0,1].
\end{equation}
Note that the annulus is covered by the images of Equations \eqref{eq:inp-interval-ann-glue}, \eqref{eq:out-interval-ann-glue}, \eqref{eq:min-intervals-union-ann-glue}, and \eqref{eq:max-intervals-union-ann-glue}, see Figure \ref{fig:gluing-strips-to-annuli-2}. Collapsing the infinite ends of each strip containing the domains of Equations \eqref{eq:inp-interval-ann-glue}, \eqref{eq:out-interval-ann-glue}, \eqref{eq:min-intervals-ann-glue}, and \eqref{eq:max-intervals-ann-glue} to the corresponding boundary interval, yields a surjective map
\begin{equation}
  G^{\Ann}_{ \vI \subset \vJ } \co  \Ann^{\infty}_{\vI \subset \vJ  }(V_{\vI \subset \vJ}) \to \Ann^{S}(V_{\vI \subset \vJ}).
\end{equation}

\subsection{Floer data on annuli}

Floer data on annuli are given by a map
\begin{align}
\Phi_\Q = (  \phi_{\Q}, J_\Q, \psi_\Q)  \co \Ann(\Q) & \to \scrH^0 \times \scrJ \times  \scrD
\end{align}
such that (i) the pullback of $\Phi_{\Q}$  under $\epsilon_{\inp} $ and $\epsilon_{\out}$ agrees with $(\Id, J_t,\Id)$, (ii)  the restriction of $J_{\Q}$ to $\Ann^{z_{\inp} < z_{\out} }(\Q)$ agrees with the pushforward of $J$ under $\psi_\Q$, and (iii)  the restriction of $J_{\Q}$ to $\Ann^{z_{\out} < z_{\inp} }(\Q)$ agrees with the pushforward of $J'$ under $\psi_\Q$.

By gluing, $\Phi^{\infty}_{\vI \subset \vJ} $  induces Floer data $\Phi_{\vI \subset \vJ} $ on a neighbourhood of $S = \infty$ in  $\Ann(V_{\vI \subset \vJ}) $ by requiring the commutativity of the diagram
  \begin{equation}
    \xymatrix{ \Ann^{\infty}_{\vI \subset \vJ  }(V_{\vI \subset \vJ}) \ar[r] \ar[dr]_{\Phi^{\infty}_{\vI \subset \vJ}} & \Ann^{S}(V_{\vI \subset \vJ}) \ar[d]^{ \Phi_{\vI \subset \vJ} } \\
      & \scrH^0 \times \scrJ \times \scrD}
  \end{equation}
whenever $S$ is sufficiently large.  The maps $(\phi_{\Q},  \psi_\Q )$ are said to be \emph{obtained by  gluing} if their restrictions to $ \Ann^{S}(V_{\vI \subset \vJ}) $ agree with $ (\phi_{\vI \subset \vJ}, \psi_{\vI \subset \vJ} )  $ whenever $S$ is sufficiently large.

As in the case of strips, a more general class of almost complex structures is needed. Define the \emph{$R$-thick} subset of $\Ann(V_{\vI \subset \vJ}) $ to be the union of the images of $[-R,R] \times [0,1]$ under the maps in Equations \eqref{eq:inp-interval-ann-glue}, \eqref{eq:out-interval-ann-glue}, \eqref{eq:min-intervals-ann-glue}, and \eqref{eq:max-intervals-ann-glue}. A section of the pullback of $T \scrJ $ by  $ J_{\vI \subset \vJ}  $ is consistent if it is supported in the interior of the $R$-thick part, and  vanishes to infinite order at $S = \infty$.
\begin{defin} \label{def:data_annulus_glued}
The  data $ \Phi_{\Q} $ are \emph{obtained by perturbed gluing} if $\phi_\Q $ and $\psi_\Q$ are obtained by gluing, and the restriction of $J_{\Q}$ to a neighbourhood of $S= \infty$ in $ \Ann(V_{\vI \subset \vJ})$ agrees with a consistent perturbation of $J_{\vI \subset \vJ}  $. 
\end{defin}

\subsection{Moduli spaces of annuli}
The identification of the complement of $z_\inp$ and $z_\out$ in $\Ann^0$ with a strip yields data
\begin{equation}
( \phi^{0}_{\Q}, J^{0}_{\Q}, \psi^{0}_{\Q}) \equiv (\Id, J_t, \Id)  \co \Ann^{0}(Q) \to  \scrH^0 \times \scrJ \times \scrD.
\end{equation}

There is a natural embedding $\Ann^{S}(Q) \subset \Ann^0(Q)$. Say that Floer data  $( \phi_{\Q}, J_{\Q}, \psi_{\Q}) $ are obtained by gluing near $S=0$ if they agree with the restriction of the above data for $S$ sufficiently small. This agrees with the usual notion of gluing under the assumption, made in Section \ref{sec:floer-equation}, that $J_t$ is constant in a neighbourhood of $t=1/2$. Such data are obtained by perturbed gluing if there is a compact subset of the interior of the punctured disc where the almost complex structure agrees with the restriction of the almost complex structure $J^{0}_\Q$ up to a perturbation which vanishes to infinite order at $S=0$.

Choose Floer data  $( \phi_{\Q}, J_{\Q}, \psi_{\Q}) $ which are obtained by perturbed gluing near $S=0$ and $S=\infty$, and whose pullbacks under the strip-like ends $\epsilon_\pm$ agree with the restrictions of $( \phi^{0}_{\Q}, J^{0}_{\Q}, \psi^{0}_{\Q}) $. Given intersection points $x_\inp ,x_\out \in L \cap L'$, define $
  \cM^{S}(x_\out; x_\inp) $
to be the union over $q \in Q$ of the space of finite-energy maps from fibres of $\Ann^{S}(\Q)$ to $X$ with boundary conditions
\begin{align} 
 \partial_{s} u(z) & = J_{\Q}(z) \partial_{t}u(z) &u(z) & \in F_{q}\textrm{ if } z \in \Ann^{\{0\}}(q)  \\ 
u(z) & \in  \phi_{Q}(z) (L) \textrm{ if } z \in  \Ann^{z_{\inp} < z_{\out} }(\Q) & u(z) & \in  \phi_{Q}(z) (L') \textrm{ if } z \in  \Ann^{z_{\out} < z_{\inp} }(\Q),
\end{align}
and which converge to $x_\inp$ and $x_\out$ at the ends. Let
\begin{equation}
  \cM^{(0,\infty)}( x_\out ; x_\inp ) \equiv \coprod_{S \in (0,\infty)} \cM^{S}(x_\out ; x_\inp),
\end{equation}
which is topologised as a parametrised moduli space over $\Q \times (0,\infty) $.

\begin{lem}
  For generic Floer data  $ \Phi_{\Q} $, $  \cM^{(0,\infty)}( x_\out ; x_\inp ) $  is a smooth manifold and
  \begin{equation}
    \dim \cM^{(0,\infty)}( x_\out ; x_\inp ) = 1 + \deg x_\out - \deg x_\inp.
  \end{equation}
\end{lem}
\begin{proof}
The formula for the virtual dimension is a special case of \cite[Section 12c]{seidel-Book}. The assertion that generic Floer data yield smooth moduli spaces follows from  \cite[Theorem  5.1]{FHS}.
\end{proof}

Denote the space of annuli such that the image of the boundary
component mapping to a fibre is null-homologous by   $\cM^{(0,\infty)}_{[0]}(x_\out ;
x_\inp)$.   Let $\Mbar^{[0,\infty]}_{[0]}(x_\out ; x_\inp)$ denote its Gromov-Floer compactification. To describe the boundary,  define $\Cont(x_\out; x_\inp) $ to be the moduli space of solutions to the Floer equation with one interior marked point lying on the interval $\bR \times \{1/2\}$. This moduli space can be thought of as the continuation moduli space for a constant Hamiltonian family, and is the product of a closed interval with $\Mbar(x_\out; x_\inp) $, unless $x_\out=x_\inp$, in which it case it is a point.
\begin{lem} \label{lem:annulus_cobordism}
Whenever  $\deg x_\out = \deg x_\inp $, $\Mbar^{[0,\infty]}_{[0]}(x_\out ; x_\inp) $ is a compact $1$-dimensional manifold if the data are chosen generically, and its boundary is stratified as follows:
  \begin{align}
& \Cont(x_\out; x_\inp) \\
&    \coprod_{x \in L \cap L'}\cM^{(0,\infty)}_{[0]}(x_\out; x) \times \cM(x;x_\inp) \\
& \coprod_{x \in L \cap L'}\cM(x_\out; x) \times \cM^{(0,\infty)}_{[0]}(x;x_\inp) \\
& \cM^{\infty}_{[0]}(x_\out ; x_\inp).
  \end{align}
\end{lem}
\begin{proof}
Having excluded interior sphere bubbling by a topological assumption, and disc bubbling by the careful choice of almost complex structure along the boundary, the virtual dimension of all strata of the Gromov-Floer compactification is negative if $ \deg x_\out < \deg x_\inp $, so the regularity of the choices of Floer data implies that they are empty. If $\deg x_\out = \deg x_\inp $, the only strata which do not have negative virtual dimension are those given in the statement. The first and last stratum respectively correspond to $S=0$ and $S=\infty$, and the middle two correspond to breaking along the ends.

The proof that $\Mbar^{[0,\infty]}_{[0]}(x_\out ; x_\inp) $  is a manifold with boundary follows from a standard gluing argument. For $S=0$, this takes the form of interior gluing of the element of $\Cont(x_\out; x_\inp)  $  with the family of constant (ghost) discs on fibres of $\pi$. For $S=\infty$, there is a natural projection map
\begin{equation}
  \Ann^{\infty}_{[0]}(x_\out ; x_\inp) \to \Q.
\end{equation}
Choose an open subset $V_{\vI \subset \vJ}$ containing the image of a map $u$ in the left hand side.  The gluing description of Floer data on $\Ann(V_{\vI \subset \vJ})$ yields a gluing chart near this boundary point.
\end{proof}

\section{Floer cochains and morphisms of sheaves}
\label{sec:floer-coch-morph}

Let us fix, as in the previous section, a choice of triangulation $\Sigma$ of $Q$ which is sufficiently fine so that the results of Section \ref{sec:floer-theory-conv} apply.

\subsection{Relative $\Pin^+$ structures and orientation lines}

Let $L$ and $L'$ be graded Lagrangians as in Section \ref{sec:floer-theory}, equipped with choices of  $\Pin^+$ structures as in Section \ref{sec:floer-theory-with}. Replacing  $F_q$ by $L'$ in the discussion of Section  \ref{sec:floer-theory-with}, we obtain a rank-$1$ free abelian group
\begin{equation} \label{eq:generator_Floer_x-L-L'}
  \delta_x = |\ro_x| \otimes \nu_x
\end{equation}
which is the tensor product of (i) the orientation line on $\ro_x$, which is the determinant line of a Cauchy-Riemann operator on the upper half-plane associated to a path from $T_xL$ to $T_xL'$ in the Grassmannian of Lagrangians of $T_x X$, and (ii) the line $\nu_x$ whose trivialisations correspond to choices of $\Pin_+$ structures on the vector bundle over the interval corresponding to this path, extending the choices at the two endpoints.

For each vertex $i$ of the triangulation $\Sigma$ of $\Q$, fix a $\Pin^+$ structure on $T^*_{q_i}\Q \oplus E_{q_i}$. By the isomorphism of the tangent space of fibres with the cotangent space of the base, this induces such a structure on 
\begin{equation}
  T F_{i} \oplus \pi^* E|F_i ,
\end{equation}
where we write $F_i = F_{q_i}$ as in Section \ref{sec:covers-q}.
For each pair $i < j$ in $\Sigma$, fix a $\Pin^+$ structure on the restriction of $T^* \Q \oplus E$ to the corresponding edge $\sigma_{ij}$ of the triangulation of $\sigma$, which agrees with the $\Pin^+$ structure chosen on the ends. Given a triple $i < j < k$ in $\Sigma$, there is a canonical homotopy between $\sigma_{ik}$ and the concatenation of $\sigma_{ij}$ and $\sigma_{jk}$, which is associated to the simplex $\sigma_{ijk}$. Define a \v{C}ech cochain with coefficients in $\bZ_2$:
\begin{equation} \label{eq:cocycle_from_Pin_-}
  v_{ijk} = \begin{cases} 0 & \textrm{ if the induced } \Pin^+ \textrm{ structures agree} \\
1 & \textrm{ otherwise.}  \end{cases}
\end{equation}
\begin{lem} \label{lem:cech-cochain_pin_structure}
Equation \eqref{eq:cocycle_from_Pin_-} is a cocycle representing $w_2(\Q) \oplus w_2(E)$.
\end{lem}
\begin{proof}
The corresponding formula for orientable vector bundles is well-known, see e.g. \cite{McLaughlin}. The general case follows from the existence of a canonical isomorphism between the set of $\Pin^+$ structures on a bundle $E$ and $\Spin$ structures on $(\det E)^{\oplus 3} \oplus E$, see \cite{KT}.
\end{proof}

Assume we are given, for each $i \in \Sigma$, a Hamiltonian diffeomorphism $\phi_i$ so that $L_i \equiv \phi_i L$ is transverse to $F_q$ for all $q \in P_i$. Since $P_i$ is convex, there is a unique way of associating to each intersection point $x \in L_i \cap F_i$ an intersection point between $L_i$ and $F_q$ which we denote $x(q)$. The Cauchy-Riemann operator associated to  $x(q)$ and the relative $\Pin^+$ structure determine a local system over $P_{i}$ with fibre $\delta_{x(q)} $. Given a pair $i <j$, the assumption that $P_j$ is contained in $P_i$ therefore yields a map 
\begin{equation} \label{eq:iso_edge_i_k}
   \delta_{x} \to \delta_{x(q_j)}.
\end{equation}
Given an ordered triple $i< j < k$,  the map $ \delta_{x} \to \delta_{x(q_k)} $ defined by $\sigma_{ik}$ agrees with the composition
\begin{equation} \label{eq:compositoin_iso_edge}
 \delta_{x} \to \delta_{x(q_j)} \to \delta_{x(q_k)}
\end{equation} 
if and only if the $\Pin^+$ structure on the restriction of $T \Q \oplus E$ to the boundary of $\sigma_{ij}$ extends to the interior. From Lemma \ref{lem:cech-cochain_pin_structure}, we conclude:
\begin{cor} \label{cor:signs_given_by_v}
 The isomorphisms in Equations \eqref{eq:compositoin_iso_edge}  and \eqref{eq:iso_edge_i_k} differ by $(-1)^{v_{ijk}}$.  \qed
\end{cor}

\subsection{Energy of strips and annuli}
\label{sec:other-auxil-choic}
Recall that we have a fixed choice of Lagrangian section $\tau_i$ over $P_i$ as in Section \ref{sec:twisting-cocycle}. The intersection of this section with $F_i$ equips this fibre with a basepoint.

 Since $L_i$ is assumed to be transverse to all fibres over $P_i$, we can identify the components of $L_i \cap X_{P_i}$ with the intersection points of $L_i $  and $F_i$.  Given an intersection point $x$, we pick, as in Section \ref{sec:areas-strips-flux}, a function
\begin{equation}
  g_x \co P_i \to \bR
\end{equation}
such that the section corresponding to $x$ is obtained by fiberwise addition of $dg_x$ to $\tau_i$. As in Section \ref{sec:areas-strips-flux}, this function determines a path along $F_i$ from the basepoint to $x$. Given a map $u$ from a strip to $X$, with one boundary component mapping to $F_i$  and converging at both ends to intersection points with $L_i$, we obtain a  homology class
\begin{equation}
  [ \partial u] \in H_1(F_i, \bZ)
\end{equation}
by concatenating the boundary of $u$ along $F_i$ to the two paths connecting the intersection points to the basepoint.

We shall need a notion of energy for continuation maps: assume that $u$ satisfies a holomorphic curve equation with moving Lagrangian boundary conditions along $\bR \times \{1 \}$ given by a Hamiltonian family $ \phi^s $, parametrised by $s \in \bR $. Let 
\begin{equation}
  H \co \bR \times X \to \bR 
\end{equation}
be the Hamiltonian generating this family, normalised so that integral of $H_s$ with respect to a fixed volume form vanishes. Define the energy of $u$ to be
\begin{equation}
    \cE(u) = \int_{B} u^{*}(\omega) - \int_{\bR}   H_{s}(u(s,1))  ds .
\end{equation}
The analogue of Lemma \ref{lem:areas-strips-flux-1} holds for such maps, i.e. whenever $u$ and $u'$ are holomorphic maps from a strip with boundary conditions given by the same path of Lagrangians along $\bR \times \{ 1\} $, and by $F_q$ and $F_{q'}$ along $\bR \times \{ 0\} $, we have
\begin{equation} \label{eq:energy_deformed_curve}
  \cE(u') - \cE(u) = \langle q' -q , [\partial u] \rangle + g_y(q) - g_x(q) + g_{x'}(q') - g_{y'}(q'),
\end{equation}
whenever $u$ and $u'$ are homotopic.  Here, $(x,y)$ are the limits of $u$ along the ends (which are intersection points of $F_q$ with the appropriate Lagrangians) and  $(x',y')$ the corresponding limits of $u'$.

A generalisation of this result to annuli shall be required: fix the normalised Hamiltonian $H \co  \Ann(\Q)  \to \bR$ generating a moving path of Lagrangians along the boundary of $\Ann(\Q) $. For a fibre in $ \Adams^{S}(x_\out; x_\inp) $, we obtain an expression for the energy
\begin{equation}
\cE(u) = \int_{\Ann^{S} } u^{*}(\omega) - \int_{s \in [0,4S]}   H_{(1,s)}(u(1,s)).
\end{equation}
Extending this map to $\Adams^{\infty}(x_\out; x_\inp) $ by the sum of the energies of each constituent strip  yields a real-valued function $\cE$ on $\Adams^{[0,\infty]}(x_\out; x_\inp)$, which is easily seen to be locally constant if the boundary condition $F_q$ is fixed.  If we change boundary conditions, Equation \eqref{eq:energy_deformed_curve} implies that the energy of (nearby) curves in $\Adams^{\infty}(x_\out; x_\inp) $ with boundary on $F_q$ and $F_{q'}$ differ by
$ \langle q' -q , \gamma \rangle$, 
where $\gamma$ is the homology class of the boundary in $H_1(F_q, \bZ)$. This implies that the energy is locally constant on $\Adams^{\infty}_{[0]}(x_\out; x_\inp) $. By expressing the difference of energy of nearby annuli as the integral of $\omega$ over an annulus in $X$ connecting the boundary conditions, this result extends to all annuli:
\begin{lem}
 $\cE$ is locally constant on $\Adams^{[0,\infty]}_{[0]}(x_\out; x_\inp) $. \qed 
\end{lem}

\subsection{From local to global}
\label{sec:from-local-global}

Let $K$ be a totally ordered subset of $\Sigma$, which we recall is a triangulation of $Q$ as produced in Section \ref{sec:unif-small-choic}.  Given $x_{\min K} \in \phi_{\min K} L \cap F_{\min K}$ and $x_{\max K} \in  \phi_{\max K} L \cap F_{\max K}$, recall that
\begin{equation}
  \cM_K ( x_{\max K}; x_{\min K} ) \equiv  \cM_{q_K,K} ( x_{\max K};  x_{\min K}(q_K)).
\end{equation}
There is a natural isomorphism
\begin{equation}
  |T_u \cM_K ( x_{\max K}; x_{\min K} )| \otimes \delta_{x_{\min K}(q_{\max K})} \equiv | \Adams_K| \otimes \delta_{x_{\max K}}.
\end{equation}
Assuming that $\deg( x_{\max K}) = \deg( x_{\min K}) + 2 - |K|$, rigidity yields an isomorphism
\begin{equation}
  \kappa_u \co \delta_{x_{\min K}} \to \delta_{x_{\max K}}
\end{equation}
by fixing (i) the isomorphism in Equation \eqref{eq:iso_edge_i_k} and (ii) the orientation of $\Adams_K$ arising from its description as a product of intervals, and the ordering on $K$.  Define
\begin{align}
  \scrF_K \co   \scrF(L,\min K) & \to   \scrF(L,\max K)[2-|K|] \\  \label{eq:component_differential_K}
\scrF_K| \delta_{x_{\min K}} & = \bigoplus_{x_{\max K}}  T^{ f_{\min K, \max K}} z^{df_{\min K, \max K} - dg_{\max K} + dg_{\min K}}_{\max K}  \\ \notag
& \qquad \sum_{u \in  \cM_K ( x_{\max K}; x_{\min K} )} T^{\cE(u) } z^{[\partial u]}_{\max K} \otimes \kappa_u.
\end{align}
The discussion given in Section \ref{sec:local-mirr-constr} readily extends to show that this map is convergent (see also \cite[Proposition 3.11]{A-ICM}).

There is a natural bijection between the boundary of the $1$-dimensional moduli spaces $\Mbar_K ( x_{\max K}; x_{\min K} ) $, given in Equations  \eqref{eq:boundary_continuation_strip_out}-\eqref{eq:boundary_continuation_skip}, and the terms of Equation \eqref{eq:a_oo_functor_equation}:
\begin{equation}
  \begin{array}{ll}
 \Mbar_{ K^\geq_{i}} \times \Mbar_{K^\leq_{i}} & \longleftrightarrow \scrF_{K^\geq_{i}} \times \scrF_{K^\leq_{i}}  \\
 \Mbar_{K \setminus \{ i \}}  &  \longleftrightarrow  \scrF_{K \setminus \{ i \}}  \\
\end{array}
\end{equation}
To conclude that Equation \eqref{eq:a_oo_functor_equation} holds, it suffices to show that the coefficient of each term is correct, i.e prove  the cancellation of the terms in Equation \eqref{eq:a_oo_functor_equation} which have as coefficient a fixed monomial.

Fix an energy $E$ and a homology class $\beta \in H^1(F_{q_{\max K}}, \bZ)$, and let $ \cM_{q_{\max K},K}^{E, \beta} ( x_{\max K}; x_{\min K} ) $ be the corresponding component of the moduli space.
\begin{lem} \label{lem:energy_contribution_boundary_strata_differential}
Up to sign, the contribution to Equation \eqref{eq:a_oo_functor_equation}  of every curve $u$ lying in a boundary stratum of $ \cM_{q_{\max K},K}^{E, \beta} ( x_{\max K}; x_{\min K} ) $ agrees with
\begin{equation}
 T^{ f_{\min K, \max K}(q_{\max K})} z^{df_{\min K, \max K} - dg_{\max K}( x_{\max K} ) + dg_{\min K}( x_{\min K}   ) }_{\max K} T^{\cE(u) } z^{[\partial u]}_{\max K} \otimes \kappa_u.
\end{equation}
\end{lem}
\begin{proof}
The only strata for which this does not follow immediately from the definition are those corresponding to $\Mbar_{ K^\geq_{i}} \times \Mbar_{K^\leq_{i}}  $, for $E$ corresponds to the energy of a holomorphic curve with boundary conditions $F_{\max K}$, while the map $\scrF_{K^\leq_{i}}  $ is defined using a moduli space with boundary conditions $F_{i}$.

In particular, the functions $f_{\min K, i}$ and $f_{i, \max K}$ both contribute to the composition $\scrF_{K^\geq_{i}} \times \scrF_{K^\leq_{i}}   $, whereas $f_{\min K, \max K}$ contributes to $\scrF_{K \setminus \{ i \}}$. By multiplying $\scrF_{K^\geq_{i}} \times \scrF_{K^\leq_{i}}   $ by the cocycle $\alpha_{\min K, i, \max K}$, we correct this discrepancy between strata (see Equation \eqref{eq:cocycle-alpha}, and \cite[Lemma 4.2]{A-ICM}).
\end{proof}
There is a final matter of signs to discuss: the strata  corresponding to $\Mbar_{ K^\geq_{i}} \times \Mbar_{K^\leq_{i}}  $ contribute terms in Equation \eqref{eq:a_oo_functor_equation} which are compositions of the two maps $ \scrF_{K^\geq_{i}} \times \scrF_{K^\leq_{i}} $. Let $(u_\geq, u_\leq)$ be elements of such a stratum, which respectively converge to intersection points  $(x_{\max K}, x_i )  $ and $(x_i, x_{\min K})  $ at the ends.  We must compare the composition $\kappa_{u_\geq} \circ \kappa_{u_\leq}$ with the map obtained as follows: apply a diffeomorphism to $u_\leq$ so that its boundary condition becomes $F_{\max K}$ then glue the two maps. Since we have to apply this diffeomorphism, the sign contribution of this stratum to Equation \eqref{eq:a_oo_functor_equation} entails comparing the composition
\begin{equation}
\delta_{x_{\min K}} \to \delta_{x_{\min K}(q_i)} \to \delta_{x_{\min K}(q_{\max K})}
\end{equation}
with the natural map from the left to the right given by parallel transport along $\sigma_{\min K, \max K}$. The sign computation from Corollary \ref{cor:signs_given_by_v} therefore implies: 
\begin{lem}[c.f. Lemma 4.2 of  \cite{A-ICM}] \label{lem:Floer_gives_sheaf}
The maps $\scrF_K$ define an $(\alpha^v)^{-1}$-twisted sheaf of perfect $\scrO_Y$-modules. \qed
\end{lem}

\subsection{From Floer to \v{C}ech} \label{sec:from-floer-vcech}

For  Lagrangian branes $(L,L')$, consider the Floer complex
\begin{equation}
  CF^*(L, L') = \bigoplus_{x \in L \cap L'}  \Lambda \otimes \delta_x.
\end{equation}
Every rigid curve $u \in \cM(y,x)$ determines a map $
  \partial_u \co \delta_{x} \to \delta_{y}$  defined as in Equation \eqref{eq:map_strip_differential} (the linearised $\bar{\partial}$ operator at $u$ has a kernel coming from translation, which we trivialise using $\partial_su$). Letting $\mu^1_u$ denote  $(-1)^{\deg x}\partial_u$, we obtain  the differential
\begin{equation} \label{eq:differential_Floer_complex}
  \mu^1| \delta_{x}  = \bigoplus_{y} \sum_{u \in \cM(y,x)} T^{\cE(u)} \otimes   \mu^1_u.
\end{equation}

Given $x_{\min K} \in L_{\min K} \cap F_{\min K} $, $x'_{\max K} \in L'_{\min K} \cap F_{\max K} $, and $x_{\inp} \in L \cap L'$, let
  \begin{equation}
      \Mbar_{K; \inp}(x'_{\max K} ; x_{\inp}, x_{\min K}) = \bigcup_{K=\max \vK}  \Mbar_{q,\vK}(x'_{\max K} ; x_{\inp}, x_{\min K}(q_{K}))
  \end{equation}
as in Section \ref{sec:moduli-spac-assoc}. Each element of the right hand side is a parametrised moduli space
\begin{equation}
  \Mbar_{q,\vK}(x'_{\max K} ; x_{\inp}, x_{\min K}(q_{K})) \to \Adams_{\vK; \inp}.
\end{equation}
If $|\vK| = |K|$, subsequent elements of $\vK$ differ by one element of $K$. We fix the orientation of $ \Adams_{\vK; \inp}$ coming from the ordering of these elements. Together with the isomorphisms in Equation \eqref{eq:iso_edge_i_k}, this results in a natural isomorphism
\begin{equation}
    \mu_u \co \delta_{x_\inp} \otimes  \delta_{x_{\min K}} \to \delta_{x_{\max K}}
\end{equation}
for every rigid element $u \in \Mbar_{K; \inp}(x'_{\max K} ; x_{\inp}, x_{\min K}) $. Define
\begin{align}
  \scrC_K \co  CF^*(L, L') \otimes  \scrF(L,\min K)  & \to \scrF(L',\max K)[1-|K|] \\ \label{eq:expression_map_CF_sheaves_K}
\scrC_K| \delta_{x_{\inp}} \otimes  \delta_{x_{\min K}} & = \bigoplus_{x'_{\max K}} T^{ f_{\min K, \max K}} z^{df_{\min K, \max K} - dg'_{\max K}(x'_{\max K}) + dg_{\min K}(x_{\min K})}_{\max K}  \\
& \qquad \sum_{u \in   \Mbar_{K; \inp}(x'_{\max K} ; x_{\inp}, x_{\min K}) } T^{\cE(u) } z^{[\partial u]}_{\max K} \otimes \mu_u,
\end{align}
and denote the direct sum of these maps by
\begin{equation}
   \scrC \co  CF^*(L, L') \to  \Hom(\scrF(L), \scrF(L')).
\end{equation}
The boundary strata of $ \Mbar_{K; \inp}(x'_{\max K} ; x_{\inp}, x_{\min K}) $ listed in Equations \eqref{eq:boundary_moduli_inp_Floer_diff}-\eqref{eq:stratum_boundary_input_pullback_max_min} are matched with the terms appearing in the equation for a chain map with respect to the Floer differential and the differential in Equation \eqref{eq:differential_morphism_sheaves} as follows:
\begin{equation}
  \begin{array}{ll}
 \Mbar_{K; \inp} \times \Mbar &  \longleftrightarrow \scrC_K \circ (\mu^1 \otimes \id) \\ 
 \Mbar_{ K^\geq_{i}; \inp} \times \Mbar_{K^\leq_{i}} &  \longleftrightarrow \scrC_{K^\geq_{i}} \circ (\id \otimes \scrF_{K^\leq_{i}})  \\
 \Mbar_{ K^\geq_{i}} \times \Mbar_{K^\leq_{i}; \inp} &  \longleftrightarrow \scrF_{K^\geq_{i}} \times \scrC_{K^\leq_{i}}  \\
 \Mbar_{K \setminus \{ i \}; \inp}  &  \longleftrightarrow \scrC_{K \setminus \{ i \}} .
\end{array}
\end{equation}
This correspondence between strata and compositions of maps implies:
\begin{lem} \label{lem:C-is-chain-map}
  $\scrC$ is a chain map. 
\end{lem}
\begin{proof}[Sketch of proof:]
The cocycle $\alpha^v$, which appears in the differential on $\Hom(\scrF(L), \scrF(L')) $ (see Equation \eqref{eq:differential_morphism_sheaves}) arises as in Lemma \ref{lem:Floer_gives_sheaf} because the functions $f_{ij}$, which enter the definition of $\scrC$ and $\scrF$ do not themselves satisfy the cocycle condition.
\end{proof}

\subsection{From \v{C}ech to Floer} \label{sec:from-vcech-floer}

Let $I $ be a totally ordered subset of $\Sigma$. Given $x'_{\min I} \in L'_{\min I} \cap F_{\min I} $, $x_{\max I} \in L_{\max I} \cap F_{\max I} $, and $x_{\out} \in L \cap L'$, consider
  \begin{equation}
      \Mbar_{I; \out}(x_{\min I} , x_{\out}; x'_{\max I}) = \bigcup_{I=\min \vK}  \Mbar_{q,\vK}( x_{\min I}(q_K) , x_{\out}; x'_{\max I})
  \end{equation}
as in Section \ref{sec:moduli-spac-assoc}. Each stratum of the right hand side is a parametrised moduli space over $\Adams_{\vK; \out}$.  If $|\vK| = n -|I| +2$, subsequent elements of $\vK$ differ by one element of the maximal element of $\vK$. Fix the orientation of $ \Adams_{\vK; \out}$ coming from the ordering of these elements. Together with the isomorphisms in Equation \eqref{eq:iso_edge_i_k}, this gives a natural isomorphism
\begin{equation}
    \rho_u \co \delta_{x'_{\max I}} \to  \delta_{x_{\min I}} \otimes  \delta_{x_\out} 
\end{equation}
for every rigid element $u \in \Mbar_{I; \out}(x_{\min I} , x_{\out}; x'_{\max I}) $. Given $\phi \in \Hom( \delta_{x_{\min I}}, \delta_{x'_{\max I}} ) $, the trace of the composition with $\rho_u$ yields a map
\begin{equation}
  \tr(\rho_u \circ \phi  ) \co \bZ \to \delta_{x_\out}.
\end{equation}
Using the fact that Laurent polynomials are dense, one defines a map
\begin{equation}
 \scrP_{u} \co  \scrO_{I} \otimes \Hom( \delta_{x_{\min I}}, \delta_{x'_{\max I}} )[1-|I|]    \to \Lambda \otimes \delta_{x_\out } \end{equation}
which is given by
\begin{equation}
 \label{eq:component_map_to_CF}
z_I^\gamma \otimes \phi \mapsto   T^{- f_{\min K, \max K}} z^{-df_{\min K, \max K} + dg'_{\max K}(x'_{\max K}) - dg_{\min K}(x_{\min K})}_{\max K}  T^{\cE(u) } \otimes \tr(\rho_u \circ \phi)
\end{equation}
whenever the homology class  $[\partial u] \in H_1(F_K, \bZ)$ vanishes and is otherwise $0$.

Using the decomposition of $\Hom_{\scrO_{\min I}}(\scrF_{\min I}(L), \scrF_{\max I}(L')) $ as a direct sum 
\begin{equation}
     \bigoplus_{ \substack{ x_{\min I} \in \phi_{\min I} L \cap F_{\min I} \\ x'_{\max I}\in \phi_{\max I} L \cap F_{\max I} } }  \scrO_{I} \otimes  \Hom( \delta_{x_{\min I}}, \delta_{x'_{\max I}} )[1-|I|] ,
\end{equation}
define the components of a map $\scrP$ to the Floer complex:
\begin{align}
  \scrP_I \co   \Hom_{\scrO_{\min I}}(\scrF_{\min I}(L), \scrF_{\max I}(L'))  & \to  CF^*(L, L') \\
\scrP_I |  \scrO_{I} \otimes \Hom( \delta_{x_{\min I}}, \delta_{x'_{\max I}} )[1-|I|]  & = \bigoplus_{ x_\out } \sum_{u \in \Mbar_{I; \out}(x_{\min I} , x_{\out}; x'_{\max I}) }  \scrP_{u}.
\end{align}

To establish that $\scrP$ is a chain map, it is convenient to use the expression for the differential on morphisms of sheaves given in Equation \eqref{eq:differential_morphism_sheaves}, which yields the equation
\begin{multline}
   \mu^1 \circ \scrP  (z_I^\gamma \otimes \phi)   = \sum_{\substack{J \in B \Sigma \\ \min J = \max I}}\alpha^{v}_{J,I}  \scrP_{I \cup J}  \left(  \scrF_{J}(L') \circ (z_I^\gamma \otimes  \phi) \right) + \\
 \sum_{\substack{J \in B \Sigma \\ \min I = \max J}}(-1)^{|I^\leq_i|(1 - | \phi|)}\alpha^{v}_{I,J}   \scrP_{J \cup I} \left( (z_I^\gamma \otimes\phi) \circ \scrF_{J}(L) \right) \\
\qquad + \sum_{\substack{ I \cup \{ j\} \in B \Sigma \\ \min I < j < \max I}} (-1)^{|I^\leq_j|-1-\deg(x_{\min I})+|\phi|} \scrP(z_I^\gamma \otimes  \phi )
\end{multline}
assuming as in Equation \eqref{eq:component_map_to_CF} that $\phi \in \Hom( \delta_{x_{\min I}}, \delta_{x'_{\max I}} )[1-|I|]   $.

The correspondence between the boundary strata of $\Mbar_{I; \out}(x_{\min I} , x_{\out}; x'_{\max I})  $ given in Equation \eqref{eq:boundary_moduli_out_strip}-\eqref{eq:boundary_moduli_out_skip} and the terms in the chain map equation is as follows:
\begin{equation} \label{eq:compare_boundary_moduli_out}
  \begin{array}{ll}
\Mbar \times \Mbar_{I; \out}  &  \longleftrightarrow \mu^1 \circ \scrP_I( \phi) \\ 
 \Mbar_{ I \cup J; \out} \times \Mbar_{J} &  \longleftrightarrow \scrP_{I \cup J} (\scrF_J(L') \circ \phi )  \\
 \Mbar_{ J} \times \Mbar_{J \cup I; \out} &  \longleftrightarrow  \scrP_{J \cup I}(   \phi \circ  \scrF_J(L))   \\
 \Mbar_{I \cup \{ j \}; \out}  &  \longleftrightarrow \scrP_{I \cup \{ j \}} .
\end{array}
\end{equation}

\begin{lem} \label{lem:P-chain-map}
   $\scrP$ is a chain map.
\end{lem}
\begin{proof}[Sketch of proof:]
To check the coefficients, restrict to  the subset of $\Mbar_{I; \out}(x_{\min I} , x_{\out}; x'_{\max I})  $ consisting of curves with $[\partial u] = \gamma$. The coefficient in $\Lambda$ of each term in $\scrP$ is given by the energy of the corresponding curve with one boundary condition on the fibre $F_I$. We conclude that the first and last lines in Equation \eqref{eq:compare_boundary_moduli_out} also have coefficients given by the energy of corresponding broken curve. The remaining two cases require the argument used in  Lemma \ref{lem:energy_contribution_boundary_strata_differential} and Equation \eqref{eq:energy_deformed_curve}. We consider only the case corresponding to the term $\scrP_{J \cup I}(   \phi \circ  \scrF_J(L'))$, leaving the other to the reader.

Let $u$ be a curve contributing to $\scrF_J(L') $.  The coefficient $\cE(u)$ appearing in Equation \eqref{eq:component_differential_K} is the area of a curve with boundary on $F_{\max J}$. The difference with the area of the corresponding curve with boundary on $F_{\max I}$ is given by
\begin{equation}
  \langle q_J -q_I , [\partial u] \rangle + g_{x'_{\max J}}(q_J) - g_{x_{\min J}}(q_J) + g_{x_{\min J}}(q_I) - g'_{x_{\max J}}(q_I)
\end{equation}
as in Equation \eqref{eq:energy_deformed_curve}. The first term appears when changing coefficients from $z_I$ to $z_J$, i.e. 
\begin{equation}
  z_I^\gamma = T^{ \langle \gamma, q_J - q_I\rangle} z_J^\gamma. 
\end{equation}
The remaining coefficients arise as the sums of the exponents of the coefficients of $\alpha^{v}_{I,J} $ and the coefficients in the definition of $\scrF_J(L')  $, see Equation \eqref{eq:component_differential_K}.
\end{proof}

\subsection{Homotopy from the composition to an isomorphism} Comparing Equations \eqref{eq:expression_map_CF_sheaves_K} and \eqref{eq:component_map_to_CF} implies that the composition 
\begin{equation}
\scrP \circ \scrC \co CF^*(L,L') \to CF^*(L,L')
\end{equation}
is given on $\delta_{x_\inp}$ by counts of elements of $\Ann^{\infty}_{[0]}(x_{\out}, x_{\inp})$, and that the corresponding Novikov coefficient is given by the energy.

By Lemma \ref{lem:annulus_cobordism}, the moduli space $ \Adams^{[0,\infty]}_{[0]}(x_{\out}, x_{\inp}) $ yields a cobordism between $\Ann^{\infty}_{[0]}(x_{\out}, x_{\inp}) $ and $\Cont(x_\out; x_\inp) $.  If $\deg(x_\out) = \deg(x_\inp)$,  the only elements of $ \Cont(x_\out; x_\inp)$ are constant curves, and the corresponding map is the identity.

\begin{prop} \label{prop:homot-from-comp}
  The composition $\scrP \circ \scrC $ is homotopic to $(-1)^{\frac{n(n-1)}{2}} \Id$.
\end{prop}
\begin{proof}
It remains to orient $\Adams^{[0,\infty]}_{[0]}(x_{\out}, x_{\inp}) $ relative to $\delta_{x_{\inp}} $  and $\delta_{x_{\out}}$. Consider a rigid element $u \in \Adams^{[0,\infty]}_{[0]}(x_{\out}, x_{\inp})$, with boundary condition on $F_q$. We have a natural isomorphism
\begin{equation}
\bR \cong    |T_u \Adams^{[0,\infty]}_{[0]}(x_{\out}, x_{\inp})| \cong |\det(\dbar_u)| \otimes |T(0,\infty)| \otimes |T_qQ|
\end{equation}
where $\dbar_u  $ is the linearised Cauchy-Riemann operator at $u$, and $\det(\dbar_u)$ its determinant line. Here, $  |(0,\infty)|$ and $ |T_qQ| $ appear because they respectively correspond to changing the modular parameter and the fibre $F_q$. On the other hand, by gluing the linearised operator to the operators associated to $x_{\inp}$, and degenerating the domain of the linearisation of the Cauchy-Riemann operator at $u$ to the union of two discs meeting at a point, we obtain an isomorphism
\begin{equation}
  |\det(\dbar_u)| \otimes \delta_{x_\inp} \cong   |TF_q| \otimes |TX|^{-1}  \otimes \delta_{x_\out}
\end{equation}
by noting that one of the discs has Lagrangians boundary condition $TF_q$, hence has determinant line naturally isomorphic to $|T F_q|$, and the other has determinant line trivialised by the choices of $\Pin^+$ structures on $L$ and $L'$.   Using the isomorphism $TX \cong TQ \oplus TF_q$ we conclude that every rigid element induces a map $
  \delta_{x_\inp} \cong \delta_{x_\out}$.

By construction, this map agrees with that defined by rigid elements of $\Cont(x_\out; x_\inp) $. However, the orientations at $\Ann^{\infty}_{[0]}(x_{\out}, x_{\inp}) $ differ from the product orientation by $(-1)^{\frac{n(n-1)}{2}} $, as first noticed by Fukaya, Oh, Ohta, and Ono \cite[Proposition 3.9.1]{FOOO-sign} (see also \cite[Lemma 5.5.23]{A-SH}). This accounts for the sign in the statement.
\end{proof}

\appendix

\section{The family Floer functor}
\label{sec:family-floer-functor}

\subsection{The $A_{\infty}$ structure}
\label{sec:a_infty-structure}
Let $L \to X$ be a closed Lagrangian immersion in $X$ in generic position, i.e. so that $L$ meets itself in pairs of transverse double points. Assume that (i) $L$ is \emph{tautologically unobstructed;} i.e there exists a tame almost complex structure for which there are no $J_L$ holomorphic maps $D^2 \to X$ such that the complement of one point on the  boundary lifts to $L$,  (ii) each irreducible component of $L$ is a graded Lagrangian, i.e. there is a fixed lift to $\bR$ of the $S^1$-valued phase on each component and (iii) the pullback of $w \in H^2(X, \bZ_2)$ to each component agrees with the second Stiefel-Whitney class.

\begin{rem}
The union of a finite collection of transverse immersed Lagrangians in generic positions which satisfy these conditions for the same $J_L$ again satisfies these properties. The constructions of this section can be carried out for different almost complex structures associated to each component, but the notation becomes more cumbersome.
\end{rem}

Choose a Hamiltonian $H \co X \to \bR$ whose time-$t$ flow $\phi^t$ lies in the contractible set $\scrH^0$ chosen in Section \ref{sec:unif-small-choic}, and such that  $\phi^1 L$ is transverse to $L$. Pick a family $J \co [0,1] \to \scrJ^0$ of almost complex structures on $X$ parametrised by $t \in [0,1]$ which agree with $\phi^t_* J_L$ at $t=0,1$. Generically, all moduli spaces of strips $\Mbar(x,y)$ are regular for $x,y \in \phi^1 L \cap L$. The orientation lines from Equation \eqref{eq:generator_Floer_x}, with differential from Equation \eqref{eq:differential_Floer_complex} define the Floer complex
\begin{equation}
  CF^*(L) = \bigoplus_{x \in L \cap \phi^1 L}  \Lambda \otimes \delta_x.
\end{equation}

Let $\Rbar^{d+1}$ denote the moduli space of discs with $d+1$ marked points on the boundary, one of which is distinguished as outgoing. Let $\Ubar^{d+1}$ denote the universal punctured curve over this moduli space. For all $d$, fix a consistent family of negative strip-like ends $\epsilon_0$  at the outgoing point, and of positive strip-like ends $\{ \epsilon_i\}_{i=1}^{d} $ at all other points as in \cite[Section (9g)]{seidel-Book}. Pick a consistent family of Floer data
\begin{equation} \label{eq:Floer_data_A_oo}
( \phi^{d+1}, J^{d+1}) \co  \Ubar^{d+1} \to \scrH^0 \times \scrJ^{0}
\end{equation}
such that (i) $J^{d+1}(z) = \phi^{d+1}_{*}(z) J_{L}$ whenever $z$ lies on the boundary of a fibre, and (ii) the pullback of $( \phi^{d+1}, J^{d+1}) $ under the strip-like ends agrees with $(\phi^t,J_t)$.
\begin{rem}
It is only for notational convenience that $\phi^{d+1} $ is defined as a map on $\Ubar^{d+1} $, since only its values on the boundary of each fibre will ever be used. The fact that $\scrH^0$ is contractible implies that there is no obstruction to extending a function to $\scrH^0$.
\end{rem}

Given a sequence $\{ x_{j} \}_{j=0}^{d} $ of intersection points between $L$ and $\phi^1 L$, let $\Mbar_{d+1}(x_0; x_d, \ldots, x_1)$ denote the moduli space of maps from a fibre $\Sigma$ of $ \Ubar^{d+1} $ to $X$ such that
\begin{align}
du(z) \circ j =  J^{d+1}(z) \circ dt && u(z) \in \phi^{d+1}(z) L \textrm{ if } z \in \partial \Sigma,
\end{align}
and $u$ converges at the $k$\th strip-like end to $x_k$. These moduli spaces are regular for generic choices of almost complex structures $J^{d+1}$, and a choice of orientation of $\Rbar^{d+1} $ yields a map
\begin{equation}
  \mu_{u} \co \delta_{x_d} \otimes \cdots \otimes \delta_{x_1} \to \delta_{x_0}
\end{equation}
whenever $u$ is rigid. Define the $A_{\infty}$ structure on $CF^*(L)$ to be given by the operations 
\begin{equation}
  \mu^{d} | \delta_{x_d} \otimes \cdots \otimes \delta_{x_1}  = \sum_{\substack{ x_0 \in L \cap \phi^1 L \\ \deg(x_0) = 2-d + \sum_{j=1}^{d} \deg(x_j)}}(-1)^{\sum_{i=1}^{d} (i+1)  \deg(x_i)  } T^{\cE(u)} \otimes \mu_u,
\end{equation}
where the sign is as in \cite{seidel-Book}, whose conventions we follow.
\subsection{Adams moduli spaces with $d$ marked points}
\newcommand{\xa}{-6*\mw pt}
\newcommand{\xb}{-1*\mw pt}
\newcommand{\xc}{-3*\mw pt}
\newcommand{\yc}{3.5*\mw pt}

\newcommand{\xe}{8*\mw pt}

\begin{figure}
\centering
\begin{tikzpicture}
\StripdrawBig{\xc}{\yc};
\draw[line width=4*\lw] (\xc-2.75*\mw pt, \yc+6/8*\mw pt)--(\xc-2.75*\mw pt, \yc+3/8*\mw pt);
\draw[line width=4*\lw] (\xc-1.5*\mw pt, \yc+6/8*\mw pt)--(\xc-1.5*\mw pt, \yc+3/8*\mw pt);
\draw[line width=4*\lw] (\xc-2.25*\mw pt, \yc+6/8*\mw pt)--(\xc-2.25*\mw pt, \yc+3/8*\mw pt);
\draw[line width=4*\lw] (\xc-2.5*\mw pt, \yc+6/8*\mw pt)--(\xc-2.5*\mw pt, \yc+3/8*\mw pt);
\node at (\xc-5*\mw pt, \yc)  {$\max K$};
\node at (\xc+1*\mw pt, \yc)  {$\min K$};
\node at (\xc-2*\mw pt, \yc - 1.5*\mw pt) { $ \Adams_{K \setminus \{i\}
}^{4} $ };

\StripdrawBig{\xa}{0};
\draw[line width=4*\lw] (\xa-2*\mw pt, 6/8*\mw pt)--(\xa-2*\mw pt, 3/8*\mw pt);
\node at (\xa-5*\mw pt, 0)  {$\max K$};
\node at (\xa+.5*\mw pt, 0)  {$i$};
\node at (\xa+.5*\mw pt, -1.5*\mw pt)  {$ \Adams_{K^\geq_i}^{1} \times  \Adams_{K^\leq_i}^{3} $};
\StripdrawBig{\xb}{0};
\draw[line width=4*\lw] (\xb-2*\mw pt, 6/8*\mw pt)--(\xb-2*\mw pt, 3/8*\mw pt);
\draw[line width=4*\lw] (\xb-1.75*\mw pt, 6/8*\mw pt)--(\xb-1.75*\mw pt, 3/8*\mw pt);
\draw[line width=4*\lw] (\xb-2.5*\mw pt, 6/8*\mw pt)--(\xb-2.5*\mw pt, 3/8*\mw pt);
\node at (\xb+1*\mw pt, 0)  {$\min K$};

\StripdrawBig{\xe}{0};
\draw[line width=4*\lw] (\xe-2.5*\mw pt, 6/8*\mw pt)--(\xe-2.5*\mw pt, 3/8*\mw pt);
\draw[line width=4*\lw] (\xe-1*\mw pt, 6/8*\mw pt)--(\xe-1*\mw pt, 3/8*\mw pt);
\draw[line width=4*\lw] (\xe-3*\mw pt, 6/8*\mw pt)--(\xe-3*\mw pt, 3/8*\mw pt);
\node at (\xe+1*\mw pt, 0)  {$\min K$};
\node at (\xe-5*\mw pt, 0)  {$\max K$};
\node at (\xe-2*\mw pt,  -1.5*\mw pt) { $ \Adams_{K}^{3}  \times \Rbar^{3}$ };
\draw [line width=\lw] (\xe-1*\mw pt, .5*\mw pt) arc (-90:270:1 * \mw pt);
\draw[line width=4*\lw] ($(\xe-1*\mw pt, 1.5*\mw pt)+(30:.875 * \mw pt)$) -- ++(60:0.375*\mw pt);
\draw[line width=4*\lw] ($(\xe-1*\mw pt, 1.5*\mw pt)+(150:.875 * \mw pt)$) -- ++(150:0.375*\mw pt);
\end{tikzpicture}

\caption{Representative boundary strata of $\Adams^{4}_{K}$.}
\label{fig:boundary_adams_with_marked_points}
\end{figure}

Let $K$ be a totally ordered set consisting of more than one element, and $\Ubar_K$ the universal curve over $\Adams_K$ from Section \ref{sec:adams-univ-curve}. Given an integer $d$, let $\Adams_{K}^{d}$ denote the compactified moduli space of fibres of $\Ubar_K$ equipped with $d$ boundary punctures along $\Ubar_K^{\{1\}}$. Let $\Ubar_{K}^{d}$ denote the punctured universal curve over $\Adams_{K}^{d}$.

We extend this definition to the case $K$ is a singleton as follows: if $d > 1$, we define $ \Adams_{\{ k \}}^{d}$ to be a copy of $\Rbar^{d+2}$, in which two successive marked points are distinguished, and $\Ubar_{\{k\}}^{d} $ the corresponding universal curve in which these two marked points have been removed. Each fibre of  $\Ubar_{\{k\}}^{d}   $ over the interior of $ \Adams_{\{ k \}}^{d} $ can be represented as a strip $B = \bR \times [0,1]$ equipped with marked points on its boundary $\bR \times \{ 1 \}$, and this representation is unique up to translation. 

\begin{rem}
We treat the case of a singleton separately because a strip with no boundary marked points is unstable, which led us to define $\Adams_K $ to be a point, while the universal curve over it was empty.
\end{rem}

The boundary strata of the universal curve, shown in Figure \ref{fig:boundary_adams_with_marked_points},  are as follows
\begin{align} \label{eq:boundary_Adams_d_K_0}
 \Ubar_{K \setminus \{ i\}}^{d}  \to & \Adams_{K \setminus \{ i\}}^{d} && i \in K  \\\label{eq:boundary_Adams_d_K_1}
 \Ubar_{K^\geq_i}^{d_2} \times  \Adams_{K^\leq_i}^{d_1} \cup \Adams_{K^\geq_i}^{d_2} \times  \Ubar_{K^\leq_i}^{d_1}  \to &  \Adams_{K^\geq_i}^{d_2} \times  \Adams_{K^\leq_i}^{d_1}  && i \in K, \, 0 \leq d_i,   \,  d_1 + d_2 = d  \\  \label{eq:boundary_Adams_d_K_2}
\Ubar_{K}^{d_2} \times  \Rbar^{d_1+1} \cup \Adams_{K}^{d_2} \times  \Ubar^{d_1+1} \to & \Adams_{K}^{d_2} \times  \Rbar^{d_1+1} &&  1 \leq j \leq d_2,  \,  2 \leq d_1 , \, d_1 + d_2 = d+1 .
\end{align}
 Fix consistent families of positive strip-like ends $\{ \epsilon_{j} \}_{j=1}^{d} $ at all punctures, and for all fibres of $ \Ubar_{K}^{d}$; the consistency requirement is an inductive choice on $d$ and $|K|$, with the base case being the choices of strip-like ends on fibres of $  \Ubar^{d+1}$ over $\Rbar^{d+1} $. As in Equation \eqref{eq:Floer_data_element_Sigma}, denote by $\Phi_k$ the triple of maps
 $(\phi_k, J_k, \Id)$ from the interval to $\scrH \times \scrJ \times \scrD$. We also write $\Phi$ for  the corresponding map $(\phi, J, \Id) $ chosen in Section \ref{sec:a_infty-structure}.

\begin{defin} \label{def:compatibly_fam_continuation_d}
  A \emph{consistent family of continuation data} parametrised by $  \Ubar^{d}_{K} $ is a map
  \begin{equation}
    \Phi^{d}_K = (\phi^{d}_K, J^{d}_K, \psi^{d}_K) \co    \Ubar^d_{K} \to \scrH \times \scrJ \times \scrD
  \end{equation}
such that (i) the pullback under every end labelled by $k \in K$ is given by $\Phi_k$, (ii) the pullback under the end $\epsilon_j$ is given by $\Phi$  (iii) the maps $\phi^d_K$ and $\psi^d_K$  are obtained by gluing, and $J^d_K$  by perturbed gluing, (iv) $ \psi^{d}_K(z)$ preserves the image of $L$ under $\phi_K^d(z)$, and  (v) for each $ z  \in \Ubar^{d,\{1\}}_{K} $, we have:
\begin{equation}
J^d_K(z) = \left( \psi^d_K(z) \circ \phi^d_K(z) \right)_* J_L.
\end{equation}
\end{defin}

Assuming that $L_k \equiv \phi_k L$ is transverse to $F_q$ for all $k \in K$, we obtain a holomorphic curve equation $ \partial_{s} u(z) = J^d_K(z) \partial_{t}u(z)$ on the space of maps from fibres of $ \Ubar^{d}_{K} $ to $X$, with boundary conditions
\begin{align}  \label{eq:boundary_conditions}
u(z)  \in F_{q}\textrm{ if } z \in \Ubar_{K}^{\{0\}}  && \qquad u(z)  \in  \phi^d_K(z) L \textrm{ if } z \in  \Ubar_{K}^{d,\{1\}}.
\end{align}

\subsection{Choices of continuation data with multiple inputs}

Fix the choices made in  Section \ref{sec:unif-small-choic}, i.e. nested sequences $\{ \scrJ^{i} \}_{i=1}^{n+3}$ of subsets of the space of tame almost complex structures and $\{ \scrD^{i} \}_{i=1}^{2^{n+2}}$ of the space of diffeomorphisms, the simplicial triangulation $\Sigma$ of $\Q$, the associated cover $P_i$ by polyhedra with basepoints $q_i$, and  maps $\phi_i \in \scrH^0$ mapping $L$ to a Lagrangian transverse to $F_{i}$. These choices should be made so that $\scrD^i(\scrL)$ is an acyclic fibration over $\scrL$, and both $\scrD^i_{P_i^2}$ and $\scrD^i_{P_i^2}(L_i)$ are acyclic fibrations over $P_i^2$. A section $\psi_i$ of $\scrD^i_{q_i,P_i}(L_i) $ is fixed.

Assume that the maps $(\Phi_I, \Psi_I)$ from Section \ref{sec:continuation-maps-at-vertex} are chosen, and define
\begin{equation}
  \Phi^{d+1} = ( \phi^{d+1}, J^{d+1}, \Id) \co \Ubar^{d+1} \to \scrH^0 \times \scrJ^{0} \times \scrD.
\end{equation}
For each ordered subset of $I$, pick continuation data and families of diffeomorphisms
\begin{align} \label{eq:diffeo_pair_cont_d}
  \Phi^d_I  = (\phi^d_I, J^d_I, \psi^d_I) \co \Ubar^d_{I} & \to \scrH^0 \times \scrJ^{|I|} \times \scrD^{2^{|I|-1}} \\
\Psi^{d}_I \co  P_{I} \times \Ubar^{d}_{I} & \to \scrD^{2^{|I|-1}}
 \end{align}
where $\Psi^{d}_{I} $ is obtained by gluing near every boundary stratum, and is the identity on $\{q_I\} \times \Ubar^{d}_{I}  $. In addition, we require that the restrictions of $(\Phi^d_I, \Psi^{d}_{I})$ agree with
\begingroup
\allowdisplaybreaks
  \begin{align} \label{eq:first_equation_family_continuation_a-oo}
& ( {\psi_{\min I}(q_I)}_*\Phi_{\min I}, \psi_{\min I} \circ \psi ^{-1}_{\min I}(q_I)) && \textrm{ along the end }  \epsilon_+ \\ 
& (\Phi_{\max I}, \psi_{\max I})  & &\textrm{ along the end }  \epsilon_- \\
&(\Phi, \Id)  & &\textrm{ along each end }  \epsilon_j \\
& (\Phi^d_{I \setminus i} ,\Psi^{d}_{I  \setminus i} ) &&\textrm{ on } \Ubar^d_{I \setminus i}  \\ 
& ( \Phi^{d_2}_{I^\geq_{i}} , \Psi^{d_2}_{I^\geq_{i} } ) &&\textrm{ on }  \Ubar^{d_2}_{I^\geq_{i}} \times \Adams^{d_1}_{I^\leq_{i}}  \\ \label{eq:second_equation_family_continuation_a-oo}
& \left({\Psi^{d_1}_{I^\leq_{i}}}(q_I)_{*}  \Phi^{d_1}_{I^\leq_{i}}  , \Psi^{d_1}_{I^\leq_{i}} \circ (\Psi^{d_1}_{I^\leq_{i} } (q_I))^{-1}   \right)  &&\textrm{ on } \Adams^{d_2}_{I^\geq_{i}} \times   \Ubar^{d_1}_{I^\leq_{i}} \\
& (\Phi^{d_2}_{I} ,\Psi^{d_2}_{I} ) &&\textrm{ on } \Ubar^{d_2}_{I}  \times \Rbar^{d_1} \\ 
& (\Phi^{d_1} ,\Id ) &&\textrm{ on } \Adams^{d_2}_{I}  \times \Ubar^{d_1} ,
  \end{align}
where Equation \eqref{eq:first_equation_family_continuation_a-oo} and \eqref{eq:second_equation_family_continuation_a-oo} are interpreted as in Remark \ref{rem:interpret_funny_composition}.
Moreover, on the boundary of each fibre:
  \begin{align} \label{eq:condition_diffeo_I_to_K_along_1_d}
\Psi^{d}_{I}\left( q, z \right) & \in  \scrD^{2^{|I|-1}}(\phi_{I,z}(L)) \textrm{ if } z \in \Ubar^{d,\{1\}}_{I} \\ \label{eq:condition_diffeo_I_to_K_along_0_d}
\Psi^{d}_{I}\left(q , z  \right) & \in  \scrD^{2^{|I|-1}}_{q_{I}, q} \textrm{ if } z \in \Ubar^{d,\{0\}}_{I} .
  \end{align}
Such data can be constructed by a double induction: assuming that they have been chosen for all pairs $(d_1,I)$ whenever $d_1 < d$, construct the data for $d$ by induction on the number of elements of $I$ as in Lemma \ref{lem:construction_by_induction}. 

The pushforward of $\Phi^d_K$ by $ \Psi^{d}_K$ yields compatible families parametrised by $P_K$:
\begin{align}
\Phi^{d,P}_K  \co P_K \times \Ubar^d_{K} & \to \scrH^{0}  \times \scrJ^{|K|+1}  \times \scrD^{2^{|K|}}.
\end{align}
\endgroup
Given $x \in  F_{\min K} \cap L_{\min K} $ and $y \in  F_{\max K} \cap L_{\max K}$, let $x(q)$ and $y(q)$ be as in Equation \eqref{eq:push_intersect_diffeo} whenever $q \in P_K$. If  $x_j \in L \cap \phi^1 L$  for $1 \leq j \leq d$, the compactified moduli space of solutions to the holomorphic curve equation determined by $\Phi^{d,P}_K(q) $, with boundary conditions as in Equation \eqref{eq:boundary_conditions}, and asymptotic conditions $x(q)$ at the positive end of the strip, $y(q)$ at the negative end, and $x_j$ along the $j$\th strip like end will be denoted:
\begin{equation}
  \Mbar^{d}_{q, K}( y(q); x_d, \ldots, x_1, x(q)).
\end{equation}
Composition with $\Psi^{d}_K(q)$ yields a homeomorphism
\begin{equation}
    \Mbar^{d}_{q_{K},K}( y(q_{K}); x_d, \ldots, x_1, x(q_{K})) \to  \Mbar^{d}_{q,K}( y(q); x_d, \ldots, x_1, x(q))    ,
\end{equation}
so we omit $q$ and $q_{K} $ from the notation.

Using the description of the boundary of $ \Adams_{K}^{d}  $ given in Equations \eqref{eq:boundary_Adams_d_K_0}-\eqref{eq:boundary_Adams_d_K_2}, the boundary of $ \Mbar^{d}_{K}( y; x_d, \ldots, x_1, x) $ decomposes as follows:
\begin{align} \label{eq:boundary_moduli_K_d-0}
 & \Mbar_{K \setminus \{ i\}}^{d}  ( y; x_d, \ldots, x_1, x)  \\ \label{eq:boundary_moduli_K_d-1}
& \coprod_{y' \in F_i \cap L_i } \Mbar_{K^\geq_i}(y,y') \times   \Mbar_{K^\leq_i}^{d}  ( y'; x_d, \ldots, x_1, x)  \\ \label{eq:boundary_moduli_K_d-2}
& \coprod_{x' \in F_i \cap L_i } \Mbar_{K^\geq_i}^{d}  ( y; x_d, \ldots, x_1, x')  \times   \Mbar_{K^\leq_i}(x',x) \\ \label{eq:boundary_moduli_K_d-3}
& \coprod_{ \substack{ 1 \leq d_1, d2 \\ d_1 + d_2 = d} } \coprod_{x' \in F_i \cap L_i }  \Mbar_{K^\geq_i}^{d_2}(y; x_d, \ldots, x_{d_1+1}, x') \times   \Mbar_{K^\leq_i}^{d_1} ( x'; x_{d_1}, \ldots, x_1, x)   \\  \label{eq:boundary_moduli_K_d-4}
& \Mbar_{K}^{d_2} ( y; x_d, \ldots , x_{d_1+j+1} , x_0, x_j, \ldots, x_1, x  ) \times  \Rbar^{d_1+1} (x_0; x_{d_1+j}, \ldots, x_{j+1}).
\end{align}
In the above, the strata corresponding to Equation \eqref{eq:boundary_Adams_d_K_1} for which $d_1$ or $d_2$ vanish are listed separately, while  the breaking of strips at the ends is incorporated into Equations \eqref{eq:boundary_moduli_K_d-1}, \eqref{eq:boundary_moduli_K_d-2}, or \eqref{eq:boundary_moduli_K_d-4} depending on whether the breaking takes place at $\epsilon_-$, $\epsilon_+$, or one of the ends $\epsilon_j$, for $1 \leq j \leq d$.

\subsection{The $A_\infty$ functor}

Assume now that the Floer data are chosen generically so that the moduli spaces in the previous section are manifolds of the expected dimension. We begin by noting that we have a natural identification of the interior of $ \Adams^d_K $ with
\begin{equation}
\cR^{d+2} \times [0,\infty)^{|K|}.
\end{equation}
Fixing the orientation on $\cR^{d+2}$ used in \cite{seidel-Book}, and the natural orientation on $[0,\infty) $  yields an orientation of the moduli space $\Adams^d_K$. We then obtain a map
\begin{equation}
  \kappa_{u} \co \delta_{x_d} \otimes \cdots \otimes \delta_{x_1} \otimes \delta_{x} \to \delta_{y}
\end{equation}
associated to every rigid element $u \in \Mbar^{d}_{K}( y; x_d, \ldots, x_1, x) $. Define
\begin{equation}
  \mu_u = (-1)^{\deg(x) + \sum_{i=1}^{d} (i+1)  \deg(x_i)  } \kappa_u,
\end{equation}
where the sign that we introduce is the same as that of \cite[Equation (12.24)]{seidel-Book}, and consider the map 
\begin{equation}
    \scrC_K^{d}  \co CF^*(L)^{\otimes d} \otimes  \scrF(L,\min K)  \to \scrF(L,\max K)[2-|K|-d]
\end{equation}
whose restriction to $ \delta_{x_d} \otimes \cdots \otimes \delta_{x_1}  \otimes  \delta_{x}  $ is given by
\begin{equation}
   \label{eq:expression_map_CF_sheaves}
\bigoplus_{y \in L_K \cap F_{K}} T^{ f_{\min K, \max K}} z^{df_{\min K, \max K} - dg'_{\max K} + dg_{\min K}}_{\max K}  \sum_{u \in   \Mbar^{d}_{K}( y; x_d, \ldots, x_1, x)  } T^{\cE(u) } z^{[\partial u]}_{\max K} \otimes \mu_u.
\end{equation}
The direct sum of these maps over all $K$ will be denoted
\begin{equation}
\scrC^{d} \co  CF^*(L)^{\otimes d}  \to \Hom(\scrF(L),\scrF(L)).
\end{equation}

Recall that an $A_\infty$ homomorphism from an $A_\infty$ algebra to a differential graded algebra consists of such maps which satisfy the equation:
\begin{multline}
  \mu^1 \left( \scrC^d(a_d, \ldots, a_1) \right) + \sum_{d_1 + d_2 = d} \mu^2 \left( \scrC^{d_2}(a_d, \ldots, a_{d_1+1}), \scrC^{d_1}(a_{d_1}, \ldots, a_1) \right) = \\
\sum_{\substack{ d_1, d_2, j \\ d_1 + d_2 = d+1}} (-1)^{\sum_{i=1}^{j} |a_j| -j} \scrC^{d_2}( a_d, \ldots, a_{d_1+j+1}, \mu^{d_1}(a_{d_1+j}, \ldots, a_{j+1}), a_j, \ldots, a_1).
\end{multline}
The terms on the right hand side correspond to the boundary strata in Equation \eqref{eq:boundary_moduli_K_d-4}, those in the second term of the left to Equation \eqref{eq:boundary_moduli_K_d-3}, and the first term to Equations \eqref{eq:boundary_moduli_K_d-0}-\eqref{eq:boundary_moduli_K_d-2}. To see the last part, use the definition of the differential in Equation \eqref{eq:differential_morphism_sheaves}. Our choice of split orientations on $ \Adams^d_K $ allows us to directly appeal to the sign considerations of \cite[Section 12]{seidel-Book}:
\begin{prop}
The maps $ \scrC^{d} $ are the components of an $A_\infty$ homomorphism from $ CF^*(L)$ to the endomorphism algebra of $\scrC(L)$ as an $(\alpha^v)^{-1}$-twisted sheaf. \qed
\end{prop}

\end{document}